\documentclass[11pt]{article}
\usepackage{amssymb,amsmath,amsthm}
\usepackage{upgreek}
\usepackage{yfonts}
\usepackage{mathrsfs}
\usepackage{stmaryrd}
\newtheorem{remark}{Remark}

\newtheorem{proposition}{Proposition}
\newtheorem{lemma}{Lemma}
\newtheorem{corollary}{Corollary}
\newtheorem{theorem}{Theorem}

\def\rest{\hskip 1pt{\hbox to 10.8pt{\hfill
\vrule height 7pt width 0.4pt depth 0pt\hbox{\vrule height 0.4pt width 7.6pt depth 0pt}\hfill}}}

\newcommand{\eps}{\varepsilon}
\def\beq{\begin{equation}}
\def\eeq{\end{equation}}

\def \rd {{\rm d}}
\def\B{{\mathbb B}}
\def\E{{\rm E}}
\def\Eps {{\rm E}_\varepsilon}
\def\C{{\mathbb C}}
\def\D{{\mathbb D}}

\def\R{{\mathbb R}}
\def\N{{\mathbb N}}
\def\S{{\mathbb S}}
\def \Cnrg{{\rm C}_{\rm nrg}}

\def \be{\vec {\bf e}}
\def \beprime{\vec {\bf e'}}

\def \upsigmam{{\upsigma}_{\rm main}}
\def \upsigman{{\upsigma}_{\rm bis}}

\def \Cdec {{\rm C}_{\rm dec}}
\DeclareMathAlphabet{\mathpzc}{OT1}{pzc}{m}{it}

\def\QED{\hbox{${\vcenter{\vbox{
   \hrule height 0.4pt\hbox{\vrule width 0.4pt height 6pt
   \kern5pt\vrule width 0.4pt}\hrule height 0.4pt}}}$}\vspace{7pt}}

\begin{document}

\author{Fabrice BETHUEL\thanks{UPMC-Paris6, UMR 7598 LJLL, Paris, F-75005 France, 
} }
\title{Concentration sets for  multiple equal-depth wells potentials in the 2D elliptic case}
\date{}
\maketitle

\begin{abstract}
The formation of codimension-one interfaces for multiwell gradient-driven problems  is well-known and established in the scalar case,   where the equation  is often referred 
to as the Allen-Cahn equation. The vectorial case  in contrast is quite open.  This lack of results and insight  is to a large extend  related  to the absence  of known appropriate  monotonicity formula. In this paper, we focus on the  elliptic case in two dimensions,  and introduce some methods which allow    to circumvent  the lack of monotonicity formula. This methods lead, as expected,   to  concentration  on one-dimensional rectifiable sets. 
\end{abstract}

\bigskip
\noindent
\section{Introduction}
\subsection{Statement of the  main results}
 Let $\Omega$ be a   smooth bouned  domain in $\R^2$. In the present paper we investigate  asymptotic properties of  families of solutions $(u_\eps)_\eps>0$ of the systems of  equations having the general form  
 \begin{equation}
 \label{elipes}
  -\Delta u_\eps=-\eps^{-2}\nabla_u V(u_\eps)  {\rm \ in \ } \Omega \subset \R^2,  
  \end{equation}
as the parameter $\eps >0$ tends to zero.  
 The function $V$, usually termed the \emph{potential},  denotes   a smooth scalar   function on   $\R^k$, where $k \in \N$ is a given integer.  Given $\eps>0$, the function  $v_\eps$  denotes   a function defined on  the domain $\Omega$ with values into the euclidian space $\R^k$, so that equation \eqref{elipes} is  a system of $k$ scalar partial differential equations  for each of the components of the map $v_\eps$. 
 
    Equation \eqref{elipes}   corresponds to the  Euler-Lagrange equation   of the  energy functional  $\mathcal{E_\eps}$ which is defined for a function $ u:\Omega \mapsto \R^k$ by the formula
    \begin{equation}
 \label{glfunctional}
 \E_\eps(u)= \int_\Omega e_\eps(u)=\int_{\Omega} \eps \frac{\vert \nabla u \vert ^2}{2}+\frac{1}{\eps} V(u).
 \end{equation}
   We  assume that the potential $V$ is bounded below, so that we may impose, without loss of generality and changing $V$ by a suitable  constant,  that
   \begin{equation}
   \label{infimitude}
   \inf V=0.
   \end{equation}
 We introduce the  set $\Sigma$ of minimizers of $V$,  sometimes called the vacuum manifold, that is the  subset of $\R^k$ defined 
     $$\Sigma\equiv \{ y \in \R^k, V(y)=0 \}.$$
   Properties of solutions to \eqref{elipes} crucially depend on the nature of $\Sigma$. In this paper, we  will assume that the 
    vacuum manifold is finite, with at least two distinct elements, so that   \\
    
    \noindent
   $\displaystyle{ (\text{H}_1) \ \ \ \ \  
 \Sigma=\{\upsigma_1, ..., \upsigma_q\},\ q\geq 2, \ \upsigma_i \in \R^k, \ \forall i=1,...,q.
 }$ \\
 
 \noindent
  We impose furthermore a condition on the behavior of $V$ near its zeroes, namely:
    
    \smallskip
    \noindent
${(\text{H}_2)}$   {\it  The matrix $\nabla^2V(\upsigma_i)$ is positive definite at each point $\upsigma_i$ of $\Sigma$, in other words, if $\lambda_i^-$ denotes its smallest eigenvalue, then  $\lambda_i^->0$. We denote by $\lambda_i^+$ its largest eigenvalue. }\\

\noindent
Finally, we also impose a growth conditions at infinity: 

\smallskip
\noindent
  ${(\text{H}_3)}$ {\it There exists  constants  $\upalpha_\infty >0$ and  $R_\infty >0$  such that  
  \begin{equation}
  \label{condinfty}
  \left\{
   \begin{aligned}
   y\cdot\nabla V( y )&\geq \upalpha_\infty \vert y \vert ^2, \ \hbox {if }  \vert y \vert >R_\infty {\rm \  and \ }  \\
 V(x) \to &+\infty {\rm \ as  \  }  \vert x \vert \to + \infty. 
 \end{aligned}
 \right.
 \end{equation}   
 }
 \smallskip
\noindent
A potential $V$ which fulfills conditions conditions ${(\text{H}_1)}$, ${(\text{H}_2)}$ and ${(\text{H}_3)}$   is  termed  throughout  the paper  a potential with multiple  equal depth wells. 
\smallskip

 A  typical  example is provided  in the scalar case $k=1$ 
 by the potential, often termed \emph{Allen-Cahn} or \emph{Ginzburg-Landau} potential,
 \begin{equation}
 \label{glexemple}
  V(u)=\frac{(1-u^2)^2}{4}, 
  \end{equation}  
  whose  infimum equals $0$ and whose minimizers are $+1$ and $-1$, 
  so that   $\Sigma=\{+1,-1\}.$  It is used as  an elementary model for \emph{phase transitions} for  materials with  two equally  preferred states,   the   minimizers $+1$ and $-1$ of the potential $V$.  
  
  \smallskip

Important efforts have been devoted so far to the study  of  solutions of the stationary \emph{Allen-Cahn} equations, i.e. solutions to \eqref{elipes} for the special choice  of potential \eqref{glexemple},   or to  the corresponding parabolic evolution equations,  in the asymptotic limit $\eps \to 0$, in arbitrary dimension $N$ of the domain $\Omega$ . The  mathematical  theory for this question is now well advanced and may be considered as satisfactory. The results found there    provides a sound mathematical foundation  to the intuitive idea that       the domain $\Omega$ decomposes  into regions  where the solution takes  values  either close to  $+1$ or  close to  $-1$, the   regions being separated by interfaces  of width of order $\eps$. These interfaces,   termed  \emph{fronts},  are expected to converge  to hypersurfaces of codimension 1. These hypersurfaces are shown to be   \emph{generalized minimal surfaces}  in the stationary case, or \emph{moved by mean curvature} for the parabolic evolution equations. 
  Several of the arguments rely on   \emph{integral methods} and \emph{energy estimates}. For instance in \cite{ilmanen},  T.Ilmanen proved  convergence  \emph {for all time}, in particular past possible singularities of the flow, to   motion  by mean curvature  in the \emph{weak sense of Brakke},   a notion  relying on  the language,  concepts  and methods of\emph{ geometric measure theory}. In the elliptic case considered in this paper, convergence to minimal surfaces was established by Modica and Mortola in their celebrated paper \cite{mortadela},  F. Hutchinson and Y. Tonegawa in \cite{hutchtone} established related results for non-minimizing solutions in \cite{hutchtone}. In \cite{ilmanen, hutchtone} and related works, the fact that the solutions are scalar are used in several arguments, in first place for the proof of a suitable monotonicity formula yielding concentration on $N-1$ dimensional set. In the present context, setting for an arbitrary subdomain $G \in \Omega$, 
  \begin{equation}
  \label{energyu}
  \Eps\left(u_\eps, G\right)=\int_{\mathcal U} e_\eps(u)  {\rm d} x, 
  \end{equation}
   we recall that the monotonicity formula
   \begin{equation*}
  \label{monotoniee}
  \frac{d}{dr} \left(\frac{1}{r^{N-2} }  \Eps\left(u_\eps, \B^N(x_0, r)\right) \right) \geq 0,  {\rm \ for \ any \ } x_0 \in \Omega,  
  \end{equation*} 
   holds for arbitrary potentials, and is relevant if one wants to establish  concentration on $N-2$ dimensional sets, as it occurs in Ginzburg-Landau theory. If one wants instead to establish concentration on $N-1$ dimensional sets, then the stronger monotonicity formula 
    \begin{equation}
  \label{monotonie}
  \frac{d}{dr} \left(\frac{1}{r^{N-1} }  \Eps\left(u_\eps, \B^N(x_0, r)\right) \right) \geq 0,  {\rm \ for \ any \ } x_0 \in \Omega,  
  \end{equation} 
  is more appropriate:   The proof of formula \eqref{monotonie} in the \emph{scalar case} relies the positivity of the  \emph{discrepancy}  function
 \begin{equation}
 \label{discretpanse}
 \xi_\eps (u_\eps)= \frac{1}{\eps} V(u_\eps)-\eps \frac{\vert \nabla  u \vert^2}{2}, 
 \end{equation}
 a property established as mentioned thanks to the maximum principle. Notice that in the one dimensional case,  that is for the equation 
 $\displaystyle{-\eps^2 \ddot u= -\nabla_u V(u)}$ on some interval  $I$,
 one has the conservation law 
 $$
 \frac{d}{dx} \left(\frac{1}{\eps} V(u)-\eps\frac{\vert \dot u\vert^2}{2}\right)=0, 
 $$
 so that the discrepancy corresponds to a Lagrangian, and it is therefore constant on any interval. In higher dimensions, 
the fact that $\xi_\eps$ is positive for \emph{scalar solutions} of \eqref{elipes} was observed first by L. Modica in \cite{modica} for entire solutions. On the other hand, concerning the vectorial case,  positivity of the discrepancy as well
 as the monotonicity formula are known to fail for some solutions of the \emph{Ginzburg-Landau system}, so that the question whether they might still hold  under some possible additional conditions on the potential  or the solution itself is widely open to our knowledge (see \cite{Ali1} for a discussion  of these issues and for additional references).
 
 \begin{remark}
 \label{mini}
 {\rm  The case of \emph{minimizing solutions} was treated  by Modica and Mortola   in \cite{mortadela} for  the  Allen-Cahn potential. In  \cite{baldo, fontar}, Baldo and Fonseca and Tartar treated  the vectorial case, for which he   obtained quite similar results. The approaches rely on ideas from Gamma convergence, and du not rely on monotonicity formulas as for general stationary solutions or solutions of the corresponding evolution equations.
 }
 \end{remark}
 
\medskip     
  The purpose of  the present  paper is to show that,  to  a  large extend,   the results obtained in the scalar case, can be transposed to the vectorial case for  potentials $V$ which fulfill  conditions ${(\text{H}_1)}$, ${(\text{H}_2)}$ and ${(\text{H}_3)}$, that is potentials with multiple  equal depth wells, if we restrict ourselves to \emph{two dimensional domains}. Since  no monotonicity  formula in this case is know, new arguments have to be worked out. Several of them  rely strongly on some specificities of dimension two.
    
  \smallskip
  We assume that we are given a constant ${\rm M}_0>0$ and a family $(u_\eps)_{0<\eps\leq 1}$ of solutions to the equation \eqref{elipes} for the corresponding value of the parameter $\eps$, satisfying the natural energy bound 
   \begin{equation}
  \label{naturalbound}
  \Eps(u_\eps) \leq{\rm  M}_0,   \    \forall \eps >0.
      \end{equation}
      Assumption \eqref{naturalbound}  is rather  standard in the field, since it corresponds to the energy magnitude required for the creation of $(N-1)$-dimensional interfaces.  
  We introduce the family $(\upnu_\eps)_{0<\eps\leq 1} $ of measures defined on $\Omega$  by
  \begin{equation}
  \label{mesure}
  \upnu_\eps \equiv e_\eps(u_\eps) \, {\rm d}\, x    {\rm  \ on  \ } \Omega.
    \end{equation}
   In view of \eqref{naturalbound}, the total mass of the measures is bounded by ${\rm M}_0$, that is 
   $\upnu_\eps (\Omega) \leq {\rm M}_0.$
 By compactness, there exists   therefore a decreasing subsequence $(\eps_n)_{n\in \N}$  tending to $0$ and a limiting measure $\upnu_\star$    on $\Omega$ with $\upnu_\star (\Omega) \leq {\rm M}_0$, such that 
  \begin{equation}
  \label{choux}
  \upnu_{\eps_n}\rightharpoonup \upnu_\star  {\rm \ in \ the \ sense \ of \ measures  \ on \ }
  \Omega  {\rm \ as \ } n\to +\infty. 
  \end{equation}
  Our main result  is the following.
  
  \begin{theorem}
 \label{maintheo}  Let $(u_{\eps_n})_{n \in \N}$  be  a sequence of solutions to \eqref{elipes}  satisfying \eqref{naturalbound} and \eqref{choux}. 
 There exist a subset $\mathfrak S_\star$of  $\Omega$ and a subsequence of $(\eps_n)_{n \in \N} $, still denoted $(\eps_n)_{n \in \N}$ for sake of simplicity,
such that the following properties hold: 
\begin{enumerate}
\item[i)] $\mathfrak S_\star$  is  a closed  1 dimensional \emph{rectifiable}  subset of $\Omega$ 
 such that 
 \begin{equation}
 \label{herbert}
 \mathcal H^1(\mathfrak S_\star)\leq \rm C_{\rm H}\, {\rm M}_0,
 \end{equation}
  where ${\rm C}_{\rm H}$ is a constant depending only on the potential $V$.
\item[ii)]   Set $\mathfrak U_\star=\Omega\setminus  \mathfrak S_\star$, and   let $(\mathfrak U_\star^i)_{i \in I}$ be  the connected components of $\mathfrak U_\star$.   For each $i\in I$ there exists an element $\upsigma_i \in \Sigma$ such that 
$$
u_ {\eps_n}  \to \upsigma_i  {\rm \ uniformly  \ on  \ every  \ compact \ subset \  of \ } \mathfrak U_\star  {\rm \ as \ } n \to +\infty.
$$
\end{enumerate}
  \end{theorem}
  
 Similar to the results obtained for \emph{the scalar case},  Theorem 1 expresses, \emph{for the vectorial case in dimension two},  the fact that the domain can be decomposed into subdomains,  where, for $n$ large,  the maps $u_{\eps_n}$ takes values close to an element of the vacuum set $\Sigma$. This subdomains  which are separated by a one dimensional subdomain, on which the map  $u_{\eps_n}$ might possibly undergo a transition from one element of $\Sigma$ to  another.  Our result extends also to \emph{non-minimizing} solutions the results\footnote{This result  hold however in arbitrary dimension and yield stronger properties for $\mathfrak S_\star$.} of \cite{baldo,fontar} (see Remark \ref{mini}).
 
 \smallskip
  An important property of the set $\mathfrak S_\star$ stated in Theorem \ref{maintheo} is its rectifiability. Recall that  a Borel set $\mathcal S\subset \R^2$
is rectifiable of dimension 1  if its one-dimensional Hausdorff dimension is locally finite, and if there there is a countable family of $C^1$ one dimensional submanifolds of $\R^2$ which cover $\mathcal H^1$ almost all of $\mathcal S$. Rectifiability of $\mathcal S$
  implies in particular, that  the set $\mathcal  S$ has \emph{an approximate tangent line} at $\mathcal H^1$-almost every point $x_0 \in \mathcal S$.   This means that there exists  a unit vector  $\vec e_{x_0}$ (depending on  the point $x_0$)  such that,  for \emph{any  number}  $\uptheta >0$ we have
 \begin{equation}
 \label{tangent}
 {\underset {r\to 0}\lim} 
 \frac{ 
 \mathcal H^1 \left( \mathcal S \cap\left( \D^2\left(x_0, r\right) \setminus   \mathcal C_{\rm one}\left(x_0, \vec e_{x_0}, \uptheta \right) \right) \right)
 }
{r}=0, 
 \end{equation}
 where, for a unit vector $\vec e$ and $\uptheta>0$, the set  $\mathcal C_{\rm one}\left(x_0, \vec e, \uptheta  \right)$ is the cone given by
 \begin{equation}
\label{conalpha}
\mathcal C_{\rm one}\left(x_0, \vec e, \uptheta  \right)=
\left \{ y \in \R^2,   \vert  \vec e^\perp \cdot  (y-x_0)  \vert \leq \uptheta \vert  \vec e \cdot  (y-x_0) \vert\right\},
 \end{equation}
 $\vec e^\perp$ being a unit vector  orthonormal to $\vec e$. A point $x_0$  such that \eqref{tangent} holds for some unit vector $\vec e_{x_0}$  is termed a \emph{regular point} of $\mathcal S$.   For  the  set  $\mathfrak S_\star$ given by Theorem \ref{maintheo}, property \eqref{tangent}  can be strengthened as follows:
 
 \begin{proposition}
 \label{tangentfort} 
 Let $x_0$ be a regular point of $\mathfrak S_\star$. Given any $\uptheta>0$ there exists  a radius $R_{\rm cone}(\uptheta, x_0)$ such that 
 \begin{equation}
 \label{radius}
\mathfrak S_\star \cap  \D^2\left(x_0, r\right) \subset  \mathcal C_{\rm one}\left(x_0, \vec e_{x_0}, \uptheta  \right), 
 {\rm  \  for \   any \  }  0<r\leq R_{\rm cone}(\uptheta, x_0).
 \end{equation}
 \end{proposition} 
 
 \medskip
 Compared to the scalar case, the picture is obviously still incomplete. In particular, one would like to obtain further properties of the set $\mathfrak S_\star$. Indeed, in the scalar case, it is know that this set is a stationary varifold,  a weak notion of minimal surfaces, so that one might conjecture that a similar property holds for the vectorial case. This   results remains still an  important challenge\footnote{This is however established in \cite{baldo, fontar} for minimizing solutions}.

  \smallskip
  The set $\mathfrak S_\star$  in the above theorem is obtained as a concentration set of the energy.  The properties stated in Theorem \ref{maintheo} are, for a large part,  consequences of the two results  we present next. The first one represents  a classical form of a  clearing-out result for the measure $\upnu_\star$ and leads directly to the fact that energy concentrates on sets which are at most one-dimensional.  
  \begin{theorem} 
\label{claire}
Let $x_0 \in \Omega$ and $r>0$ be given such that $\D^2(x_0, r) \subset \Omega$.  There exists a constant $\upeta_0>0$ such that, if we have 
\begin{equation}
\frac{\upnu_\star \left(\overline{\D^2(x_0, r)} \right)}{r} < \upeta_0, {\rm \ then   \  it  \  holds \  }
\displaystyle{\upnu_\star \left(\overline{\D^2(x_0,\frac{r}{2}}) \right)=0.}
\end{equation}
\end{theorem}

  The  previous  statement leads   to consider the 1-dimensional  lower density   of the measure $\upnu_\star$ defined, for $x \in \Omega$,  by 
$$\theta_\star (x_0)={\underset{r \to 0} \liminf } \frac{\upnu_\star \left(\overline{\D^2(x_0, r)} \right)}{r},   $$
 and motivates us to define the set $\mathfrak S_\star$ as the concentration set  of the measure $\upnu_\star$. More precisely,  we set 
\begin{equation}
\label{mathfrakSstar}
\mathfrak S_\star=\{ x \in \Omega, \theta_\star (x_0)\geq\upeta_0\}, 
\end{equation}
where $\upeta_0>0$ is the constant provided by Theorem \ref{claire}. 
The fact that $\mathfrak S_\star$ is   closed of finite one-dimensional Hausdorff measure is then a rather direct consequence of the clearing-out property for the measure $\upnu_\star$ stated in Theorem \ref{claire}. The connectedness properties of $\mathfrak S_\star$ stated in Theorem \ref{maintheo}, part ii) require  a different type of clearing-out result.   Its statement involves general  regular subdomains $\mathcal U \subset \Omega$,  and,  for $\updelta>0$,  the related sets 
\begin{equation}
\label{Udelta}
\left\{
\begin{aligned}
 \mathcal U_\updelta&=\left \{  x \in \Omega,  {\rm dist}(x, \mathcal U)\leq \updelta\right\} {\rm \ and \ }  \\
 \mathcal V_\updelta&=\mathcal U_\delta \setminus \mathcal U
 =\left \{  x \in \Omega, 0\leq  {\rm dist}(x, \mathcal U)\leq \updelta\right\}. 
 \end{aligned}
\right.
\end{equation}

\begin{theorem}
\label{bordurer} 
 Let $\mathcal U\subset \Omega$ be a open  subset of $\Omega$ and $\updelta>0$ be given. If we have  
\begin{equation}
\label{clearstream}
 \upnu_\star  (\mathcal V_\delta)=0  , {\rm \ then   \  it  \  holds \  }
\upnu_\star \left(\overline{\mathcal U} \right)=0.
\end{equation}
\end{theorem}
In other terms, if the  measure $\upnu_\star$ vanishes in some neighborhood of the boundary $\partial \mathcal U$, then it vanishes on 
$\overline {\mathcal U}$. This result will allow us to establish  connectedness properties of $\mathfrak S_\star$.   For instance, we will prove the following \emph{local connectedness property}:

\begin{proposition}
\label{connective}
Let $x_0\in \Omega$, $r>0$ tels que $\D^2(x_0, 2r) \subset \Omega$. There exists a radius $r_0\in (r, 2r)$ such that $\mathfrak S_\star \cup \D^2(x_0, r_0)$ contains a finite union of path-connected components. 
\end{proposition}

 This connected properties  imply  the  rectifiability  of $\mathfrak S_\star$, invoking classical  results on continua  of bounded one-dimensional Hausdorff measure (see e.g \cite{falconer}). 

  \subsection{Elements in  the proofs of Theorem \ref{claire} and  Theorem \ref{bordurer}}
  The proofs of the above theorems are derived from corresponding results at the $\eps$ level for the map $u_\eps$, for given $\eps >0$.  We describe nexts these results. 
  
  \subsubsection{Invariance of the equation}
  \label{squale}
  As a first  preliminary remark,    we notice the invariance of the equation by translations as well as scale changes, which  plays an important role in our later arguments. Given  fixed  $r>0$  and $\eps>0$,  we consider the scalar parameter $\displaystyle{ \tilde \eps=\frac{\eps}{r} }$. For a given map $u_\eps : \D^2(x_à, r) \to \R^k$, we   introduce  the  \emph{scaled map}  $ \tilde u_{\eps}$ defined on the disk $\D^2$  by 
  $$ \tilde u_{\eps}(x)=u_\eps ( rx+ x_0)), \forall x \in \D^2.$$
  If the map $u_\eps$  is a solution to \eqref{elipes}, when  the map $\tilde u_{\eps}$ is a solution to \eqref{elipes} with  the parameter $\eps$ changed into $\tilde \eps$.  The scale invariance of the energy  is given by  the relation 
   \begin{equation}
   \label{scaling0}
   e_{\tilde \eps}(\tilde u_\eps)(x)= r e_\eps(u)(rx+ x_0), \, \forall x \in \D^2,
   \end{equation}
    which yields in its  integral forms
   \begin{equation}
   \label{scalingv}
   {\E_\eps}\left(u_\eps, \D^2(r)\right)= {r} {\E}_{\tilde \eps}\left(\tilde u_\eps, \D^2(1)\right)  {\rm \ and \ } 
   \mathbb V_\eps \left(u_\eps, \D^2(r)\right)= {r} \mathbb V_{\tilde \eps} \left(\tilde u_\eps, \D^2(1)\right),
      \end{equation}
   where we have set, for a  given domain $G$ and a map $u: G \to \R^k$
   $$ \E_\eps \left(u, G \right)\equiv \int_G e_\eps(u){\rm d}x  {\rm \ and \ }
   \mathbb V_\eps \left(u, G\right)  \equiv \int_G \frac{V(u)}{\eps}{\rm d}x. 
   $$
It follows from the previous discussion that the parameter $\eps$ as well as  the energy $\E_\eps$  behave, according to scaling, essentially  as lengths. In this  loose  sense,  inequality \eqref{scalingv} shows  that the quantity $\eps^{-1} E_\eps$ is scale invariant, according to the previous scale changes. 
  
  \subsubsection{The $\eps$-clearing-out theorems}
  We next provide clearing-out results for solutions of the PDE \eqref{elipes}. 
  
  \smallskip
    In view of the   assumptions  ${(\text{H}_1)}$, ${(\text{H}_2)}$ and ${(\text{H}_3)}$ on the potential $V$,  we may choose some  constant  $\upmu_0>0$ sufficiently small so that
    \begin{equation}
    \label{kiv}
    \left\{
    \begin{aligned}
 B^k(\upsigma_i, 2\upmu_0)\cap \B^k(\upsigma_j, 2\upmu_0) &= \emptyset 
 {\rm \ for \  all  \ } i\neq j  {\rm \  in \ }  \{1,\cdots,q\} {\rm \ and \  such  \   that  \ }  \\
 \frac{1}{2}\lambda_i^-{\rm Id} \leq  
 \nabla^2 V (y) &\leq 2\lambda_i^+{\rm Id}   \  \ 
  {\rm \ for \  all  \ } i\in \{1,\cdots,q\} {\rm \  and  \ }  y \in B(\upsigma_i,2\upmu_0).
 \end{aligned}
  \right. 
\end{equation}
We then have:

 \begin{theorem}
 \label{clearingoutth}
  Let  $0<\eps \leq 1$ and  $u_\eps$  be a solution of 
  \eqref{elipes}  on $\D^2$. There exists some constant $\upeta_0>0$ such that 
  if 
  \begin{equation}
  \label{petitou}
   \Eps(u_\eps, \D^2) \leq  \upeta_0,
   \end{equation}
  then  there exists some $\upsigma \in \Sigma$ such that 
  \begin{equation}
  \label{benkon}
  \vert u_\eps(x)-\upsigma  \vert  \leq \frac{\upmu_0}{2}, {\rm \ for \ every \ } x \in \D^2(\frac34), 
  \end{equation}
  where $\upsigma_0$ is defined in \eqref{kiv}. Moreover, we have the energy estimate, for  some constant 
  $\Cnrg>0$ depending only on the potential $V$
 \begin{equation}
 \label{engie}
\E_\eps\left(u_\eps, \D^2\left(\frac 58\right)\right) \leq \Cnrg \, \eps E_\eps(u_\eps, \D^2). 
 \end{equation}
 \end{theorem}
  
 The main ingredient in the proof of Theorem \ref{clearingoutth} is provided by the following  estimate:
  
  \begin{proposition}
  \label{sindec}
   Let  $0<\eps \leq 1$ and  $u_\eps$  be a solution of 
  \eqref{elipes}  on $\D^2$. There exists a constant  $\Cdec>0$ such that
   \begin{equation}
   \label{bien}
    \int_{\D^2(\frac {9}{16})}e_\eps (u_\eps) {\rm d} x \leq  \Cdec  \left[ 
    \left(\int_{\D^2} e_\eps (u_\eps){\rm d} x\right)^{\frac 32}+ 
   \eps  \int_{\D^2} e_\eps (u_\eps){\rm d} x
    \right]. 
   \end{equation}
  \end{proposition}
  
  Proposition \ref{sindec} is perhaps the main new ingredient provided by the present paper: When both $\E_\eps(u_\eps)$ and $\eps$ are small, it provides a fast decay of  the energy on smaller balls.

 \smallskip
  
  Combining the result \eqref{bien} of proposition \ref{sindec}  with the scale invariance properties of the equation given in  subsection \ref{squale},  we obtain corresponding results for arbitrary discs $\D^2(x_0, r)$.  Indeed, applying  Proposition \ref{sindec}  to the map $\tilde u_{\eps}$ with parameter $\tilde \eps$ and  expressing the corresponding inequlity \eqref{bien}  back by scale invariance  in terms of the original map $u_\eps$, we are led, provide $\eps \leq  r$, to the inequality
  \begin{equation}
  \label{bienscale}
 \E_\eps\left(u_\eps, \D^2\left(x_0, \frac{9r}{16}\right)\right)  \leq  \Cdec
  \left[\frac{1} {\sqrt r}\, 
    {\left(\E_\eps \left(u_\eps, \D^2(x_0, r)\right) \right)}^{\frac 32}+ 
  \frac{ \eps}{r} \E_\eps \left(u_\eps, \D^2(x_0, r)\right)
    \right]. 
    \end{equation}

 Iterating this decay estimate  on concentric discs centered at $x_0$, and combinig with elementary properties of the solution $u_\eps$, we eventually obtain  the proof of Theorem \ref{clearingoutth}.

 \smallskip
 Invoking once more  the scale invariance properties of the equation given in  subsection \ref{squale}, the scaled version of Theorem \ref{clearingoutth} writes then  as follows: 
 
 \begin{proposition} 
 \label{cestclair}
 Let $x_0 \in \Omega$ and $0<r \leq \eps$ be given,  assume that $\D^2(x_0, r) \subset \Omega$ and let $u_\eps$ be a solution of \eqref{elipes} on $\Omega$. If
 \begin{equation}
 \label{donnez}
 \frac{\E_\eps\left (u_\eps, \D^2\left(x_0, r\right)\right)}{r} \leq  \upeta_0, 
 \end{equation}
 then there exist some $\upsigma \in \Sigma$ such that 
  \begin{equation}
  \label{benkono}
  \left\{
  \begin{aligned}
  \vert u_\eps(x)-\upsigma  \vert & \leq \frac{\upmu_0}{2}, {\rm \ for \ }  x \in \D^2(x_0, \frac{3r}{4}) {\rm \ and }\\\
 \E_\eps\left(u_\eps, \D^2\left(x_0, \frac{ 5r}{8}\right)\right)& \leq \Cnrg \, \frac{\eps}{r} E_\eps\left(u_\eps, \D^2\left(x_0, r\right)\right).
  \end{aligned}
  \right.
  \end{equation}
\end{proposition}
 The proof of Proposition \ref{cestclair} is straightforward. 
  Passing to the limit $\eps \to 0$, Proposition \ref{cestclair} yields rather directly a proof to Theorem \ref{claire}. 

\smallskip
The proof of  Theorem \ref{bordurer} requires some slightly different argument. The main step, at the $\eps$-level, is provided by Proposition \ref{pave}.

\subsection{Plan of the paper}
  This paper is organized as follows.   The next two section are devoted to preliminary results paving the way to the proofs of the main results: Section \ref{engie2} presents some consequences of the energy bound, starting with  estimates  on one-dimensional sets, as well as consequences of the co-area formula, whereas Section \ref{pde} presents   properties,  including standard ones,  of the PDE \eqref{elipes}. For a large part, in both parts,  special  emphasis is put on energy estimates on  \emph{level sets}.  Section \ref{challengedata} presents the proof of Proposition \ref{sindec}.  In Section \ref{solde}, we provide the proof of Theorem \ref{clearingoutth}.  Section \ref{sectionsix} provides properties of the set $\mathfrak S_\star$. The proof of the main result is completed in Section \ref{proofmain}.

   \numberwithin{theorem}{section} \numberwithin{lemma}{section}
\numberwithin{proposition}{section} \numberwithin{remark}{section}
\numberwithin{corollary}{section}
\numberwithin{equation}{section}

 \section{First consequences of the  energy bounds}   
 \label{engie2} 
    The next results are  based on an idea of Modica and Mortola \cite{mortadela} adapted to the vectorial case in \cite{baldo, fontar}. The results in this section  apply to maps having a suitable bound   on there  energy $\Eps$, of the type of the bound \eqref{naturalbound}. They \emph{do not involve the PDE}. We stress in particular $BV$ type bounds obtained under these energy bound.  
\subsection {Properties of the potential}
\label{potentiel}
 It follows  from the definition of $\upmu_0$ and  property \eqref{kiv}  that we have the following behavior near the points of $\Sigma$:
 
\begin{proposition} 
\label{potto}
 For any $i=1, \ldots, q$  and any   $y \in \B^k(\upsigma_i, 2\upmu_0)$, we have the local bound
\begin{equation}
\label{glutz}
\left\{
\begin{aligned}
\frac{1}{4}\lambda_i^-\vert y-\upsigma_i \vert^2 & \leq V(y) \leq  \lambda_i^+\vert y-\upsigma_i \vert^2
  \\
  \frac{1}{2}\lambda_i^-\vert y-\upsigma_i \vert^2 & \leq  \nabla V(y) \cdot (y-\upsigma_i) \leq  2\lambda_i^+\vert y-\upsigma_i \vert^2, 
    \end{aligned}
  \right.
 \end{equation}   
  Setting $\lambda_0=\inf\{\lambda_i^-, i=1, \ldots, q\}$, we may assume, choosing possibly en even smaller constant $\upmu_0$,  that 
   \begin{equation}
   \label{extrut}
   V(y) \geq  \upalpha_0\equiv\frac{1}{2} \lambda_0 \upmu_0^2 {\rm \ on \ } 
   \R^k  \setminus  \underset{i=1}{\overset {q} \cup} \B^k(\upsigma_i, \upmu_0). 
  \end{equation}
  
  \end{proposition}
  The proof relies on a straightforward integration of \eqref{kiv} an we therefore omit it .
 Proposition \ref{potto} hence shows that  the potential   $V$  essentially behaves  as  a quadratic potential near points of the vacuum manifolds $\Sigma$.   This will be used throughout as a guiding thread. Proposition \ref{potto}  leads to a first  elementary observation: 
  
  \begin{lemma}
  \label{watson} Let $ y \in \R^k$ be such that $V(y)<\upalpha_0$. Then there exists some point $\upsigma\in \Sigma$ such that 
  $$\displaystyle{ \vert y-\upsigma  \vert \leq \upmu_0.}$$
  Moreover, we have the upper bound
  $$ \vert y-\upsigma  \vert  \leq \sqrt{ 4\lambda_0^{-1} V(y)}.$$
  \end{lemma}

  We next turn to the behavior at infinity. For that purpose, we introduce the radius 
\begin{equation}
\label{grandr0}
{\rm R}_0=\sup\{\vert \upsigma  \vert  , \upsigma \in \Sigma\}
\end{equation}
On study the properties of $V$ on the  set  $\R^k\setminus \B^k(2{\rm R}_0)$.

\begin{proposition} 
\label{barre}
 There exists a constant $\upbeta_\infty >0$ such that 
\begin{equation}
\label{upbetainfty}
V(y) \geq \upbeta_\infty \vert y \vert^2  {\rm \ for \ any \ } y {\rm \ such \ that \ }  \vert y \vert  \geq  2 {\rm R}_0.  
\end{equation}
\end{proposition}
\begin{proof}
Integrating assumption ${\rm H}_3$ we obtain that, for some constant $C_\infty>0$, we have
   \begin{equation}
   \label{cahors}
   V(y) \geq \frac{\upalpha_\infty  \vert y \vert^2}{2}-C_{\infty}, {\rm \ for \ any \ } y \in \R^k. 
   \end{equation}
It follows that
\begin{equation}
\label{tramoo}
 V(y) \geq  \frac{\upalpha_\infty  \vert y \vert^2}{4}, {\rm \ provided   \  } \vert y \vert  \geq  {\rm  R'}_0 \equiv \sup \left\{ 2\sqrt{\frac{C_\infty}{\upalpha_\infty}},\,  4R_0\right\}.
 \end{equation}
On the other hand, by assumption
$$\frac{V(y)}{\vert y \vert^2} >0 {\rm \ for \ } y \in \overline{\B^k({\rm R'}_0)\setminus \B^k(2{\rm R}_0)}, $$
so that, by compactness, we deduce that there exist some constant $\upalpha'_\infty>0$, such that
$$V(y)  \geq  \upalpha'_\infty \vert y \vert^2   {\rm \ for \ } y \in \overline{\B^k(2{\rm R'}_0)\setminus \B^k(2{\rm R}_0)}. $$
Combining the last inequality with \eqref{tramoo},  the conclusion follows, choosing 
$\displaystyle{\upbeta=\inf \{ \frac{\upalpha_\infty}{4}, \upbeta'_\infty\}  }$.
\end{proof}
\subsection{Modica-Mortola type inequalities}  
\label{momo} 
   
  Let $\upsigma_i$ be an arbitratry  element in $\Sigma$.   We consider  the function  $\chi_i: \R^k \to \R^+$  defined by
  $$ \chi_{_i}(y)=\varphi (\vert y-\upsigma_i  \vert) {\rm \ for \ } y \in \R^k, $$
where $\varphi$ denotes a function  $\varphi : [0, +\infty [ \to \R^+ $ such that $0\leq \varphi'  \leq 1 $  and 
    $$  \varphi(t)=t {\rm \ if \  } 0\leq t \leq {\upmu_0}  {\rm \ and \ }  \varphi (t)=\frac{5\upmu_0 }{4}
     {\rm \ if \  }  t\geq \upmu_0.$$
  Given a function $u: \Omega \to \R^k$ we finally define the \emph{scalar} function $w_i$  on $\Omega$  as
  \begin{equation}
  \label{doublevi}
    w_i(x)=\chi_i ( u(x)), \forall x  \in \Omega. 
 \end{equation}
 First properties of the map $w_i$ are summarized in the next Lemma.

   \begin{lemma} Let $w_i$ be as above. We have    
   \begin{equation}
      \label{monster}
      \left\{
      \begin{aligned}
   w_i (x)&=\vert u(x)-\upsigma_i  \vert,   {\rm \ if \ }\vert u(x)-\upsigma_i \vert \leq  \frac{\upmu_0}{2},\\
    w_i(x)&= \frac{3\upmu_0 }{4},   {\rm \ hence  \  }  \nabla w_i= 0 \ {\rm \ if \ }\vert u(x)-\upsigma_i \vert  \geq \upmu_0, \\
    \vert \nabla w^i \vert &\leq \vert \nabla u \vert {\rm \  on  \ }  \Omega, 
    \end{aligned}
    \right.
    \end{equation}
and 
 \begin{equation}
      \label{bvbound}
    \vert   \nabla (w_i)^2\vert
     \leq 4\sqrt{\lambda_0}^{-1} J (u) (x), 
           \end{equation}
            where we have set 
           \begin{equation}
           \label{jetski}
           J(u)= \vert \nabla u \vert \sqrt{V(u)}.
           \end{equation}
\end{lemma}
\begin{proof}
Properties \eqref{monster} is a straightforward consequence of the definition \eqref{doublevi}. For \eqref{bvbound}, we  notice that, in view of \eqref{monster}, we may restrict ourselves to the case $u(x) \in \B^k(\upsigma_i, \upmu_0)$, since otherwise $\nabla w_i=0$, and inequality \eqref{bvbound} is hence straightforwardly satisfied. In that case, it follows from \eqref{watson}, we have 
$$
 \vert  w_i (x) \vert \leq  \vert u(x)-\upsigma_i\vert \leq  \sqrt{ 4\lambda_0^{-1} V(u(x))},  {\rm \ for \ all \ } x   {\rm \ such \ that \ }  u(x) \in \B^k(\upsigma_i, \upmu_0), 
$$
 so that 
 \begin{equation}
      \label{bvbounda}
    \vert   \nabla (w_i)^2(x)\vert=2\left\vert w_i  (x) \right\vert.
    \left\vert \nabla \left \vert  w_i (x) \right\vert \right \vert \,  
    \leq  2 \vert \nabla u \vert \sqrt{ 4\lambda_0^{-1} V(u(x))} 
     \leq 4\sqrt{\lambda_0}^{-1} J (u) (x), 
           \end{equation}
 and the proof  is complete.
  \end{proof}

\begin{lemma} 
\label{ab0}
 We have, for any $x \in \Omega$, the inequality
   \begin{equation}
          \label{ab}
         J(u(x)) 
        \leq 
         e_\eps(u(x)). 
\end{equation}
\end{lemma}
\begin{proof}
We have, by definition of the energy $e_\eps(u)$,
 \begin{equation}
          \label{ab}
         J(u(x)) =(\sqrt{\eps} \vert \nabla u(x) \vert ). \sqrt{\eps^{-1}V(u(x) } 
          \end{equation}
          We invoke next the inequality $\displaystyle{ab\leq \frac{1}{2}(a^2+b^2)}$ to obtain
          $$J(u(x) \leq \frac12 \left (  \eps\vert \nabla u (x)\vert^2+ \eps^{-1} V(u(x)\right), $$ 
          which yields the desired result. 
      \end{proof}
    \subsection{The one-dimensional case}

     In dimension $1$  estimate \eqref{bvbound} directly leads to uniform bound on $w_i$, as expressed in our next result. 
 For that purpose, we consider, for $r>0$, the  circle $\S^1(r)=\{x\in \R^2, \vert x \vert =r\}$ and maps $u: \S^1(r) \to \R^k$.

     \begin{lemma} 
     \label{valli}
     Let $0< \eps\leq 1$ and $\eps<r\leq1$ be given. There exists a constant ${\rm C}_{\rm unf}>0$ such that, for any   given   $u:\mathbb S^1(r) \to \R^k$,  there exists an element $\upsigmam \in \Sigma$ such that 
      \begin{equation}
      \label{bornuni}
            \vert u(\ell)-\upsigmam \vert \leq  {\rm C}_{\rm unf}\sqrt{\int_{\S^1(r)} \frac 12 (J(u(\ell))+r^{-1} V(u(\ell))){\rm d} \ell}, \\
\  \   {\rm \ for \ all \ } \ell \in \S^1(r),
      \end{equation}
    and hence 
      \begin{equation}
      \label{bornunif}
     \vert u(\ell)-\upsigmam \vert    \leq  {\rm C}_{\rm unf}\sqrt{\int_{\S^1 (r)} e_\eps(u){\rm d}\ell}.  
      \end{equation}
       \end{lemma}
      
      \begin{proof} 
      By the mean-value formula, there exists some point $\ell_0 \in \S^1(r)$ such that 
      \begin{equation}
      \label{buena}
      V(u(\ell_0)) = \frac{1}{2\pi r} \int_{\S^1(r) } V(u(\ell)) {\rm d} \ell.
      \end{equation}
   We distinguish two cases.  
   
   \smallskip
   \noindent
   {\bf  Case 1.} {\it The function $u$ satisfies additionnally the estimate }
   \begin{equation}
    \label{val}
   \frac{1}{2\pi r} \int_{\S^1(r)} V(u(\ell))  \, {\rm d} \ell <  \upalpha_0, 
     \end{equation}
{\it  where $\upalpha_0$ is the constant introduced in Lemma \ref{watson}}.
    Then,  we deduce from inequality \eqref{val}  that
   \begin{equation*}
      \label{buena}
      V(u(\ell_0)) \leq \frac{1}{2\pi r}  \int_{\S^1(r)} V(u(\ell))  \, {\rm d} \ell < \upalpha_0.
            \end{equation*}
      It follows from Lemma \ref{watson} that there exists some $\upsigmam \in \Sigma$ such that 
     $$ \vert u(\ell_0))-\upsigmam  \vert^2  \leq 4 \lambda_0^{-1} V(u(\ell_0) )\leq 
      \frac {2\lambda_0^{-1}}{\pi r} \int_{\S^1(r)} V(u)  {\rm d} \ell. $$
    On the other hand,  we deduce,  integrating  the bound \eqref{bvbound},  that, for any $\ell \in \S^1(r)$, we have 
    $$ \vert \left\vert u-\upsigmam\right\vert^2(\ell)-\left \vert u-\upsigmam\right\vert ^2(\ell_0))\vert \, {\rm d} \ell
     \leq
      4\sqrt{\lambda_0^{-1}} \int_{\S^1(r)} J(u).$$
    Combining the two previous estimates,  we obtain the desired result  in case 1, using the fact that $\eps \leq 1$ and provided the constant $ {\rm C}_{\rm unf}$ satisfies the bound
    $$ {\rm C}_{\rm unf}^2 \geq  4\sqrt{\lambda_0^{-1}}+2\lambda_0^{-1}.$$
   
   \bigskip
   \noindent
   {\bf  Case 2}. {\it  Inequality \eqref{val} does not hold}. 
   In that case,  we have hence 
  \begin{equation}
  \label{trouville}
\frac{1}{2\pi r}  \int_{\S^1(r)} V(u(\ell))  {\rm   d} \ell \geq \upalpha_0.
   \end{equation}
 We consider  the number $\rm R_0=\sup\{\vert \upsigma  \vert  , \upsigma \in \Sigma\}$,   introduced in definition \eqref{cahors} of the proof of Proposition  \ref{barre},   and discuss next three subcases.
 
 \smallskip
 \noindent
 {\it Subcase 2a : For any $\ell \in \S^1(r)$,   we have
 $$ u(\ell) \in \B^k(2 \rm R_0).$$}
Then, in this case,  for any  $\upsigma \in \Sigma$, we have
\begin{equation}
\vert u(\ell)- \upsigma\vert^2 \leq 9{\rm R}_0^2=\left( \frac{9{\rm R}_0^2}{\upalpha_0}\right)\upalpha_0 \leq   
\left( \frac{9{\rm R}_0^2}{\upalpha_0}\right)\frac{1}{2\pi r}  \int_{\S^1(r)} V(u(\ell))  {\rm   d} \ell, 
\end{equation}
 so that in that case, inequality \eqref{bornuni} is immediately satisfied, whatever the choice of $\upsigma_{\rm main}$,  provided we impose the additional condition
 \begin{equation}
 \label{fayat}
  {\rm C}_{\rm unf}^2 \geq \frac{9{\rm R}_0^2}{2\upalpha_0}.
\end{equation}

\smallskip
 \noindent
 {\it Subcase 2b : There exists some  $\ell_1 \in \S^1(r)$, and some $\ell_2\in \S^1(r)$ such that,  we have}
 $$ u(\ell_1) \in \B^k(2 {\rm R}_0)  {\rm \ and \ }  u(\ell_2) \not  \in \B^k(2 \rm R_0).$$
Let $\ell \in \S^1(r)$. If $u(\ell) \in \B^k(2\rm R_0)$, then we argue as in subcase 2a, and we are done. Otherwise, by continuity, there exists some $\ell' \in \S^1(r)$ such that $u(\ell') \in \partial \B^k(2{\rm R}_0)$ and such for  any point $a\in \mathcal C(\ell, \ell')$ we have   $u(a) \not \in \B^k(2 \rm R_0)$,   where $\mathcal C(\ell, \ell')$ denotes the arc on $\S^1(r)$ joining counterclockwise on $\ell$ and $\ell'$.   We have, by integration  and using inequality \eqref{upbetainfty}, 
\begin{equation*}
\begin{aligned}
\vert u(\ell)\vert^2-\vert u(\ell')\vert ^2 &\leq 2\int_\ell^{\ell'} \vert u(a) \vert \cdot \vert \nabla u(a) \vert  \,  {\rm d} a  \\
 &\leq   \frac{2}{\sqrt{\upbeta_\infty}  }\int_\ell^{\ell'} V(u(a)) \vert \nabla u(a) \vert  \,  {\rm d} a
 \leq  \frac{2}{\sqrt{\upbeta_\infty}  } \int_{\S^1(r)} J(u(a))   {\rm d} a.
 \end{aligned}
\end{equation*}
Since $\vert u (\ell') \vert=2{\rm R}_0$, we obtain, for any $\upsigma \in \Sigma$, 
\begin{equation*}
\begin{aligned}
  \vert u(\ell)- \upsigma\vert^2  &\leq 2 \left(\vert u(\ell)\vert^2 + \vert \upsigma \vert^2 \right)\leq 2 \left(\vert u(\ell)\vert^2 + {\rm R}_0^2 \right) \\
   &\leq 2 \left( \frac{2}{\sqrt{\upbeta_\infty}  } \int_{\S^1(r)} J(u(a))   {\rm d} a +{\rm R}_0^2 + \vert u(\ell') \vert^2  \right) 
   \leq  \left( \frac{4}{\sqrt{\upbeta_\infty}  } \int_{\S^1(r)} J(u(a))   {\rm d} a + 10{\rm R}_0^2   \right) \\
   &\leq   \left( \frac{4}{\sqrt{\upbeta_\infty}  } \int_{\S^1(r)} J(u(a))   {\rm d} a + 10\frac{{\rm R}_0^2 }{\upalpha_\infty}\upalpha_\infty   \right) \\
  & \leq   \left( \frac{4}{\sqrt{\upbeta_\infty}  } \int_{\S^1(r)} J(u(a))   {\rm d} a +
   \frac{10{\rm R}_0^2 }{2\pi\upalpha_0 r}  \int_{\S^1(r)} V(u(\ell))  {\rm   d} \ell 
    \right) \\
   \end{aligned}
\end{equation*}
So that the conclusion follows, imposing  again  an appropriate lower bound on   ${\rm C}_{\rm unf}$.

  \medskip
 \noindent
 {\it Subcase 2c :  For any $\ell \in \S^1(r)$,   we have
 $$  \vert  u(\ell) \vert \geq 2{ \rm R}_0 .$$}
  Let $\ell_0$ satisfy \eqref{buena}, so that, in view of Proposition \ref{barre}
  \begin{equation*}
  \vert u(\ell_0)\vert^2 \leq  \frac{1}{\upbeta_\infty}  V(u(\ell_0))=  
  \frac{1}{\upbeta_\infty } \left( \frac{1}{2\pi r}\int_{\S^1(r)} V(u(\ell)\right).
  \end{equation*}
  We obtain hence, for any arbitrary $\upsigma \in \Sigma$ 
 \begin{equation}
 \label{monge}
 \begin{aligned}
  \vert u(\ell_0)- \upsigma\vert^2 &\leq 2 \left(\vert u(\ell_0)\vert^2 + \vert \upsigma \vert^2 \right)\leq 
   \frac{2}{\upbeta_\infty} \left( \frac{1}{2\pi r} \int_{\S^1(r)} V(u(\ell)) { \rm d} \ell +  {\rm R}_0^2 \upbeta_\infty\right ) \\
   &\leq \frac{2}{\upbeta_\infty} \left( \frac{1}{2\pi r} \int_{\S^1(r)} V(u(\ell) ){ \rm d} \ell + 
   \upalpha_0\left( \frac{  {\rm R}_0^2 \upbeta_\infty}{\upalpha_0}\right)\right ) \\
   &\leq  \  \frac{1}{\pi \upbeta_\infty  } \left(  1+ \left(\frac{2{\rm R}_0^2 \upbeta_\infty}{\upalpha_0}\right) \right)
   \left( r^{-1} \int_{\S^1(r)} V(u(\ell)){ \rm d} \ell\right). 
    \end{aligned}
     \end{equation}
This yields again \eqref{bornuni} for an arbitrary choice of $\upsigmam \in \Sigma$ and imposing an additional  suitable lower bound on 
${\rm C}_{\rm unf}$.     

We have hence established  for upper bound \eqref{bornuni}  in all three possible  cases $2a, 2b$ and $2c$, for a suitable an arbitrary choice of $\upsigmam \in \Sigma$ and imposing an additional  suitable lower bound on 
${\rm C}_{\rm unf}$.      It is hence established in case $2$. Since we alreday establishes it in Case 1, the proof of \eqref{bornuni} is complete. 

\medskip 
Turning to inequality \eqref{bornunif}, we first observe that, since  by assumption $r\geq \eps$, we have 
\begin{equation}
\label{lepape}
r^{-1} \int_{\S^1(r)} V(u(\ell)){ \rm d} \ell \leq \int_{\S^1(r)} \eps^{-1}  V(u(\ell)){ \rm d} \ell
      \leq     \int_{\S^1(r)} e_\eps(u(\ell)){ \rm d} \ell.  
      \end{equation}
      Combining \eqref{bornuni} with \eqref{ab} and \eqref{lepape}, we obtain the desired result \eqref{bornunif}. 
     \end{proof}
    
   \subsection{Controlling the energy on  circles}
     \label{radamel}
    When working on two dimensional disk, the tools  developed in the previous section allow to choose  radii with appropriate control on the energy, invoking a standard  mean-value argument. More precisely, we have: 
     
     \begin{lemma}  
     \label{moyenne}
     Let   $\eps \leq r_0< r_1\leq 1$  and $u: \D^2 \to \R^k$ be given. There exists a radius $\mathfrak r_\eps \in [r_0, r_1]$ such that     
     $$ \int_{\S^1(\mathfrak r_\eps)}e_\eps(u){\rm  d } \ell  \leq \frac{1}{r_1-r_0} \, \E_\eps(u, \D^2(r_1)). $$
     \end{lemma}
   Energy estimates yield also uniform bounds in dimension one: Indeed, it     follows from Lemma \ref{valli} that there exists 
  some point  $\upsigma_{\mathfrak r_\eps} \in \Sigma$, \emph{depending on $\mathfrak r_\eps$},  such that 
   \begin{equation}
      \label{bornuni2}
     \vert u(\ell)-\upsigma_{\mathfrak r_\eps}  \vert \leq \frac{ {\rm C}_{\rm unf}}{\sqrt{r_1-r_0}} \sqrt{{\E}_\eps(u, \D^2(r_1)}),  \  \   {\rm \ for \ all \ } \ell \in 
     \S^1(\mathfrak r_\eps).
      \end{equation}
 Moreover, it follows from \eqref{bvbounda} that 
 \begin{equation}
 \label{bvbound2}
 \int_{\S^1(\mathfrak r_\eps)} \vert J(u) \vert \leq \frac{1}{r_1-r_0} \int_{\D^2(r_1)} e_\eps(u_\eps) {\rm d}x. 
 \end{equation}
 
     \subsection {BV estimates and the coarea formula}
     \label{detroit}
The right-hand side  of estimate \eqref{bornunif}, in particular the term  involving  $J(u)$, may be analyzed as a  $BV$ estimate (as in \cite{mortadela}). In dimension $1$, as expected, it yields  used a uniform estimates.  In higher dimensions of course, this is no longer true. Nevertheless our $BV$-estimates  interesting estimates on the measure of specific  level sets. In order to state the kind of results we have in mind,  we consider   a smooth function $\varphi:  \Omega\to \R$,  where $\Omega \subset \R^N$ is a general domain, and introduce, for an arbitrary number $s \in \R$,  the  level set
$$\varphi^{-1} (s)=\{s \in \Omega, {\rm \ such \ that \ } \varphi(x)=s\}.$$
If $w$ is assumed to be sufficiently smooth, then Sard's theorem asserts that $w^{-1}(s)$ is a  regular submanifold of dimension $(N-1)$, for almost every $s\in \R$, and the coarea formula relates the integral of the total  length of these curves to the $BV$-norm  through the formula     
     \begin{equation}
     \label{coarea}
     \int_{\R} {\mathcal H}^{N-1}\left(\varphi^{-1} (s)\right){\rm d}s=\int_\Omega  \vert  \nabla  \varphi  (x)\vert {\rm d}x. 
     \end{equation}
We specify this formula to the case $N=2$, $\Omega=\D^2(r)$, for some $r>\eps$,   and $\varphi=(w_i)^2$,   where $i \in \{1, \ldots, q\}$ and   where $w_i$ is the map constructed   in \eqref{doublevi} for a given $u: \Omega\to \R^k$. Combining \eqref{coarea} with \eqref{bvbound} and \eqref{ab}, we are led to the inequality
\begin{equation}
\label{coaforme}
\begin{aligned}
\int_{\R^+} {\mathcal L}\left((w_i^2)^{-1} (s)\right){\rm d}s&\leq 4 \sqrt{\lambda_0}^{-1}\int_{\D^2(r)}  J(u(x)){\rm d}x \\
&\leq  4 \sqrt{\lambda_0}^{-1}\int_{\D^2(r)}  e_\eps(u){\rm d}x=  4 \sqrt{\lambda_0}^{-1} \E_\eps\left(u_\eps, \D^2(r)\right), 
\end{aligned}
\end{equation}
where 
$\mathcal  L=\mathcal H^{1}$ denotes length. In most places, we will invoke this inequality jointly with a mean value argument.  This yields:

\begin{lemma} 
\label{claudio}
Let $u$, $w_i$ and $r$  be as above. Given  any  number $A>0$,     there exists some 
$\displaystyle{A_0  \in [ \frac{A}{2}, A]}$ such that $w_i^{-1} (s_0)$ is a regular curve and such that 
\begin{equation}
{\mathcal L}\left(w_i^{-1} (A_0)\right)  \leq\frac{6}{\sqrt{\lambda_0}A^2}\int_{\D^2}  e_\eps(u){\rm d}x
=\frac{8 \,  \E_\eps\left(u_\eps\right)} {\sqrt{\lambda_0}A^2}.
\end{equation}
\end{lemma}
\begin{proof}  In view of  Definition \ref{doublevi}, the map $w_i$ takes values in the interval $\displaystyle{[0, \frac{3\upmu_0}{4}]}$, so that 
$\displaystyle{w_i^{-1}(s) =\emptyset}$, if $\displaystyle{s>\frac{3\upmu_0}{4}}$. Hence, it remains only to consider the case $\displaystyle{A \leq  \frac{3\upmu_0}{4}}$.  We introduce the domain $\displaystyle{\Omega_{i, A}=\{x\in \D^2(r),\frac{A}{2} \leq  \vert u(x)-\upsigma_i \vert \leq A\}}$. Using formula \eqref{coaforme} on this domain, we are led to the inequality 
\begin{equation*}
\int_{\frac {A^2} {4}} ^{A^2} {\mathcal L}((w_i^2)^{-1} (s) ){\rm d} s  \leq 4 \sqrt{\lambda_0}^{-1} \E_\eps\left(u_\eps, \D^2(r)\right).
\end{equation*}
The conclusion that follows by a mean-value argument. 
\end{proof}


  \subsection{Controlling uniform bounds on good circles  }
Whereas in subsection  \ref{radamel} we have selected radii with controlled energy for the map $u$,   in this subsection, we select radii with appropriate uniform bounds on $u$. 
 We assume that we  are given a radius $\varrho \in [\frac 12, 1]$, a number 
  $\displaystyle{0<\upkappa <\frac{\upmu_0}{2}}$, a smooth map $u:\overline{ \D^2(\varrho)} \to \R^k$   and an element $\upsigma \in \Sigma$ such that 
   \begin{equation}
  \label{kappacite}
   \vert u-\upsigma \vert  < \upkappa  {\rm \ on \ } \partial \D^2(\varrho).
  \end{equation}
  We introduce  the subset  $\mathcal I (u, \upkappa)$  of radii $\displaystyle{r \in  [\frac 12, \varrho]}$  such that 
\begin{equation}
\label{sunyu} 
\mathcal I (u, \upkappa) =\left \{ r \in [\frac{1}{2}, \varrho ] {\rm \ such \ that \ }
 \vert u  (\ell) -\upsigma \vert \leq  \upkappa, \, \forall \ell \in \S^1(r) \right \}. 
 \end{equation}
   We have: 

  \begin{proposition}   
  \label{jarre}
  We have  the lower bound
 \begin{equation}
    \label{clamart}
    \vert \mathcal I (u, \upkappa) \vert \geq \varrho-\frac{9}{16}, 
    \end{equation}
    provided 
    \begin{equation} 
    \label{camembert}
 \upkappa^2  \geq  \frac{1}{32\sqrt{\lambda_0}} \E_\eps(u). 
\end{equation}

    \end{proposition}
\begin{proof}
We consider the number $\displaystyle{A_0  \in [ \upkappa , 2\upkappa]} $ given by Lemma \ref{claudio} with the choice $r=\varrho$ and $A=\upkappa$, so that $w^{-1} (A_0)$ is smooth and 
\begin{equation*}
\mathcal L (w^{-1} (A_0)) \leq \frac{6E_\eps (u_\eps)}{4\sqrt{\lambda_0} \upkappa^2} \leq \frac{2E_\eps (u_\eps)}{\sqrt{\lambda_0} \upkappa^2}.
 \end{equation*}
If moreover \eqref{camembert} is satisfied, then we have
\begin{equation}
\label{brie}
\mathcal L (w^{-1} (A_0)) < \frac{1}{16}.
\end{equation}
We introduce the auxiliary set 
\begin{equation*}
\left\{
\begin{aligned}
\mathcal J (u, \upkappa) &=\{ r \in [\frac{1}{2}, \varrho],  {\rm \ such \ that \ }
 \vert u_\eps  (\ell) -\upsigma \vert < A_0, \, \forall \ell \in \S^1(r) \},   {\rm \ and \ } \\
 \mathcal Z (u, \upkappa) &=\{ r \in [\frac{1}{2}, \varrho],  {\rm \ such \ that \ }
 \vert u_\eps  (\ell) -\upsigma \vert > A_0, \, \forall \ell \in \S^1(r) \}.
   \end{aligned}
 \right.
  \end{equation*}
We first show that 
\begin{equation}
\label{euphrate0}
\displaystyle{\mathcal Z (u, \upkappa)= \emptyset}. 
\end{equation}
  Indeed,  consider any arbitrary radius 
 $\frac 12 \leq r\leq \varrho$ in  $\mathcal Z (u, \upkappa)$.     Since $\vert u_\eps-\upsigma \vert <\upkappa <A_0$ on $\partial \D^2(\varrho)$ and since,    by definition of $\mathcal Z(u, \upkappa)$,  we have $\vert u_\eps-\upsigma \vert >A_0$ on $\partial \D^2(r)$, it follows that there is a  smooth domain $V$ such that  $u(x)=A_0$ for $ x \in \partial V$ and $\D^2(r)\subset V\subset \D^2(\varrho)$.  We deduce  from the two previous assertions   that, since by assumption $1\slash 2 \leq r \leq \varrho$,  
 $$
    \partial V\subset w^{-1}(A_0)               {\rm \ and \ } \mathcal L(\partial V) \geq 2\pi r \geq \pi , 
 $$
Hence, we obtain  
$$
\mathcal L(w^{-1}(A_0) ) \geq \pi. 
 $$
 This however contradicts  inequality \eqref{brie} and hence establishes \eqref{euphrate0}.

We next  consider an arbitrary radius 
 $\frac 12 \leq r\leq \varrho$  such that $r \not \in  \mathcal   J (u, \upkappa)$. It follows from  the definition of $\mathcal J(u, \upkappa)$ that there exists some $\ell_r \in \S^1(\varrho)$ such that  $\vert u_\eps(\ell_r) -\upsigma \vert \geq  A_0$. We deduce therefore from \eqref{euphrate0} and the intermediate value theorem that 
  $$ 
 w^{-1}(A_0)   \cap \S^1(r) \neq \emptyset,  \, \,  \forall r \not \in \mathcal J(u, \upkappa).
$$
This relation implies,  by Fubini's theorem,  that 
$$
\mathcal  L ( w^{-1}(A_0) ) \geq \left (\varrho-\frac 12\right)- \vert \mathcal J (u, \upkappa)\vert, 
 $$
 so that 
 \begin{equation}
 \label{carensac0}
  \vert \mathcal J (u, \upkappa) \vert \geq  \left (\varrho-\frac 12\right)-\mathcal  L ( w^{-1}(A_0) ) \geq 
  \varrho- \frac{9}{16}, 
  \end{equation}
where we made use of  estimate \eqref{brie}.  Since  $0< \upkappa \leq  A_0 $ by construction, we have 
$$\mathcal J (u, \upkappa) \subset \mathcal I (u, \upkappa),  {\rm so \ that \ }
\vert \mathcal J (u, \upkappa) \vert \leq \vert \mathcal I (u, \upkappa) \vert.
$$  
 Combining with inequality \eqref{carensac0}, we obtain   the desired inequality \eqref{clamart}. 
 \end{proof}
 
 \subsection{Revisiting the control of the energy on concentric circles}
Using  the results of  the previous section, we  may work out variants of the Lemma \ref{moyenne}. For that purpose, given a radius $\varrho \in [\frac{3}{4}, 1]$, a number 
  $\displaystyle{0<\upkappa \leq \frac{\upmu_0}{2}}$, a smooth map $u:\overline{ \D^2(\varrho)} \to \R^k$   and an element $\upsigma \in \Sigma$ such that \eqref{kappacite}  holds, we introduce the set 
  \begin{equation}
  \label{subsub}
  \Upsilon_\upsigma(u,  \varrho,  \upkappa)=\left \{ x\in \D^2(\varrho),{ \rm \ such \ that \ } \vert u(x)-\upsigma \vert \rm  \leq \upkappa\right\}.
  \end{equation}
The following result is a major tool  in the proof of our main results:

     \begin{lemma}  
     \label{remoyen}
     Let $u, \varrho$ and $\upkappa$ be as above  and assume that the bound \eqref{camembert} holds.  Assume that $\varrho \geq \frac 34 $ There exists a radius $\displaystyle{ \uptau_\eps \in[\frac{5}{8}, \varrho]}$  such that $\S^1(\uptau_\eps) \subset \Upsilon_\upsigma(u,  \varrho,  \upkappa)$, i.e. 
     $$
     \vert u(\ell)-\upsigma) \vert \leq \upkappa,   {\rm  \ for \ any \ } \ell \in \S^1(\uptau_\eps), 
     $$
      and such that 
         $$ \int_{\S^1(\uptau_\eps)}e_\eps(u){\rm  d } \ell  \leq   \frac{1}{\varrho-\frac{11}{16}} \, \E_\eps(u, \Upsilon_\upsigma(u,  \varrho,  \upkappa)). $$
     \end{lemma}
     
\begin{proof} In view of definition  \eqref{subsub} of $  \Upsilon_\upsigma(u,  \varrho,  \upkappa)$ and the definition \eqref{sunyu} of $\mathcal I(u, \upkappa)$, we have 
 $ \S^1(r) \subset  \Upsilon_\upsigma(u,  \varrho,  \upkappa) $ for  any $r \in \mathcal I (u, \upkappa)$, so that, by Fubini's theorem, we have 
$$ 
\int_{\mathcal I (u, \upkappa)}  \left( \int_{\S^1(\varrho)}e_\eps(u_\eps){\rm d} \ell \right) {\rm d}\varrho \leq 
\int_{  \Upsilon_\upsigma(u,  \varrho,  \upkappa)} e_\eps(u_\eps){\rm d}x.
$$ 
Since we assume that the bound \eqref{camembert} holds, it follows from Proposition \ref{jarre} that 
$$ 
\vert \mathcal I (u, \upkappa) \vert \geq \varrho-\frac{9}{16} {\rm \ and  \ hence \ }
\vert \mathcal I (u, \upkappa) \cap   [\frac 58, \varrho] \vert \geq \varrho-\frac{11}{16}. 
$$
 Hence by a mean value argument that there exists some radius $\uptau_\eps   \in [\frac 58, \varrho] \cap \mathcal I_\eps$ such that 
 $$
  \int_{\S^1(\uptau_\eps )}e_\eps(u_\eps){\rm d} \ell \leq  \frac{1}{\varrho-\frac{11}{16}}\int_{ \Upsilon_\upsigma(u,  \varrho,  \upkappa)} e_\eps(u_\eps){\rm d}x,
 $$
 which is precisely  the conclusion. 
\end{proof}     

\noindent
{\bf Comment.}  The result above will be used in connection with the estimates for $u$ when $u$ is the solution to \eqref{elipes}. Thanks to the equation, we will be able to estimate the growth of $\E_\eps(u, \Upsilon_\upsigma(u,  \varrho,  \upkappa))$ with $\upkappa$. We will choose $\upkappa$ as small as possible to satify \eqref{camembert}, which amounts to choose of the magnitude of $\sqrt{\E_\eps(u)}$, as we will see in \eqref{cabrovski}.
\subsection{Gradient estimates on level sets}
  Given a arbitrary   smooth function $\varphi : \Omega \to \R$, where   $\Omega$  denotes a denote of $\R^N$, and an arbitrary integrable  function $f: \Omega \to \R$,  the coarea formula  \eqref{coarea} generalized as 
  \begin{equation}
     \label{coarea2}
     \int_{\R} \left (\int_{\varphi^{-1}(s)}  f(\ell){\rm d}\ell \right){\rm d} s=     \int_\Omega  \vert  \nabla  \varphi  (x)\vert  f(x){\rm d}x.
     \end{equation}
Given a smooth function $u: \Omega \to \R^k$,  we specify identity \eqref{coarea2}  with choices $\varphi =\vert u \vert$ and $f=\vert \nabla u \vert$: We are led to  the identity
 \begin{equation}
 \begin{aligned}
     \label{coarea3}
     \int_{\R} \left (\int_{\vert u \vert ^{-1}(s)} \vert \nabla u \vert (\ell){\rm d}\ell\right){\rm d} s&=     \int_\Omega  \vert  \nabla  u  (x)\vert. 
     \vert \nabla \vert u \vert  \vert {\rm d}x,  \\
     &\leq \int_\Omega  \vert  \nabla  u  (x)\vert^2{\rm d} x. 
    \end{aligned}
     \end{equation}
     We specify furthermore  this formula,  as in Subsection \ref{detroit},   for  a given map  $u$ defined on a disk $\D^2(r)$ and $w_i$ being  the corresponding  maps $w_i$ defined  on $\D^2(r)$ by formula  \eqref{doublevi}.  We introduce the subdomain 
       \begin{equation}
 \label{nablalala}
 \begin{aligned}
 \Theta (u, r)&=\left\{  x \in \D^2(r)  {\rm \ such \ that \ } u(x) \in \D^2(r) \setminus 
 \underset{i=1}{\overset {q} \cup} \B^k (\upsigma_i, \frac{\upmu_0}{2})
   \right \}   \\
   &=u^{-1}\left(\D^2(r) \setminus \underset{i=1}{\overset {q} \cup} \B^k(\upsigma_i, \frac{\upmu_0}{2}\right) 
   =  \underset{i=1}{\overset {q} \cup}  \Upsilon_{\upsigma_i} (u, r, \frac{\upmu_0}{2}).\\
  \end{aligned}
 \end{equation}
     We have: 
     
     \begin{lemma}
     \label{claudius}  
     Let $u$  be as above. 
  There exists   some  number   
  $\displaystyle{  \tilde \upmu  \in  [\frac{\upmu_0}{2},\upmu_0] }$, where $\upmu_0$ denotes the constant introduced Paragraph \ref{potentiel},  such that 
  \begin{equation}
  \label{claudius2}
\underset{i=1}{\overset{q}  \sum }  \int_{w_i^{-1}(\tilde \upmu)} \vert \nabla u \vert (\ell){\rm d}\ell  \leq  
  \frac{2}{\upmu_0} \int_{\Theta(u, r)} \vert \nabla u \vert^2 \leq \frac{4}{\upmu_0 \eps} \E_\eps(u, \Theta (u, r)).
  \end{equation}
        \end {lemma}
        
        \begin{proof}   It follows from identity \eqref{coarea3}, applied to $u-\upsigma_i$,  that 
     \begin{equation}
    \begin{aligned} 
     \underset{i=1}{\overset{q}  \sum }   \int_{\frac{\upmu_0}{2}}^{\upmu_0 }    \left( \int_{w_i^{-1}(s)} \vert \nabla u \vert (\ell){\rm d}\ell\right){\rm d} s&=   \int_{\frac{\upmu_0}{2}}^{\upmu_0 }  \underset{i=1}{\overset{q}  \sum }   \left( \int_{w_i^{-1}(s)} \vert \nabla u \vert (\ell){\rm d}\ell\right){\rm d} s \\
&\leq \int_{\Theta(u, r)} \vert \nabla u \vert^2  {\rm d} x. 
\end{aligned}
     \end{equation}
     We conclude once more by a mean-value argument.
        \end{proof}
        \section{Some  properties of the PDE}
        \label{pde}
       In this section, we recall first several classical  properties of the solutions to the equation \eqref{elipes}.  We then provide some energy and potential estimates (see e. g \cite{BBH}).
       
       \subsection{Uniform bound  through the maximum principle}
       We have:
       
       \begin{proposition}  Let $u_\eps\in H^1(\Omega)$  be a solution of \eqref{elipes}. Then we have the  uniform bound bound, for $x\in \Omega$
\begin{equation}
\label{princours}
\vert u(x) \vert^2 \leq  \frac{4 {\rm C}_{\rm unf}}{{\rm dist}(x, \partial \Omega)} {\E}_\eps(u_\eps) +2 \sup \{ \vert \sigma \vert^2, \upsigma \in \sigma \}.
\end{equation}
       \end{proposition}
       
       \begin{proof} 
       Arguing as in \cite{BBH0}, we compute, using equation \eqref{elipes}
       \begin{equation}
       \label{multiplie}
       \begin{aligned}
       \Delta \vert u_\eps \vert^2&=u_\eps\cdot \Delta u_\eps+ \vert \nabla u_\eps \vert^2   
       =\eps^{-2} u_\eps\cdot \nabla  V_u( u_\eps)+ \vert \nabla u_\eps \vert^2  \\
     &  \geq \eps^{-2} u_\eps\cdot \nabla  V_u( u_\eps),   {\rm \ on \ } \Omega. 
       \end{aligned}
       \end{equation}
     On the other hand, it follows from assumption \eqref{condinfty} that, there exists some constant $\upbeta_\infty \geq 0$ such that 
     \begin{equation}
     \label{convexita}
     y.\nabla  V( y )\geq \upalpha_\infty \vert y \vert^2- \upbeta_\infty  {\rm \ for \ any \ } y \in \R^N.
     \end{equation}
     Hence, combining \eqref{multiplie} and \eqref{convexita} we obtain the  inequality
     \begin{equation*}
    -\Delta \vert u_\eps \vert^2  +\upalpha_\infty \eps^{-2} \left( \vert u_\eps\vert^2 -\frac{ \upbeta_\infty}{\upalpha_\infty}\right) \leq 0 {\rm \ on \ } \Omega.
     \end{equation*}
     We set $W_\eps=\vert u_\eps \vert^2-\frac{ \upbeta_\infty}{\upalpha_\infty}$,  so that  we obtain the inequlaity for $V_\eps$ 
       \begin{equation}
       \label{prince}
    -\Delta W_\eps   +\upalpha_\infty \eps^{-2} W_\eps \leq 0  {\rm \ on \ } \Omega. 
     \end{equation}
     Let $x\in \Omega$ and set $R_x={\rm dist}( x, \partial \Omega)$, so that $\D^2(x, R_x)\subset \Omega$. It follows from Lemma \ref{moyenne} and inequality \ref{bornuni2} that there exists some radius $\displaystyle{\uptau \in [\frac{R_x}{2}, R_x]}$ and some element $\upsigma \in \Sigma$ such that 
      \begin{equation*}
      \label{bornunivert}
     \vert u_\eps(\ell)-\upsigma \vert \leq \frac{\sqrt{2} {\rm C}_{\rm unf}}{\sqrt{{R_x}}} \sqrt{{\E}_\eps(u_\eps, \D^2(R_x)})
    \leq  \frac{\sqrt{2} {\rm C}_{\rm unf}}{\sqrt{{R_x}}} \sqrt{{\E}_\eps(u_\eps)}), 
      \  \   {\rm \ for \ all \ } \ell \in 
     \S^1(\uptau).
      \end{equation*}
      Hence 
     $$  W_\eps (\ell)  = \vert u_\eps (\ell) \vert^2 -\frac{ \upbeta_\infty}{\upalpha_\infty}\leq  \frac{4 {\rm C}_{\rm unf}}{{R_x}} {\E}_\eps(u_\eps) +
      2\sup \{ \vert \sigma \vert^2\}-\frac{ \upbeta_\infty}{\upalpha_\infty},
       \   { \rm \ for \ all \ } \ell \in    \S^1(\uptau).
     $$
Since $W_\eps$ satisfies inequality    \eqref{prince} we may apply the maximum principle to assert that 
$$  W_\eps(x)  \leq  \frac{4 {\rm C}_{\rm unf}}{{\rm dist}(x, \partial \Omega)} {\E}_\eps(u_\eps) +2 \sup \{ \vert \sigma \vert^2\}-\frac{ \upbeta_\infty}{\upalpha_\infty}  \   { \rm \ for \ all \ } x \in    \D^2(x, \frac{R}{2}) 
     $$
     so that the conclusion follows.
       \end{proof}
       
       \subsection{Regularity and gradient bounds}       
    The next  result is  a  standard  a consequence of the smoothness of the potential, the regularity theory for the Laplacian and the maximum principle. 
    
    \begin{proposition}
    \label{classic}  Let $u_\eps\in H^1(\Omega)$  be a solution of \eqref{elipes} and assume that $\E_\eps (u_\eps) \leq M$, where $M\geq 0$. Then $u_\eps$ is smooth on $\Omega$ and given any $\delta>0$,  there exists some constant $K_{\rm dr}(M, \delta)>0$,  depending only on the potential $V$, $M$ and $\delta$, such  that,
   \begin{equation}
   \label{classicic}
 \vert    \nabla u_\eps  \vert (x) \leq \frac{K_{\rm dr}(M, \delta)}{\eps}, {\rm \ if \ } {\rm dist} (x, \partial \Omega)\geq \delta.
 \end{equation}
       \end{proposition}
   
 \begin{proof}  The estimate is a  consequence of Lemma A.1 of \cite{BBH0}, which assert that, if $v$ is a solution on  some domain $\mathcal O$ of 
 $\R^n$ of  $-\Delta v=f$,  then we have the inequality
 \begin{equation}
 \label{bbhoo}
 \vert \nabla v \vert^2(x) \leq C(\Vert f \Vert_{L^{\infty}(\mathcal O)} \Vert v \Vert_{L^{\infty}(\mathcal O)}  +
 \frac{1}{{\rm dist } (x, \partial \mathcal O)^2} \Vert v \Vert_{L^{\infty}(\mathcal O)}^2,   {\rm \ for \ all \ } x \in \mathcal  O.
 \end{equation}
 We apply inequality \eqref{bbhoo} to the solution $u_\eps$, with  source term  $f=\eps^{-2}\nabla_u V (u_\eps) $  on the domain $\mathcal O= 
 \D^2(x,  \frac{\delta}{2})$.  In view of Proposition \ref{princours}, we have 
 \begin{equation*}
 \left\{
 \begin{aligned}
 \vert u(x) \vert^2 &\leq  C\left ( \frac{M}{\delta}+1\right) ,   {\rm \ for \ } x  \in \mathcal O= \D^2(x,  \frac{\delta}{2}) \\
 \vert f (x)\vert &\leq   \eps^{-2} C(M, \delta) ,   {\rm \ for \ } x  \in  \mathcal O=\D^2(x,  \frac{\delta}{2}). 
 \end{aligned}
 \right.
 \end{equation*}
 Combining with \eqref{bbhoo} we derive the conclusion.
 \end{proof}
   
    The gradient bound described in Proposition \ref{classic} has important consequences when one compares the two terms  involved in the energy, the gradient term and the potential term.   As we will see in Lemma \ref{bornepote} below, it shows that the  potential term yields an upper bound for  the gradient term, at least when $u_\eps $ takes values  far  from the potential wells. 
 Restricting ourselves to the  case $\Omega=\D^2$, we introduce for $r>0$  the set 
 \begin{equation}
 \label{nablala}
 \begin{aligned}
 \Theta_\eps(r)\equiv \Theta( u_\eps, r)&=\left\{  x \in \D^2(r)  {\rm \ such \ that \ } u_\eps(x) \in \R^k \setminus 
 \underset{i=1}{\overset {q} \cup} \B^k(\upsigma_i, \frac{\upmu_0}{4})
   \right \}   \\
  & ={({u_\eps}_{_{\vert \D^2(r)}})}^{-1}\left(\R^k\setminus  \underset{i=1}{\overset {q} \cup} \B^k(\upsigma_i, \frac{\upmu_0}{4})
 \right).
    \end{aligned}
 \end{equation}
 On  $\Theta_\eps$ the energy can be estimated by the potential as follows:
 
 \begin{lemma}
\label{bornepote}
 $u_\eps\in H^1(\D^2)$  be a solution of \eqref{elipes} such that $E_\eps(u_\eps) \leq M_0$. There exists a constant $C_{\rm T}$ depending only on the potential $V$ and $M_0$ such that 
 $$
e_\eps(u_\eps)  \leq  C_{\rm  T }  \frac{V(u_\eps)}{\eps} \ 
{\rm  \ on \ } \Theta_\eps (\frac 34). 
 $$
 \begin{proof} It follows from the definition of $\Theta_\eps$  and in view of inequality \eqref{extrut} that 
 $$
 V(u_\eps(x))  \geq  \frac{\alpha_0}{16}, \, {\rm \ for \ } x \in \Theta_\eps.
 $$
 Going back to  \eqref{classic} we obtain, for $x\in \Theta_\eps$
$$ \eps  \vert \nabla u_\eps \vert^2(x) \leq \frac{K_{\rm dr}^2}{\eps} = \frac{\alpha_0}{4\eps} \left( \frac{4K_{\rm dr}^2}{\alpha_0}\right)
\leq   \left( \frac{4K_{\rm dr}^2}{\alpha_0}\right) \frac{V(u_\eps(x))}{\eps}, 
$$
 so that 
 $$e(u_\eps) \leq \left( \frac{2K_{\rm dr}^2}{\alpha_0}+1\right) \frac{V(u_\eps)}{\eps}.$$
 The conclusion follows choosing the constant   $C_{\rm T}$ as $\displaystyle{C_{\rm T}=  \left( \frac{4K_{\rm dr}^2}{\alpha_0}\right)}$. 
 \end{proof}
 \end{lemma}
 
    \subsection{The stress-energy tensor}
    The stress-energy tensor is an important tool in the analysis of singularly perturbed gradient-type problems. In dimension two, its expression is simplified thanks to complex analysis.  
      \begin{lemma}
      \label{canardwc}
Let $u_\eps$ be a solution of \eqref{elipes} on $\Omega$. Given any vector field $\vec X \in \mathcal{D}(\Omega,\R^2)$ we have
\begin{equation}
\label{canardwc1}
\int_{\Omega}   {A_\eps(u_\eps)}_{i, j}
 \cdot
\frac{\partial X_i}{\partial x_j}\,dx =0 {\rm \ where \ }
A_\eps(u_\eps)=e_\eps(u_\eps)\delta_{ij}-\eps
\frac{\partial u_\eps}{\partial x_i} \cdot\frac{\partial u_\eps}{\partial x_j}.
\end{equation}
  \end{lemma}
 The proof is standard (see \cite{BOS2}
 and references therein): It is derived multiplying the equation \eqref{elipes} by the function $\displaystyle{v=\sum X_i 
 \partial _i u_\eps}$. The $2 \times 2$ stress-energy matrix $A_\eps$ may  be decomposed as
  \begin{equation}
  \label{matrixaeps}
 A_\eps\equiv A_\eps(u_\eps)= T_\eps(u_\eps)
  + \frac{V(u_\eps)}{\eps}  \,\text{I}_2\, ,
  \end{equation}
  where the matrix $T_\eps(u)$ is defined, for a map $u: \Omega \to \R^2$, by
  \begin{equation}\label{stresstensor}
  T_\eps(u)=\frac{\eps}{2} \left(
  \begin{array}{cc}
   |u_{x_2}|^2-|u_{x_1}|^2  &  -2u_{x_1}\cdot u_{x_2} \\
    -2u_{x_1}\cdot u_{x_2}  & |u_{x_1}|^2-|u_{x_2}|^2 \\
  \end{array}
  \right).
\end{equation}

\begin{remark}
\label{scratch}
{\rm  Formula \eqref{canardwc} corresponds to  the first variation of the energy when one performs  deformations of the domain induced by the diffeomorphism related to the vector field $\vec X$. More precisely, it can be derived from the fact that 
$$\frac{d}{dt} \E_\eps (u_\eps \circ \Phi_t)=0,$$
where, for $t \in \R$  $\Phi_t: \Omega \to \omega$ is a diffeomorphism such that 
$$\frac{d}{dt} \Phi_t(x)=\vec X(\Phi_t(x)),  \forall x \in \Omega.
$$
}
\end{remark}

In dimension two, one may use complex notation to obtain a simpler expression of  $\displaystyle{T_{ij}\, \frac{\partial
X_i}{\partial x_j}}$.  Setting  $\displaystyle{X= X_1+ iX_2}$ we consider the complex function $\omega_\eps: \Omega \to \C$ defined by 
\begin{equation}
\label{hopfique}
\omega_\eps=\eps \left( \vert {u_\eps}_{_{x_1}}\vert^2-\vert {u_\eps}_{_{x_2}} \vert^2-2i
{u_\eps}_{_{x_1}}\cdot {u_\eps}_{_{x_2}}\right) , 
\end{equation}
the quantity $\omega_\eps$ being usually termed the \emph{Hopf
differential} of $u_\eps$.
 We obtain the identity
\begin{equation*}
T_{ij}(u_\eps)\frac{\partial X_i}{\partial x_j} =
\mathrm{Re}\left( - \omega_\eps \frac{\partial X}{\partial
\bar{z} }\right).
\end{equation*}
 Identity  \eqref{canardwc1} is turned into
 \begin{equation}
 \label{canardwc2}
 \int_{\Omega}\mathrm{Re}\left(  \omega_\eps \frac{\partial X}{\partial
\bar{z} }\right)=\frac{4}{\eps} \int_{\Omega}  {V(u_\eps)}\, \mathrm {Re} \left(\frac{\partial X}{\partial
{z} }\right)=\frac{2}{\eps} \int_{\Omega}  {V(u_\eps)}\, \mathrm {div} \, \vec X.
 \end{equation}
  
  \begin{remark}
  \label{scratch2}  {\rm  Recall that the Dirichlet energy is invariant by conformal transformation. Such transformation are locally obtained through vector-fields $\vec X$ which are holomorphic.  
  
  }
  \end{remark}
  
 \subsection{Pohozaev's identity on disks}     
  Identity \eqref{canardwc2} allows to derive integral estimates of the potential $V(u_\eps)$ using a suitable choice of test vector fields. We restrict ourselves to the special case the domain is  $\Omega=\D^2(r)$, for some $r>0$. We notice that for the vector field  $X=z$, we have
   $$
   \frac{\partial X}{\partial\bar{z}}=0   {
   \rm \ and \ }
\frac{\partial X}{\partial{z}}=1.
$$
However $X=z$ is  not a test vector field, since in does not have compact support, so that we consider instead vector fields $X_\delta$ of the form 
$$X_\delta=z \varphi_\delta ({\vert z \vert })),$$
 where $0<\delta<\frac 12$ is a small parameter  and $\varphi_\delta$ is a scalar function defined on $[0, r]$ such that  
\begin{equation}
\varphi_\delta (s)=1 {\rm \ for \ } s \in [0, r-\delta) ,\ \,  \vert \varphi' (s)\vert \leq 2 \delta {\rm \ for \ } s\in [ r-\delta, r]{\rm \ and \ }  \varphi(r)=0. 
\end{equation}
  A rapid computation shows  that, dropping the subscript $\eps$ and writing $u=u_\eps$ 
  \begin{equation*}
 \frac{\partial X_\delta}{\partial
\bar{z}}=0  {\rm \ on \ }    \D^2( r-\delta) {\rm \ and  \ }
{\mathrm {Re}}\left( \omega  \frac{\partial X_\delta}{\partial\bar{z}}\right)
= \left( \vert u_r\vert^2-r^2 \vert u_\theta \vert^2\right) \vert z \vert \varphi'_\delta (\vert z \vert) {\rm \ on \ } 
\D^2 (r) \setminus \D^2(r-\updelta), 
  \end{equation*}   
   whereas
    \begin{equation*}
   \frac{\partial X_\delta}{\partial
{z}}=1   {\rm \ on \ }    \D^2( r-\delta)   {\rm \ and  \ } 
{\mathrm {Re}}\left( \frac{\partial X_\delta}{\partial{z}}\right)
=\frac{1}{2} \vert z \vert \varphi'(\vert z \vert) {\rm \ on \ } 
  \D^2 (r) \setminus \D^2(r-\updelta).
\end{equation*}
    Inserting these relations into \eqref{canardwc2} and passing to the limit $\delta\to 0$ yields the following identity, termed Pohozaev's identity:
   
   \begin{lemma}
    \label{poho}     Let $u_\eps$  be  a solution of \eqref{elipes} on $\D^2$. We have, for any  radius $0<r \leq 1$      
    \begin{equation}
   \label{poho1}
  \frac{1}{\eps^2} \int_{\D^2( r)} V(u_\eps) =\frac r 4\int_{\partial \D^2( r) } 
  \left(
  \left\vert \frac{\partial  u_\eps}{\partial \tau}\right \vert^2 -\left\vert \frac{\partial  u_\eps}{\partial r} \right\vert^2
  +\frac{2}{\eps^2}V(u_\eps)
  \right) {\rm d} \tau. 
    \end{equation}
      \end{lemma} 
         This identity has the remarkable property that it yields an identity   of the integral of the potential \emph{inside}   the disk involving only    energy terms on  the  \emph{boundary}. A straightforward consequence of Lemma \ref{poho} is  the estimate: 
        
         \begin{proposition}
     \label{pascap0}
     Let $u_\eps$  be  a solution of \eqref{elipes} on $\D^2$. We have, for any $0<r\leq 1$ 
     \begin{equation}
     \label{pascap}
      \frac{1}{\eps} \int_{\D^2( r)} V(u_\eps) \leq  \frac r 2 \int_{\S^1(r)} e_\eps(u_\eps) {\rm d}\ell.
     \end{equation}
     \end{proposition}  
     Proposition \ref{pascap0} follows immediately from Lemma \ref{poho} noticing that the absolute value of the integrand on the left hand  side is bounded by $2\eps^{-1}e_\eps(u_\eps)$.
     
\smallskip
Besides Proposition \ref{pascap0}, we notice      that     Pohozaev's identity leads directly to remarkable  consequences:  For instance,  all solutions which are constant with values in $\Sigma$ on $\D^2(r)$ are necessarily constant. 
      \begin{remark}{\rm 
        The previous  results are  specific to dimension $2$, however the use of the stress-energy tensor yields other results in higher dimensions (for instance monotonicity formulas).   
    }
    \end{remark}
   \begin{remark}
   {\rm Identity \eqref{poho1} leads to the monotonicity formula
   \begin{equation}
   \label{monantoine}
   \frac{d}{dr}\left( \frac{\Eps\left(u_\eps, \D^2(r)\right)}{r}\right)=\frac{1}{r^2}\int_{\D^2(r)} \xi_\eps(u_\eps){\rm d} x  +
   \frac{1}{r} \int_{S^1(r)}\vert \frac{\partial u_\eps}{\partial r} \vert ^2 {\rm d} \ell, 
   \end{equation}
   where  the discrepancy 
   $
   \xi_\eps(u_\eps)$
   is defined in \eqref{discretpanse}.
   }
   \end{remark}
   
 \subsection{Pohozaev's type inequalities on general subdomain}
  We present in this subsection a related tool which will be of interest in the proof of Theorem \ref{bordurer}. We consider a solution $u_\eps$ of \eqref{elipes} on a general domain $\Omega$, a  subdomain $\mathcal U$ of $\Omega$ and for $\updelta >0$ the domain 
  $\mathcal U_\updelta$ introduced in \eqref{Udelta}.  As  a variant of Proposition \ref{pascap0}, we have:   
     
          \begin{proposition}
     \label{pascapit}
     Let $u_\eps$  be  a solution of \eqref{elipes} on $\Omega$. We have, for any $0<\updelta $ 
     \begin{equation}
     \label{pascap36}
      \frac{1}{\eps} \int_{\mathcal U_{\frac{\updelta}{2}} }V(u_\eps) {\rm d} x \leq  C(\mathcal U, \updelta)  \int_{\mathcal V_\updelta} e_\eps(u_\eps) {\rm d}x,
     \end{equation}
      where the constant $C(\mathcal U, \updelta) >0$ depends on $\mathcal U$, $\updelta$ and $V$.
     \end{proposition}
     
     The main difference with Proposition \ref{pascap0} is that, in the case of a disk,  the form of the $C(\mathcal U, \updelta) >0$  is determined more accurately. 
\begin{proof} [Proof of Proposition \ref{pascapit}] Turning back to identity \eqref{canardwc2},  we choose once more a  test vector field $\vec X_\updelta $ of the form $X_\updelta(z)=z \varphi_\delta(z)$,  where the function $\varphi_\delta$ is a  smooth scalar positive function  such that 
$$
\varphi(z)=1 {\rm \ for \ } z \in {\mathcal U}_{\frac{\updelta}{2}}  {\rm \ and \ } \varphi(z)=0 {\rm \ for \ } z \in \R^2 \setminus,  
{\mathcal U}_{\updelta}
$$
 so that $\nabla \varphi_\updelta=0$ on the set ${\mathcal U}_{\frac{\updelta}{2}}$ and hence 
 \begin{equation*}
 \frac{\partial X_\delta}{\partial
\bar{z}}=0  
{\rm \ and \ } 
\frac{\partial X_\delta}{\partial
{z}}=1 \,  
 {\rm \ on \ } {\mathcal U}_{\frac{\updelta}{2}}.
 \end{equation*}   
Inserting these relations into \eqref{canardwc2}, we are led  to inequality \eqref{pascap36}.

\end{proof}
     

%
\subsection{  General energy estimates on level set }
  We consider again for given $0<\eps \leq 1$ a solution  $u_\eps: \D^2 \to \R^k$  to \eqref{elipes}.  We assume that we  are given a radius $\varrho_\eps \in [\frac 12, \frac 34]$, a number 
  $\displaystyle{0<\upkappa <\frac{\upmu_0}{4}}$   and an element $\upsigmam \in \Sigma$ such that 
   \begin{equation}
  \label{kappacite}
   \vert u_\eps-\upsigmam \vert  < \upkappa  {\rm \ on \ } \partial \D^2(\varrho_\eps).
  \end{equation}
 We introduce the subdomain $\Upsilon_\eps(\varrho_\eps, \upkappa)$ defined  by
   \begin{equation}
\label{lilu}
\begin{aligned}
\Upsilon_\eps(\varrho_\eps, \upkappa) &=\left\{ x \in \D^2 ( \varrho_\eps ) {\rm \ such \ that \  \ }   \vert  u_\eps(x) -\upsigma_i \vert  <\upkappa, \ \ 
{\rm\ for  \ some  \ } i=1\ldots q \ \right\} \\
&=\underset {i=1}{\overset {q} {\large  \cup}}   \Upsilon_{\eps, i}  (\varrho_\eps, \upkappa), 
\end{aligned}
\end{equation}
 where  we have set 
 $$
  \Upsilon_{\eps, i}(\varrho_\eps, \upkappa)=w_i^{-1} ([0,   \upkappa) \cap  
\D^2 ( \varrho_\eps )=\Upsilon_{\upsigma_i}(u_\eps,  \varrho_\eps,  \upkappa)=
\{x\in \D^2 ( \varrho_\eps ), \vert u_\eps-\upsigma_i \vert  \leq   \upkappa \}.
 $$
 The set $\Upsilon_{\upsigma}(u,  \varrho,  \upkappa)$ has alreday been introduced in \eqref{subsub}.
The set $\Upsilon_\eps(\varrho_\eps, \upkappa)$  corresponds  to a truncation of the domain 
  $\D^2(\varrho_\eps)$ where points with values far from the set $\Sigma$  have been removed. 
   By construction, the solution $u_\eps$ is close,  on $\Upsilon_\eps(\varrho, \upkappa)$,  to one of the points $\upsigma_i$ in $\Sigma$:  Near this point the potential is convex, close to a quadratic potential. 
  The main result of the present  section is to establish  an estimate on the integral of the energy on the domain $\Upsilon_\eps(\varrho_\eps, \upkappa)$ in terms of  the integral of the potential  as well as boundary integrals.

  \begin{proposition}
  \label{kappacity}
  Let $u_\eps$ be a solution of \eqref{elipes} on $\D^2$ and 
  assume that \eqref{kappacite} is satisfied.  We have, for some constant ${\rm C}_\Upsilon >0$,  depending only on the potential  $V$, 
  \begin{equation}
  \label{nonodes}
   \int_{\Upsilon_\eps (\varrho_\eps, \upkappa)} e_\eps(u_\eps)(x) {\rm d}x \leq  {\rm C_\Upsilon} \left[ 
 \upkappa  \int_{\D^2(\varrho_\eps)} \frac{V(u_\eps)}{\eps} {\rm d} x  +\eps    \int_{\partial \D^2(\varrho_\eps)} e_\eps(u_\eps){\rm d} \ell
     \right].
   \end{equation}
  \end{proposition}
  
 Of major  importance in estimate \eqref{nonodes} is the presence of the term $\upkappa$ in front of the integral of the  potential, so that the energy on 
 $\Upsilon_\eps (\varrho_\eps, \upkappa)$ grows essentially at most linearly with respect to $\upkappa$.
 
 \smallskip
 The proof of Proposition \eqref{kappacity}  relies on a multiplication of the equation \eqref{elipes} by the solution $u_\eps$ itself,  and then integration by parts on appropriate sets. In order to simplify the presentation,  we  divide  the proof of Proposition \ref{kappacity} into several  intermediate results.

 \medskip
  Concerning first the behavior of $u_\eps$ on  the boundary 
  $\partial \D^2 ( \varrho_\eps)$, we may assume without loss of generality that $\upsigmam=\upsigma_1$, so that it follows from   assumption \eqref{bornuni3} that 
\begin{equation}
\label{versailles}
\vert u_\eps(\ell)-\upsigma_1\vert < \upkappa  {\rm \  for \ }  \ell \in \partial \D^2 ( \varrho_\eps). 
\end{equation}
We may also assume, since $u_\eps$ is smooth and in view of Sard's Lemma, that the boundary $\partial \Upsilon_\eps (\varrho_\eps, \upkappa)$ is a finite union of smooth curves. 
We deduce from  inequality  \eqref{versailles} that  
$\displaystyle { \partial \D^2 ( \varrho_\eps)  \subset \overline{\Upsilon_{\eps, 1}(\varrho_\eps, \upkappa)}}$, 
and that, for 
   $i=2, \ldots, q$, we have 
   $$\partial \D^2 ( \varrho_\eps) \cap  \partial   \Upsilon_{\eps, i} (\varrho_\eps, \upkappa)=\emptyset.$$ 
    Hence, for $i=2, \ldots, q$
 the set 	$\partial \Upsilon_{\eps, i}$ is an union of smooth curves intersecting the boundary $\partial \D^2(\varrho_\eps)$ transversally. We define the curves $\Gamma_\eps^i$ as  
   so that 
  \begin{equation}
  \left\{
  \begin{aligned}
\Gamma_\eps^i(\varrho_\eps, \upkappa) & \equiv  \partial   \Upsilon_{\eps, i}  (\varrho_\eps, \upkappa)= w_i^{-1}(\upkappa)\cap  \D^2 ( \varrho_\eps)  
 \  {\rm \ for \ } i=2 \ldots   q,  \\
\Gamma_\eps^1 (\varrho_\eps, \upkappa)&\equiv  \partial   \Upsilon_{\eps, 1}  (\varrho_\eps, \upkappa)\setminus \partial \D^2( \varrho_\eps) =
\left( \left[w_1^{-1}(\upkappa)\cap  \D^2 ( \varrho_\eps) \right] \right)\setminus \partial \D^2( \varrho_\eps).
  \end{aligned}
  \right. 
  \end{equation}										
    A first intermediate step in the proof of  Proposition \ref{kappacity}  is: 
  
    \begin{lemma} 
  \label{eliot}
  Assume that $0<\eps \leq 1$  and  that $u_\eps$ is a solution to \eqref{elipes} satisfies  assumption \eqref{kappacite}.  Then  we have 
  \begin{equation}
  \label{klisko}
  \int_{\Upsilon_\eps(\varrho_\eps, \upkappa)} e_\eps(u_\eps)  {\rm d}x \leq   
  C\left[ \upkappa  \,  \eps \, 
  \underset{i=1} {\overset {q}\sum}
   \int_{\Gamma_{\eps, i}} \frac{\partial \vert u_\eps(\ell)-\upsigma_i \vert}{\partial \vec n (\ell)} {\rm d}\ell+ 
   \eps  \int_{\partial \D^2(\varrho_\eps)} e_\eps(u_\eps(\ell)){\rm d}\ell
   \right], 
    \end{equation}
    where $C>0$ is some constant depending only on the potential $V$  and where $\vec n(\ell)$ denotes the unit vector normal to  $\Gamma_{\eps, i}$ pointing in the direction increasing $\vert u_\eps-\upsigma_i \vert $. 
      \end{lemma}

\begin{proof}    
 For $i=1, \ldots, q$, we multiply equation  \eqref{elipes}  by $u_\eps-\upsigma_i$ and integrate by parts on the domain 
  $ \Upsilon_{\eps, i} (\varrho_\eps, \kappa)$. This yields, for $i=2, \ldots, q$
\begin{equation}
  \label{stokounette0}
  \begin{aligned}
  \int_{ \Upsilon_{\eps, i} (\varrho_\eps, \upkappa)}\eps \vert \nabla u_\eps \vert^2 +\eps^{-1} \nabla_u V(u_\eps)\cdot (u_\eps-\upsigma_i)
  &=\int_{ \Gamma_{\eps, i}(\varrho_\eps, \upkappa)} \eps  \frac{\partial u_\eps}{\partial \vec n} \cdot (u_\eps-\upsigma_i)\\
  &=\frac{\eps}{2}\int_{ \Gamma_{\eps, i}(\varrho_\eps, \upkappa)}  \frac{\partial \vert u_\eps-\upsigma_i\vert^2}{\partial \vec n}.
  \end{aligned}
   \end{equation}
   Since, by the definition of $\Upsilon_{\eps, i}$,  we have $\displaystyle{\vert u-\upsigma_i  \vert \leq \upkappa \leq \frac{\upmu_0}{2}}$,   we  are  in position to invoke estimates \eqref{glutz},  which yields, for $i \in \{1, \ldots, q\}$,
\begin{equation}
\label{duglandmoins}
 \frac {\lambda_0}{2\lambda_{\rm max}}V(u_\eps) \leq \frac{1}{2}\lambda_0\vert u_\eps -\upsigma_i \vert^2  \leq  \nabla V(u_\eps) \cdot (u_\eps-\upsigma_i)
{\rm \ on \ }  \Upsilon_{\eps, i} (\varrho_\eps, \kappa), 
\end{equation}
where  $\displaystyle{\lambda_{\rm max}=\sup \{\lambda_i^+, i=1, \ldots, q_i\}}$.
Going back  to \eqref{stokounette0}, which we multiply by $\displaystyle{\frac{2\lambda_{\rm max}}{\lambda_0}\geq 2}$,  we deduce that 
\begin{equation}
\label{dugoin}
 \int_{ \Upsilon_{\eps, i}(\varrho_\eps, \upkappa) }e_\eps(u_\eps) {\rm d} x \leq  \frac{2\lambda_{\rm max}}{\lambda_0} 
  \int_{\Gamma_{\eps, i}(\varrho_\eps, \upkappa)} \frac{\eps}{2} \frac{\partial \vert u_\eps-\upsigma_i \vert^2}{\partial \vec n}.
\end{equation}
  Arguing similarly in the case $i=1$, we obtain
  \begin{equation}
  \label{dugoin2}
   \int_{ \Upsilon_{\eps, 1} }\eps \vert \nabla u_\eps \vert^2 +\eps^{-1} \nabla_u V(u_\eps)\cdot (u_\eps-\upsigma_i)
  =\int_{ \Gamma_{\eps, 1}} \eps \frac{\partial u_\eps}{\partial \vec n} \cdot (u_\eps-\upsigma_i)+
  \int_{\partial \D^2(\varrho_\eps)} \eps \frac{\partial u_\eps}{\partial r} \cdot (u_\eps-\upsigma_i),
   \end{equation}
    and hence
   \begin{equation}
\label{dugoin3}
\begin{aligned}
 \int_{ \Upsilon_{\eps, 1} }e_\eps(u_\eps) {\rm d} x &\leq  \frac{2\lambda_{\rm max}}{\lambda_0}  \left[
  \int_{\Gamma_{\eps, 1}}\frac {\eps}{2} \frac{\partial \vert u_\eps-\upsigma_i \vert^2}{\partial \vec n} 
  +
    \int_{\S^1(\varrho_\eps)} \eps \frac{\partial u_\eps}{\partial r} \cdot (u_\eps-\upsigmam) \right] \\
    &\leq  \frac{2\lambda_{\rm max}}{\lambda_0}  \left[
  \int_{\Gamma_{\eps, 1}}\frac {\eps}{2} \frac{\partial \vert u_\eps-\upsigma_i \vert^2}{\partial \vec n} 
  +  \frac{2}{\sqrt{\lambda_0}}
    \int_{\S^1(\varrho_\eps)} \eps J(u_\eps){\rm d} \ell \right] \\
    &\leq  C\left[
  \int_{\Gamma_{\eps, 1}}\upkappa  {\eps} \frac{\partial \vert u_\eps-\upsigma_i \vert}{\partial \vec n} 
  +
    \int_{\S^1(\varrho_\eps)} \eps e_\eps (u_\eps) {\rm d} x\right],
 \end{aligned}
\end{equation}
where we used Lemma \ref{ab0} for the last inequality. 
 Summing estimates \eqref{dugoin} for $i=2, \ldots, q$  together with  \eqref{dugoin3}, we are led to \eqref{klisko}. 
\end{proof}

  In order to deduce Proposition \ref{kappacity} from Lemma \ref{eliot}   we need one additional  ingredient.    
  
  \begin{lemma}  
  \label{repro}  Let  $\upkappa_1 \geq \upkappa_0 \geq \upkappa$. If $u_\eps$ satisfies condition \eqref{kappacity},  then we have, for $i=1, \ldots, q$,   the inequality
   \begin{equation}
     \label{sting}
 0\leq  \int_{\Gamma_{\eps, i}(\varrho_\eps, \upkappa_0)} \frac{\partial \vert u_\eps(\ell)-\upsigma_i \vert }{\partial \vec n(\ell)} {\rm d} \ell \leq 
   \int_{\Gamma_{\eps, i}(\varrho_\eps, \upkappa_1)}  \frac{\partial \vert u_\eps(\ell)-\upsigma_i \vert }{\partial \vec n (\ell)}{\rm d} \ell .
\end{equation}
  \end{lemma}

  \begin{proof}  
  The proof involves again Stokes formula, now on the domain 
  $$
  \mathcal C(\upkappa_0, \upkappa_1)=\Upsilon_{\eps, i} (\varrho_\eps, \upkappa_1)\setminus  
 \Upsilon_{\eps, i} (\varrho_\eps, \upkappa_0).
  $$
 It follows from assumption \eqref{kappacity} that  
 $$
 \overline{  \mathcal C(\upkappa_0, \upkappa_1)} \cap \partial \D^2(\varrho_\eps)=\emptyset, 
$$
 so that 
 $$\partial \mathcal C(\upkappa_0, \upkappa_1)=  \partial  \Upsilon_{\eps, i} (\varrho_\eps, \upkappa_1)\cup
\partial  \Upsilon_{\eps, i} (\varrho_\eps, \upkappa_0).
$$
We multiply the equation \eqref{elipes}  by $\displaystyle{ \frac{u_\eps-\upsigma_i} {\vert u_\eps-\upsigma_i \vert }}$ which is well defined   on $  \mathcal C(\upkappa_0, \upkappa_1)$ and integrate by parts.  Since, on $\Gamma_{\eps, i}(\varrho_\eps, \upkappa)$, we have 
$$
\frac{\partial u_\eps}{\partial  \vec n} \cdot \frac{u_\eps-\upsigma_i} {\vert u_\eps-\upsigma_i \vert }=\frac{\partial (u_\eps-\upsigma_i)}{\partial  \vec  n} \cdot \frac{u_\eps-\upsigma_i} {\vert u_\eps-\upsigma_i \vert }= \frac{\partial \vert u_\eps-\upsigma_i \vert }{\partial \vec n}, 
$$
whereas on  $\mathcal C(\upkappa_0, \upkappa_1)$, we have 
\begin{equation*}
\begin{aligned}
\nabla u_\eps\cdot \nabla \left( \frac{u_\eps-\upsigma_i} {\vert u_\eps-\upsigma_i \vert }\right)&=
\nabla( u_\eps-\upsigma_i) \cdot \nabla\left (\frac{u_\eps-\upsigma_i} {\vert u_\eps-\upsigma_i \vert }\right) \\
&=\frac{1}{\vert u_\eps-\upsigma_i \vert} \vert \nabla( u_\eps-\upsigma_i)  \vert^2 +
\left[\nabla( u_\eps-\upsigma_i) \cdot( u_\eps-\upsigma_i)\right] \cdot \nabla (\frac{1}{\vert u_\eps-\upsigma_i\vert})\\
&=\frac{1}{\vert u_\eps-\upsigma_i \vert} \left[\vert \nabla( u_\eps-\upsigma_i)  \vert^2-\left \vert \nabla \vert u_\eps-\upsigma_i \vert \right \vert ^2\right]
\end{aligned}
\end{equation*}
integration by parts  thus   yields
 \begin{equation}
  \label{stokounette}
  \begin{aligned}
  \int_{\Gamma_{\eps, i}(\varrho_\eps, \upkappa_1)} \frac{\partial \vert u_\eps-\upsigma_i \vert }{\partial \vec n}-
   \int_{\Gamma_{\eps, i}(\varrho_\eps, \upkappa_0)}  \frac{\partial \vert u_\eps-\upsigma_i \vert }{\partial \vec n} 
  &=
  \int_{ \mathcal C(\upkappa_0, \upkappa_1)} \frac{1}{\vert u-\upsigma_i \vert  }\left[\vert \nabla u_\eps \vert^2  -  
   \left \vert \nabla \vert u_\eps-\upsigma_i \vert \right \vert ^2\right]\\
+  & \int_{\mathcal C(\upkappa_0, \upkappa_1)} \eps^{-2} \nabla_u V(u_\eps)\cdot \frac{(u_\eps-\upsigma_i)}{\vert u-\upsigma_i \vert}.
  \end{aligned}
  \end{equation}
 Since 
 $$\displaystyle{\vert \nabla u_\eps \vert^2  -  
   \left \vert \nabla \vert u_\eps-\upsigma_i \vert \right \vert ^2}=\vert \nabla (u_\eps-\upsigma_i) \vert^2  -  
   \left \vert \nabla \vert u_\eps-\upsigma_i \vert \right \vert ^2\geq 0,$$  
   it follows that the r.h.s of inequality \eqref{stokounette} is positive.   Hence, we deduce \eqref{sting}.

  \end{proof}
  
   \begin{lemma} Assume that $0<\eps \leq 1$  and  that $u_\eps$ is a solution to \eqref{elipes} which  satisfies  \eqref{kappacite}. Then, there exits a constant $\rm C>0$ depending only on $V$ such that have 
  \label{denada}
  \begin{equation}
  \label{denada1}
     0\leq  \eps \int_{\Gamma_{\eps, i}(\varrho_\eps, \upkappa)} \frac{\partial \vert u_\eps-\upsigma_i \vert }{\partial \vec n(\ell)} {\rm d } \ell 
   \leq  
   C \int_{\D^2(\varrho_\eps)}\frac{V(u)}{\eps} {\rm d} x   \leq C\mathbb V(u_\eps, \D^2(\frac34)),  
     \end{equation}
     where, for a point $\ell \in \Gamma_\eps$, $\vec n (\ell)$  denotes the unit vector  perpendicular to $\Gamma_\eps$  and oriented in the direction  which  increases $\vert u-\upsigma_i \vert$. 
  \end{lemma}
  \begin{proof}
  We invoke first   Lemma \ref{claudius} with the choices $r=\varrho_\eps$ and  $u=u_\eps$.  This yields  a number 
  $\displaystyle{\tilde \upmu_\eps \in [\frac{\upmu_0}{4}, \frac{\upmu_0}{2}]}$ such that 
  $$
 \eps  \int_{\Gamma_{\eps, i}(\varrho, \tilde \upmu_\eps)} \frac{\partial \vert u_\eps-\upsigma_i \vert }{\partial \vec n(\ell)} {\rm d } \ell 
 \leq \eps \int_{\Gamma_{\eps, i}(\varrho, \tilde \upmu_\eps)} \vert \nabla u_\eps \vert  {\rm d } \ell
  \leq  \frac{1}{\upmu_0 } \E_\eps(u, \Theta (u_\eps,  \varrho_\eps))
    $$
where  $\Theta (u_\eps,  \varrho_\eps)$ is defined in \eqref{nablalala}. On the level set $\Theta (u_\eps,  \varrho_\eps)$, we may however bound point-wise the energy in terms of the potential, as stated in Lemma \ref{bornepote}. This yields by integration
  $$
 \E_\eps(u, \Theta (u_\eps, \varrho_\eps)) \leq {\rm C_T}  \mathbb V(u_\eps, \Theta  (u_\eps, \varrho_\eps) ). 
 $$
Combining the two previous inequalities, we obtain
 \begin{equation}
 \label{eau}
  \eps  \int_{\Gamma_{\eps, i}(\varrho, \tilde \upmu_\eps)} \frac{\partial \vert u_\eps-\upsigma_i \vert }{\partial \vec n(\ell)} {\rm d } \ell 
\leq \frac{\rm C_T} {\upmu_0}  \mathbb V(u_\eps, \Theta_\eps (u, \mathfrak  r_\eps )).
 \end{equation}
 On the other hand, we invoke  to Lemma \ref{repro} with the  $\upkappa_1=\tilde \upmu_\eps$ and $\upkappa_0=\upkappa$ to deduce  that
$$
 \int_{\Gamma_{\eps, i}(\varrho,  \upkappa)} \frac{\partial \vert u_\eps-\upsigma_i \vert }{\partial \vec n(\ell)} {\rm d } \ell  \leq \int_{\Gamma_{\eps, i}(\varrho, \tilde \upmu_\eps)} \frac{\partial \vert u_\eps-\upsigma_i \vert }{\partial \vec n(\ell)} {\rm d } \ell, 
 $$
 which together with \eqref{eau} leads  to   the desired result \eqref{denada1}. 
  \end{proof}
  
  \begin{proof}[Proof of Proposition \ref{kappacity} completed]   Combining \eqref{klisko} with \eqref{denada1}, we derive 
 the  desired inequality  \eqref{nonodes}.
      \end{proof}
      
      \subsection{Bounding  the total energy by the  integral of the potential}
      The main result of the present paragraph is the following result:
      
 \begin{proposition}
 \label{borneo}
  Let $u_\eps$ be a solution of \eqref{elipes} on $\D^2$ and $M_0>0$ be given. There exists  a constants $K_{\rm pot}(M_0)>0$,   depending   possibly on $V$ and $M_0$, and  a constant ${\rm C}_{\rm pot}$,  depending only on the potential  $V$, such that, if we have  
 \begin{equation}
 \label{lajoie}
 \E(u_\eps) \leq M_0 {\rm \ and \  }
 \eps^{-1} \int_{\D^2(\frac 34)} V(u_\eps) \leq K_{\rm pot}(M_0),   
 \end{equation}
 then  we have the estimate
 \begin{equation}
 \label{dupont}
 \int_{\D^2(\frac 12)} e_\eps(u_\eps)(x) {\rm d}x \leq  {\rm C_{\rm pot}} \left[ 
   \int_{\D^2(\frac34)} \frac{V(u_\eps)}{\eps} {\rm d} x  +\eps    \int_{ \D^2 \setminus \D^2(\frac 12) } e_\eps(u_\eps){\rm d} x. 
     \right].
 \end{equation}
 \end{proposition}
 
 In  the context of the present paper, the main contribution of the r.h.s  of inequality \eqref{dupont}  is given by the potential terms, so that Proposition \ref{borneo}  yields an estimate of the energy by the integral of  potential, provided  the later is sufficiently small, according to assumption \eqref{lajoie}.

 \smallskip
Before turning to the proof of Proposition \ref{borneo}, we observe, as a preliminary remark, that the result of proposition \ref{borneo} is, at first sight,  rather close to the result of Proposition  \ref{kappacity}.  However, let us emphasize thar  estimate  \eqref{nonodes} yields an energy bound  only  for the domain where the  value of \emph{$u_\eps$ is close to one of the wells}. on the other hand, estimate \eqref{nonodes} presents also some improvement compared to \eqref{dupont}, since it involves an  additional factor $\upkappa$,  measuring the distance to the well.  

\smallskip
Starting from Proposition \ref{kappacity}, a first step in the proof of Proposition \ref{borneo} is to  deduce \emph{ global estimates}, that means on the whole domain,  using  Proposition \ref{bornepote}. Indeed,  if we  choose the constant $\upkappa$ in the statement of Proposition \ref{bornepote} so that 
$\displaystyle{\upkappa=\frac{\upmu_0}{4}}$, then  assumption \eqref{kappacite}  is turned into 
 \begin{equation}
  \label{kappaciting}
   \vert u_\eps-\upsigmam \vert  < \frac{\upmu_0}{4} {\rm \ on \ } \partial \D^2(\varrho_\eps), 
  \end{equation}
   for some element $\upsigmam \in \Sigma$ and some radius $\varrho_\eps   \in [\frac 12, \frac 34].$  We then have the following:

\begin{proposition}
\label{bornepoting}
Let $\varrho_\eps \in [\frac 12, \frac34]$ and  let $u_\eps$ be a solution of \eqref{elipes} on $\D^2$ and 
  assume that \eqref{kappaciting} is satisfied.  We have, for some constant $\rm C_{\rm pot} >0$ depending only on the potential  $V$
    \begin{equation}
  \label{nonodesing}
   \int_{\D^2(\varrho_\eps)} e_\eps(u_\eps)(x) {\rm d}x \leq  {\rm C_{\rm pot}} \left[ 
   \int_{\D^2(\varrho_\eps)} \frac{V(u_\eps)}{\eps} {\rm d} x  +\frac{\eps}{4}    \int_{\partial \D^2(\varrho_\eps)} e_\eps(u_\eps){\rm d} \ell
     \right].
   \end{equation}
  \end{proposition}

\begin{proof}  We observe first  that 
\begin{equation}
\label{fruc}
\D^2(\varrho_\eps)=\Theta(\varrho_\eps) \cup \Upsilon_\eps( \varrho_\eps, \frac{\upmu_0}{4}).
\end{equation}
In view  of Lemma \ref{bornepote},  we have
$$
 \int_{\Theta(\varrho_\eps)} e_\eps(u_\eps){\rm d} x  \leq {\rm C}_{\rm T} \int_{\Theta(\varrho_\eps)} \frac{V(u_\eps)}{\eps} 
 {\rm d}x, 
$$
whereas Proposition \ref{kappacite} yields
 \begin{equation*}
    \int_{\Upsilon_\eps (\varrho_\eps, \frac{\upmu_0}{4})} e_\eps(u_\eps)(x) {\rm d}x \leq  {\rm C_\Upsilon} \left[ 
\frac{ \upmu_0 }{4} \int_{\D^2(\varrho_\eps)} \frac{V(u_\eps)}{\eps} {\rm d} x  +\eps    \int_{\partial \D^2(\varrho_\eps)} e_\eps(u_\eps){\rm d} \ell
     \right].
   \end{equation*}
The proof  of \eqref{nonodesing}  then follows straightforwardly from our first observation  \eqref{fruc}. 
\end{proof}

 \begin{proof}[Proof of Proposition \ref{borneo} completed] Inequality \eqref{dupont} is for a large part a rather direct consequence of Proposition \ref{bornepoting}, the main point being  a suitable choice of the radius $\varrho_\eps$, so that condition \eqref{kappaciting} can be deduced from condition \eqref{lajoie}. As usual,  a mean-value argument allows us to 
choose  some radius $\displaystyle{\varrho_\eps \in [\frac 12, \frac 34]}$  such that 
\begin{equation}
\label{zebulon}
\left\{
\begin{aligned} 
\int_{\partial \D^2(\varrho_\eps)}   V(u_\eps){\rm d} \ell & \leq 8\int_{ \D^2(\frac 34) \setminus \D^2(\frac 12) }  V(u_\eps){\rm d} x 
{\rm \   \  and \ } \\
\int_{\partial \D^2(\varrho_\eps)}   e_\eps(u_\eps) {\rm d} \ell &\leq 
8\int_{\D^2(\frac 34) \setminus \D^2(\frac 12)}  e_\eps(u_\eps) {\rm d} x.
\end{aligned}
\right.
\end{equation}
 It follows from inequality \eqref{ab} that 
\begin{equation}
\label{zebulon2}
\int_{\partial \D^2(\varrho_\eps)}  J(u_\eps) {\rm d} \ell\leq
 \left(\int_{\partial \D^2(\varrho_\eps)}  \eps^{-1} V(u_\eps){\rm d} \ell\right)^{\frac 12}
 \left(\int_{\partial \D^2(\varrho_\eps)}   e_\eps(u_\eps){\rm d} \ell \right)^{\frac 12}.
\end{equation}
We assume next that the bound \eqref{lajoie} holds, for some constant $K_{\rm pot}(M_0)$ to be determined later.  Inequalities \eqref{zebulon} and \eqref{zebulon2} yield
$$
\int_{\partial \D^2(\varrho_\eps)} \left( J(u_\eps(\ell)) +\eps^{-1}V(u_\eps(\ell) \right)
{\rm d} \ell  \leq 8\left( M_0\, K_{\rm pot}(M_0)\right)^{\frac 12} + 8 K_{\rm pot}(M_0).
$$ 
Applying Lemma \ref{valli}, we deduce that there exists some element $\upsigmam \in \Sigma$ such that
\begin{equation}
      \label{born}
           \vert u(\ell)-\upsigmam \vert \leq  {\rm C}_{\rm unf}\sqrt{8\left( M_0\, K_{\rm pot}(M_0)\right)^{\frac 12} + 8 K_{\rm pot}(M_0)}. 
\end{equation} 
Next we choose the constant  $K_{\rm pot}(M_0)$ so small that 
$$
 {\rm C}_{\rm unf}\sqrt{8\left( M_0\, K_{\rm pot}(M_0)\right)^{\frac 12} + 8 K_{\rm pot}(M_0)} 
\leq \frac
{\upmu_0}{4}.
$$
For such a choice of the constant, we obtain, combining with \eqref{born}
\begin{equation}
      \label{borninf}
           \vert u(\ell)-\upsigmam \vert \leq  \frac{\upmu_0}{4}.
\end{equation}
Hence,  condition \eqref{kappaciting} is fullfilled on $\S^1 (\varrho_\eps)$ so that we are in position to apply Proposition \ref{bornepoting}, which yields 
 \begin{equation}
  \label{gaston}
  \begin{aligned}
   \int_{\D^2(\frac 12)} e_\eps(u_\eps)(x) {\rm d}x &\leq \int_{\D^2(\varrho_\eps)} e_\eps(u_\eps)(x) {\rm d}x  \\
 &  \leq  {\rm C_{\rm pot}} \left[ 
   \int_{\D^2(\varrho_\eps)} \frac{V(u_\eps)}{\eps} {\rm d} x  +\frac{\eps}{4}    \int_{\partial \D^2(\varrho_\eps)} e_\eps(u_\eps){\rm d} \ell
     \right] \\
     &\leq  {\rm C_{\rm pot}} \left[ 
   \int_{\D^2(\frac 34)} \frac{V(u_\eps)}{\eps} {\rm d} x +
  \eps    \int_{\D^2(\frac 34)} e_\eps(u_\eps){\rm d} \ell
     \right].
        \end{aligned}
   \end{equation}
The proof  of Proposition \ref{borneo} is complete. 
 \end{proof}
 
 \begin{remark}{\rm    In the course of the paper, we will invoke the  scaled version of Proposition \ref{borneo}. Given 
 $\varrho>\eps>0$ and $x_0 \in \Omega$, we consider a solution $u_\eps$ on $\Omega$ and assume it satisfies the bound
 \begin{equation}
 \label{lajoiebis}
 \E(u_\eps,\D^2(x_0, \varrho)) \leq M_0  \varrho {\rm \ and \  }
\eps^{-1} \int_{\D^2(x_0, \frac 34\varrho)} V(u_\eps) \leq K_{\rm pot}(M_0) \varrho.
 \end{equation}
Then,  thanks to the relations \eqref{scalingv}, we have the scaled version of \eqref{dupont}
 \begin{equation}
 \label{dupontbis}
 \int_{\D^2(x_0, \frac 12\varrho)} e_\eps(u_\eps)(x) {\rm d}x \leq  {\rm C_{\rm pot}} \left[ 
   \int_{\D^2(x_0,\frac34\varrho)} \frac{V(u_\eps)}{\eps} {\rm d} x  +\frac{\eps}{\varrho}    
   \int_{ \D^2(x_0, \varrho) \setminus \D^2(x_i, \frac 12\varrho) } e_\eps(u_\eps){\rm d} x 
     \right].
 \end{equation}
 }
 \end{remark}

\subsection{Bounds energy by integrals on    external domains}
      Our next result paves the way for the proof of Theorem \ref{bordurer}. As there, we consider a open subset $\mathcal U$ of $\Omega$ and define $\mathcal U_\delta$ and $\mathcal V_\delta$ according to \eqref{Udelta}. 

  \begin{proposition}
  \label{pave}
let $u_\eps$ be a solution of \eqref{elipes} on $\Omega$, $\mathcal U$ be an open bounded subset of $\Omega$ and 
$1>\delta>\eps >1>0$ be given such that $\mathcal U_\delta \subset \Omega$. Assume that 
 \begin{equation}
 \label{carpediem}
 \int_{\mathcal V_\updelta} 
      e_{\eps}(u_{\eps}) \, {\rm d} x  \leq {\rm  K}_{\rm ext} (\mathcal  U, \delta), 
 \end{equation}
  where  ${\rm  K}_{\rm ext} (\mathcal  U, \delta)>0$ denotes some constant depending possibly on $\mathcal U$ and $\delta$. 
 Then, we have the bound,   for some constant  ${\rm C}_{\rm ext} (\mathcal  U, \delta)$  depending possibly on $\mathcal U$ and $\delta$
\begin{equation}
\label{borges}
 \int_{\mathcal U_{\frac{\delta}{4}} }e_\eps (u_\eps) {\rm d} x \leq {\rm C}_{\rm ext} (\mathcal  U, \delta)
  \left(
  \int_{\mathcal V_\updelta} 
      e_{\eps}(u_{\eps})  + \eps \int_{\mathcal U_\delta} e_\eps(u_\eps) {\rm d} x.
\right)
 \end{equation}
 \end{proposition}
 \begin{proof}  The proof combines Proposition \ref{borneo}, Proposition \ref{pascapit} with a standard covering by disks. 
We first bound the potential on the set $\mathcal U_{\frac{\delta}{2}}$  thanks to  of Proposition \ref{pascapit}, which yields 
 \begin{equation}
     \label{pascap360}
      \frac{1}{\eps} \int_{\mathcal U_{\frac{\updelta}{2}} }V(u_\eps) {\rm d} x \leq  C(U, \updelta)  \int_{\mathcal V_\updelta} e_\eps(u_\eps) {\rm d}x\leq  C(U, \updelta)  {\rm  K}_{\rm ext} (\mathcal  U, \delta). 
     \end{equation}
   In inequality \eqref{pascap360}, we have  assumed  that the bound \eqref{carpediem} is fullfilled for some constant  ${\rm  K}_{\rm ext} (\mathcal  U, \delta)$, which we choose now as
   \begin{equation}
   \label{thechoice} 
  {\rm  K}_{\rm ext} (\mathcal  U, \delta) = \frac{{\rm  K}_{\rm pot}(M_0) \delta}{8C(\mathcal U, \delta)}.
   \end{equation}
    Inequality  \eqref{pascap360} then yields 
    \begin{equation}
    \label{pascap361}
     \frac{1}{\eps} \int_{\mathcal U_{\frac{\updelta}{2}} }V(u_\eps) {\rm d} x  \leq  \frac{\delta}{8} K_{\rm pot}(M_0). 
    \end{equation}
This bound will allow us  to apply inequality \eqref{lajoiebis} on disks of radius $\frac{\delta}{8}$ covering $\mathcal U_{\frac{\delta}{4}}$. In this direction, we   claim that there exists  a finite collections of disks 
  $\displaystyle{\left\{\D^2\left(x_i, \frac{\delta}{8}\right)\right\}_{i \in I}}$  such that
  \begin{equation}
  \label{rez}
  \mathcal U_{\frac{\delta}{4}} \subset  \underset{i \in I} \cup \D^2\left(x_i, \frac{\delta}{8}\right)  {\rm \ and \ } x_i \in  \overline{\mathcal U_{\frac{\delta}{4}}}, {\rm \ for \ any \ } i \in I.
  \end{equation}
   Indeed, such a collections may be obtained invoking the  collection of  disks $\displaystyle{\left\{D^2\left(x, \frac{\delta}{8}\right)\right\}}$  with $x \in \overline{\mathcal U_{\frac{\delta}{4}}}$ and then extracting a finite subcover thanks to Lebesgue's Theorem. Notice that we also have 
 \begin{equation}
 \label{chausson}
  \underset{i \in I} \cup \D^2\left(x_i, \frac{\delta}{4}\right)  \subset \mathcal U_{ \frac{\delta}{2}}.
 \end{equation}
 On each of  the disks $\D^2\left(x_i, \frac{\delta}{4}\right)$, we  have, thanks to  \eqref{pascap361}
 $$
  \frac{1}{\eps} \int_{\D^2(x_i, {\frac{\delta}{4}} )}V(u_\eps) {\rm d} x  \leq  \frac{\delta}{8} K_{\rm pot}(M_0),
  $$
  so that we may apply the scaled version \eqref{dupontbis} of Proposition \ref{borneo}  on the disk 
  $\D^2(x_i, \frac{1}{4} \delta)$: 
   This  yields  the estimate
   \begin{equation*}
 \int_{\D^2(x_i, \frac 18\delta )} e_\eps(u_\eps)(x) {\rm d}x \leq  {\rm C_{\rm pot}} \left[ 
   \int_{\D^2(x_i,\frac3{16}\delta )} \frac{V(u_\eps)}{\eps} {\rm d} x  +\frac{\eps}{\delta}    
   \int_{ \D^2(x_i, \frac{\delta}{4})} e_\eps(u_\eps){\rm d} x 
     \right].
 \end{equation*}
 Adding these relations for $i \in I$ and invoking relations \eqref{rez} and \eqref{chausson}
 we are led to
 \begin{equation}
  \int_{\mathcal U_{\frac{\delta}{4}}} e_\eps(u_\eps)(x) {\rm d}x \leq \sharp (I) 
   {\rm C_{\rm pot}} \left[ 
   \int_{\mathcal U_{\frac{\delta}{2}}} \frac{V(u_\eps)}{\eps} {\rm d} x  +\frac{\eps}{\delta}  
    \int_{\mathcal U_{\frac{\delta}{2}} } e_\eps(u_\eps){\rm d} x 
     \right].
 \end{equation}
 Invoking again the first inequality in \eqref{pascap360} we may bound the potential term on the right hand side,  so that   we obtain 
$$
\int_{\mathcal U_{\frac{\delta}{4}}} e_\eps(u_\eps)(x) {\rm d}x \leq \sharp (I) 
   {\rm C_{\rm pot}} \left[  C(U, \updelta)  \int_{\mathcal V_\updelta} e_\eps(u_\eps) {\rm d}x
     +\frac{\eps}{\delta}  
    \int_{\mathcal U_{\frac{\delta}{2}} } e_\eps(u_\eps){\rm d} x 
     \right].
$$
This inequality finally leads to the conclusion \eqref{borges}.
 \end{proof}

 \section {Proof of the energy decreasing property}
 \label{challengedata}
  The purpose of this section is to provide a proof to  Proposition \ref{sindec}.

  \subsection{An improved estimate of the energy on level sets}
  In this paragraph,   we consider again for given $0<\eps \leq 1$ a solution  $u_\eps: \D^2 \to \R^k$  to \eqref{elipes} and  specify the result of Proposition \ref{kappacity} for special choices of $\upkappa$ and $\varrho_\eps$. More precisely, we choose
  \begin{equation}
  \label{cabrovski}
  \varrho_\eps=\mathfrak r_\eps  {\rm \ and \ } \upkappa_\eps=C_{\rm bd} \sqrt{\E_\eps(u_\eps)}, 
  \end{equation}
   where $\frac 3 4 \leq \mathfrak r_\eps \leq 1$ is  the radius introduced in  
  subsection \ref{radamel}, Lemma \ref{moyenne}  for   the choice $\displaystyle{r_1=1, r_0=\frac 34}$ and where the constant  $C_{\rm bd}$ is choosen as 
  \begin{equation}
  \label{Cbd}
  C_{\rm bd}=\sup \{ 2{C_{\rm unf}}, \sqrt{\frac{1}{16\sqrt{\lambda_0}}}\},
  \end{equation}
   $C_{\rm unf}$ being  the constant provided in Lemma \ref{valli}. With this choice, we have 
   \begin{equation}
   \label{minos}
   \upkappa_\eps^2 \geq \frac{1}{16\sqrt{\lambda_0}}, 
   \end{equation}
so that the bound \eqref{camembert} is  satisfied for $\upkappa=\upkappa_\eps$. 
  We notice that, in view of \eqref{bornuni2},  there exists some element $\upsigmam \in \Sigma$ such that 
  \begin{equation}
      \label{bornuni20}
     \vert u(\ell)-\upsigmam \vert \leq  2{\rm C}_{\rm unf} \sqrt{{\E}_\eps(u_\eps, \D^2)})\leq \upkappa_\eps,  \  \   {\rm \ for \ all \ } \ell \in 
     \S^1(\tilde {\mathfrak r}_\eps),
      \end{equation}
      so that condition \eqref{kappacite} is automatically fullfilled in view of our choice  our choices of parameters, in particular \eqref{Cbd}.
   The main result of this subsection is  the following:

  \begin{proposition}  
  \label{pasmain}
Assume that $0<\eps \leq 1$ and that $u_\eps$ is a solution of \eqref{elipes} on $ \D^2$.   There exists a constant $C_{\Upsilon}>0$ such 
 \begin{equation}
 \label{smalto}
  \int_{\Upsilon_\eps( \mathfrak r_\eps, \upkappa_\eps)} e_\eps(u_\eps)(x) {\rm d}x \leq  C_{\Upsilon} \left[ \left(\int_{\D^2} e_\eps(u_\eps)(x) {\rm d}x\right)^{\frac32} 
  +
  \eps \int_{\D^2} e_\eps(u_\eps)(x) {\rm d}x\right]. 
  \end{equation}
  \end{proposition}    
   
  \begin{proof} Notice first that   the  result \eqref{smalto} is non trivial only when the energy is small, otherwise it is obvious. We introduce therefore  the smallness condition  on the energy
    \begin{equation}
  \label{smallo}
  \int_{\D^2} e_\eps(u_\eps){\rm d}x \leq \upnu_1\equiv  \frac{\upmu_0^2}{4C_{\rm bd}^2}, 
  \end{equation}
    and distinguish two cases.
   
   \medskip
  \noindent
   {\bf Case 1}: {\it   Inequality \eqref{smallo} does not hold, that is} $\E(u_\eps) \geq \upnu_1$. In this case  \eqref{smalto} is straightforwardly satisfied, provided we choose the constant $C_{\Upsilon}$ sufficiently large so that
      $$ C_{\Upsilon} \geq \frac{1}{\sqrt{\upnu_1}}. 
       $$
     Indeed, we obtain, since \eqref{smallo} is not  satisfied, 
      \begin{equation}
      \begin{aligned}
      C_{\Upsilon}  \left(\int_{\D^2} e_\eps(u_\eps)(x) {\rm d}x\right)^{\frac32} &\geq  C_{\Upsilon} (\upnu_1)^{\frac 12} \int_{\D^2} e_\eps(u_\eps)(x) {\rm d}x \\
      &\geq \int_{\D^2} e_\eps(u_\eps)(x) {\rm d}x  
      \geq  \int_{\Upsilon_\eps( \mathfrak r_\eps, \upkappa_\eps)} e_\eps(u_\eps)(x) {\rm d}x.
      \end{aligned}
      \end{equation}

      \medskip
     \noindent
   {\bf Case 2}: {\it Inequality \eqref{smallo} does  hold}.     Since assumption \eqref{kappacite} is satisfied  for $\varrho_\eps=\mathfrak r_\eps$  thanks to \eqref{bornuni20}, we are in position to apply Proposition \ref{kappacity}.   It yields 
       \begin{equation}
  \label{nonodes}
   \int_{\Upsilon_\eps( \mathfrak r_\eps, \upkappa_\eps)} e_\eps(u_\eps)(x) {\rm d}x \leq  {\rm C} \left[ 
 \upkappa_\eps  \int_{\D^2(\mathfrak r_\eps)} \frac{V(u_\eps)}{\eps} {\rm d} x  +\eps    \int_{\partial \D^2(\mathfrak  r_\eps)} e_\eps(u_\eps){\rm d} \ell
     \right].
        \end{equation}
   Inequality  \eqref{smalto}   then follows directly  from \eqref{nonodes} in view of the definition  $\upkappa_\eps=C_{\rm bd} \sqrt{\E_\eps(u_\eps)}$ of $\upkappa_\eps$ and the fact that, by definition of the energy,  
   $\displaystyle{\frac{V(u_\eps)}{\eps} \leq e_\eps(u_\eps)}$. 
   \end{proof}

     At this stage,  we have already derived an inequality very close to \eqref{bien}, namely  inequality \eqref{smalto} of Proposition \ref{pasmain}. However it holds only on a domain where points on which the value of $\vert u_\eps-\upsigma_i \vert $ is large in some suitable sense have been removed. To go further, we invoke iimproved estimates on the potential $V$ which are derived in the next subsection.     
     
     \subsection{Improved potential estimates}
   
   \begin{proposition}  
   \label{propsac}  Assume that $0<\eps \leq 1$ and that $u_\eps$ is a solution of \eqref{elipes} on $ \D^2$.
   There exists a constant $C_{\rm V}>0$  such that
   \begin{equation}
   \label{propsac0}
   \frac{1}{\eps} \int_{\D^2(\frac{5}{8})}V(u_\eps) {\rm d} x \leq   C_{\rm V} \left[ 
  \left(\int_{\D^2} e_\eps(u_\eps)(x) {\rm d}x\right)^{\frac32} 
  +
  \eps \int_{\D^2} e_\eps(u_\eps)(x) {\rm d}x
\right]. 
 \end{equation}  
 \end{proposition}
   \begin{proof}   The proof  combines the energy estimates  of Proposition \ref{pasmain},   the avering argument of Lemma  \ref{remoyen} together with the potential estimate provided in Proposition \ref{pascap0}.  We first apply Proposition \ref{remoyen} with the choice $\varrho=\mathfrak r_\eps$ and $\upkappa=\upkappa_\eps$, where  $\mathfrak r_\eps$ and $\upkappa_\eps$ have been defined in  \eqref{cabrovski}. Since in view of definitions \eqref{cabrovski}, \eqref{Cbd} and \eqref{minos}  the lower-bound \eqref{camembert} is verified for $\upkappa_\eps$, we may invoke Proposition \ref{remoyen} to assert that there exists some radius $\displaystyle{\uptau_\eps \in [\mathfrak r_\eps, \varrho]}$ such that 
   \begin{equation*}
   \label{rugby}
     \int_{\S^1(\uptau_\eps)}e_\eps(u_\eps){\rm  d } \ell  \leq   \frac{1}{\varrho_\eps-\frac{11}{16}} \, \E_\eps(u_\eps, \Upsilon( \uptau_\eps,  \upkappa_\eps)).
     \leq   16 \, \E_\eps(u, \Upsilon_\eps(\tilde {\uptau}_\eps,  \upkappa_\eps)).
  \end{equation*}
Invoking Inequality \eqref{smalto} of  Proposition   \ref{pasmain}, are led to 
  \begin{equation}
   \label{rugby2}
     \int_{\S^1(\uptau_\eps)}e_\eps(u_\eps ){\rm  d } \ell  \leq   16C_{\Upsilon} \left[ \left(\int_{\D^2} e_\eps(u_\eps)(x) {\rm d}x\right)^{\frac32} 
  +
  \eps \int_{\D^2} e_\eps(u_\eps)(x) {\rm d}x\right]. 
 \end{equation}
On the other hand, thanks to Proposition \ref{pascap0}, we have    
   \begin{equation}
     \label{pascabul0}
      \frac{1}{\eps} \int_{\D^2( \uptau_\eps)} V(u_\eps){\rm d}x \leq  2\uptau_\eps \int_{\S^1(\uptau_\eps)} e_\eps(u_\eps) {\rm d}\ell\leq  
      2 \int_{\S^1(\uptau_\eps)} e_\eps(u_\eps) {\rm d}\ell.
       \end{equation}
Combining \eqref{rugby2} and \eqref{pascabul0} with the fact that $\uptau_\eps \geq \frac 58$, we derive \eqref{propsac} with
$$C_{\rm V}=32  C_\Upsilon. $$
The proof is complete.    
   \end{proof}

\subsection{Proof of Proposition \ref{sindec} completed}

We introduce   first a new radius $\displaystyle{\tilde{\mathfrak r}_\eps\in [\frac {9}{16},\frac 5 8]}$ corresponding  to  the intermediate  radius defined in Lemma \ref{moyenne}  for   the choice  $\displaystyle{r_1=\frac{9}{16}, r_0=\frac7 8}$  so that it satisfies 
\begin{equation}
\label{igloo}
\int_{\S^1(\tilde {\mathfrak r}_\eps)}e_\eps(u){\rm  d } \ell  \leq 16 \, \E_\eps(u, \D^2(\frac 5 8)). 
\end{equation}
It follows  as above from Lemma \ref{valli}  that there exists   some  element $  \upsigman \in \Sigma$, possibly different from $\upsigmam$ defined in \eqref{bornuni20}, such that 
   \begin{equation}
      \label{bornuni3}
     \vert u(\ell)-\upsigman \vert \leq  4 {\rm C}_{\rm unf} \sqrt{ \E_\eps\left(u, \D^2(\frac5 8)\right)},  \  \   {\rm \ for \ all \ } \ell \in 
     \S^1( \tilde {\mathfrak r}_\eps).
      \end{equation}
In order to  apply Proposition \ref{bornepoting}, we introduce once more a smallness condition on the energy, namely
\begin{equation}
\label{smallitude}
\E_\eps( u_\eps) \leq \upeta_2\equiv \frac{ \upmu_0^2}{256  {\rm C}_{\rm unf}^2}.
\end{equation}
 We then  distinguish two cases:

\medskip
\noindent
{\bf Case 1:} {\it The smallness condition \eqref{smallitude} holds}. In this case, we have, in view of \eqref{bornuni3}
  \begin{equation*}
     \vert u(\ell)-\upsigman \vert \leq  4 {\rm C}_{\rm unf} \sqrt{\upeta_2}=\frac{\upmu_0}{4},  \  \   {\rm \ for \ all \ } \ell \in 
     \S^1( \tilde {\mathfrak r}_\eps), 
      \end{equation*}
 so that condition \eqref{kappaciting} holds fo $\varrho_\eps=  \tilde {\mathfrak r}_\eps$
(with $\upsigmam$ replaced by $\upsigman$). We are therefore in position to apply Proposition \ref{bornepoting} on the disk 
$\D^2(\tilde{ \mathfrak r}_\eps)$,  which yields
 \begin{equation}
  \label{nonodesing0}
   \int_{\D^2( \tilde {\mathfrak r}_\eps)} e_\eps(u_\eps)(x) {\rm d}x \leq  {\rm C_{\rm pot}} \left[ 
   \int_{\D^2(\frac 5 8)} \frac{V(u_\eps)}{\eps} {\rm d} x  +\eps    \int_{\partial \D^2(\tilde {\mathfrak r}_\eps)} e_\eps(u_\eps){\rm d} \ell
     \right].
   \end{equation}
 Invoking Proposition \ref{propsac}  and inequality \eqref{igloo} we are hence led to 

$$\int_{\D^2( \tilde {\mathfrak r}_\eps)} e_\eps(u_\eps)(x) {\rm d}x   \leq   {\rm C_{\rm pot}} C_{\rm V} 
  \left(\int_{\D^2} e_\eps(u_\eps)(x) {\rm d}x\right)^{\frac32} 
  + {\rm C_{\rm pot}} \left( {\rm C}_{\rm V}  + 16  \right)
  \eps \int_{\D^2} e_\eps(u_\eps)(x) {\rm d}x,  
$$
 which yields \eqref{bien}, fore a suitable choice of the constant ${\rm C}_{\rm dec}$.

 \bigskip
\noindent
{\bf Case 2:} {\it The smallness condition \eqref{smallitude}  \emph{does not } holds}. In this case, inequality \eqref{bien} is straightforwardly fullfilled, provided we choose
$${\rm C}_{\rm dec} \geq  \upeta_2^{-\frac 12}.$$
The proof is  hence complete in both cases.
\qed
\section{  Proof of the Clearing-out theorem}
\label{solde}
The purpose of this section  is to provide the proof of the clearing-out property stated in Theorem \ref{clearingoutth}.
 We  first turn to  the uniform bound \eqref{benkon}.   As a matter of fact, we  will first prove  a slightly weaker version of \eqref{benkon}.
 
 \begin{proposition}
 \label{benko0}
Let  $0<\eps \leq 1$ and  $u_\eps$  be a solution of 
  \eqref{elipes}  on $\D^2$. There exists a constant $\upeta_1>0$ such that 
  if 
 \begin{equation}
 \label{benkito}
 E_\eps(u_\eps, \D^2) \leq  \upeta_1
 \end{equation}
  then, we have, for  some $\upsigma \in \Sigma$, the bound 
  $\displaystyle{
  \vert u_\eps(0)-\upsigma  \vert  \leq \frac{\upmu_0}{2}.
  }$
 \end{proposition}

For the proof of Proposition \ref{benko0}, we rely on weaker form of the clearing-out statement we present in  the next subsection. 
\subsection{A weak form of the clearing-out}
The following result is classical in the field (see e.g. \cite{ilmanen, BBH}.

\begin{proposition}
\label{brioche}
   Let $u_\eps$ be a solution of \eqref{elipes} on $\D^2$ with $0< \eps \leq 4$. There exists a constant $\upeta_2>0$ such that if 
$\displaystyle{ \E_\eps (u) \leq \upeta_2 \eps}$, 
 then \eqref{benkon} holds.
\end{proposition}

\begin{proof} Assume that the bound $\displaystyle{ \E_\eps (u) \leq \upeta_2 \eps}$ holds, for some constant $\upeta_2$ to be determined later. 
Imposing first $\upeta_2\leq 1$, it follows from Proposition \ref{princours}  and Proposition \ref{classic} that there exists a constant $C_0>0$ depending only on $V$ such that 
\begin{equation*}
 \vert \nabla u_\eps (x) \vert \leq \frac{C_0}{\eps}   {\rm \ and \ } \vert u_\eps(x) \vert \leq C_0 ,  {\rm \ for \ } x \in \D^2(\frac{7}{8}).
\end{equation*}
Since the potential $V$ is smooth, and hence its gradient is bounded on the disc $\B^k(C_0)$, we deduce that there exists a constant $C_1$ such that 
\begin{equation}
\label{bbq}
\vert \nabla  V(u_\eps) (x) \vert \leq \frac{C_1}{\eps}    {\rm \ for \ } x \in \D^2(\frac{7}{8}).
\end{equation}
Since $\E_\eps (u_\eps) \leq \upeta_2 \eps$, we deduce from the definition of the energy that 
\begin{equation}
\label{lucien}
\int_{\D^2 (\frac{7}{8})} V(u_\eps(x)) {\rm d} x \leq \int_{\D^2} V(u_\eps(x)) {\rm d} x \leq \upeta_2 \eps^2. 
\end{equation}
We claim that 
\begin{equation}
\label{ouidada}
V(u_\eps(x)) \leq \upalpha_0 {\rm \ for \  any  \   } x \in \D^2(\frac{3}{4}).
\end{equation}
Indeed, assume by contradiction  that there exists  some $x_0 \in \D^2(\frac{3}{4})$ such that 
$\displaystyle{V(u(x_0)>\upalpha}$.  Invoking the gradient bound \eqref{bbq}, we deduce that 
$$
V(u_\eps(x) )\geq \frac{\upalpha_0}{2}  {\rm \ for \ }   x\in \D^2\left(x_0, \frac {\upalpha_0 \eps}{2C_1}\right).
$$
Without loss of generality, we may assume that  $C_1$ is chosen sufficiently large so that   
$\displaystyle{\frac {4\upalpha_0}{2C_1}\leq  \frac 18}$ and hence 
$\displaystyle{ \D^2\left(x_0, \frac {\upalpha_0 \eps}{2C_1}\right)\subset \D^2(\frac 78)}$. Integrating  \eqref{ouidada}, on the disk  
$\displaystyle{\D^2\left(x_0, \frac {\upalpha_0 \eps}{2C_1}\right)}$, we are led to 
\begin{equation*}
 \int_{\D^2(\frac 78)}  V(u_\eps(x)) {\rm d} x  \geq  \int_{\D^2(x_0, \frac {\upalpha_0 \eps}{2C_1})}  V(u_\eps(x)) {\rm d} x \geq 
\pi  \frac{\upalpha_0^3}{8C_1^2} \eps^2.
\end{equation*}
This yields a contradiction with \eqref{lucien}, provided we impose  the upper bound on $\upeta_2$ given by
\begin{equation}
\label{imposition}
\upeta_2 \leq \pi  \frac{\upalpha_0^3}{8C_1^2} \eps^2, 
\end{equation}
and  established the claim \eqref{ouidada}.   To complete the proof, we may invoke  Lemma \ref{watson}   and the continuity of the map $u_\eps$ to   asserts that 
there exists some $\upsigma \in \Sigma$ such that
\begin{equation}
\label{presque}
 \vert u_\eps(x)-\upsigma \vert \leq \upmu_0  {\rm \ for \ any \ } x \in \D^2(\frac{3}{4}).
 \end{equation}
This yields almost estimate \eqref{benkon}, except that we   still have to replace $\upmu_0$ by  $\upmu_0 \slash 2$ on the right-hand side of \eqref{presque}. In  order to improve the constant, we merely rely on the  same  type of argument. Arguing as above by contradiction, let us assume that there exists a point $x_1\in \D^3(3\slash 4)$ such that 
\begin{equation}
\label{claimobil}
\vert u_\eps(x_1)-\upsigma \vert > \frac{\upmu_0}{2} {\rm \ and \ hence  \ } V(u_\eps(x_1))> \frac{\lambda_0 \,  \upmu_0^2}{16}, 
\end{equation}  
the second inequality in \eqref{claimobil} being a consequence of the  second statement in  Lemma \ref{watson}.  Invoking again  the gradient bound \eqref{bbq}, we deduce that 
$$V(u_\eps(x) )\geq \frac{\lambda_0 \,  \upmu_0^2}{32}, {\rm \ for \ }   x\in \D^2\left(x_1, \frac {\lambda_0 \upmu_0^2\eps}{32C_1}\right). $$
Integrating the previous inequality, we obtain 
\begin{equation*}
 \int_{\D^2(\frac 78)}  V(u_\eps(x)) {\rm d} x  \geq  \int_{\D^2(x_0, \frac {\upalpha_0 \eps}{2C_1})}  V(u_\eps(x)) {\rm d} x 
 \geq \pi  \frac{ \lambda_0^3 \upmu_0^5}{32768 C_1} \eps^2, 
\end{equation*}
 a contradiction with  \eqref{lucien}, provided we impose that $\upeta_2$  is sufficiently small.
\end{proof}

\subsection {Proof of the Proposition \ref{benko0}}   
 The proof   of the Proposition \ref{benko0} relies  on   inequality  \eqref{bienscale}  of Proposition \ref{sindec}, a standard scaling argument combined  with an iteration procedure. 
 
 \medskip
 \noindent
 {\it   Step 1: A scaled version of inequality \eqref{bienscale}}. Set for   $0<r\leq 1$, 
 $\displaystyle{\E_\eps(r)= \E_\eps \left(u_\eps, \D^2(r)\right)}$, and assume that
 \begin{equation}
  \label{atletico}
  E_\eps (r) \geq  \frac{\eps^2}{r}. 
  \end{equation}
   Then, we have 
   \begin{equation}
  \label{bienscale2}
 \E_\eps(\frac r 2)  \leq  2\Cdec
 \frac{  {\E_\eps (r)\ }^{\frac 32}
}{\sqrt{r}}, 
   {\rm \  \   \ provided  \ }  r\geq \eps. 
    \end{equation}
Indeed, scaling inequality \eqref{bienscale}, we obtain 
  \begin{equation}
  \label{bienscale2}
 \E_\eps(\frac r 2)  \leq  \Cdec
  \left[ \frac{1}{\sqrt{r}}
    {\E_\eps (r)\ }^{\frac 32}+ 
  \frac{ \eps}{r} \E_\eps (r)
    \right],
     {\rm  \ provided  \ }  r\geq \eps, 
    \end{equation}
 which yields \eqref{atletico}.
 
 \medskip
 \noindent
 {\it Step 2: The iteration procedure}.   
     We  consider the  sequence $(r_n)_{n \in \N}$ of decreasing  radii $r_n$  defined as   $\displaystyle{r_n=\frac{1}{2^n}}$,  for $n \in \N$,  and set
     $\displaystyle{\E^\eps_n=   \E_\eps(r_n)= \E_\eps(\frac {1}{2^n})}$, dropping the superscript in case this  induces  no ambiguity.  We  introduce the number
     \begin{equation}
     \label{benkita}
     n_\eps=\sup\left \{ n \in \N,   {\rm \ such \ that \ }   \E^\eps_n \geq   2^n \eps^2 {\rm \  and \ } r_n=\frac{1}{2^n} \geq \eps \right\}.
     \end{equation}
 If we impose that  $\upeta_1\leq 1$, then condition \eqref{benkito} implies that $\E_\eps (u_\eps)\leq 1$, so that $0$ belongs to the set of the r.h. s of \eqref{benkita}, which is hence not empty. On the other hand,    since $2^n$ tends to infinity as $n$ tends to infinity, and since the sequence $(\E_n)_{n \in \N}$ is bounded by $\E^\eps_0$,   the set of the r.h. s of \eqref{benkita} is bounded and the number $n_\eps$ is a well-defined integer. In view of the defintion of $n_\eps$, inequality \eqref{atletico} is satisfied for every $r_n< r_{n_\eps}$.   We have hence  the inequality
 $$
 \E_{n+1} \leq 2{\sqrt 2}^n\Cdec \left(\E_{n} \right)^{\frac 32},   {\rm \ for \  } n=0, \ldots n_\eps-1.
 $$    
  Set, for $n\in \N$,   $A_n= -\log E_n $. The previous inequality  is turned into
     \begin{equation}
     \label{turnlog}
     A_{n+1} \geq  \frac 3 2 A_n -\frac{(\log 2)}{2} \, n -\log (2\Cdec),   \   {\rm \ for \  } n=0, \ldots n_\eps-1.
     \end{equation}
   In order to study   the sequence $(A_n)_{n \in \N}$,  we will invoke  the next result. 
   
   \begin{lemma} 
   \label{suites}
    Let $n_\star \in \N^*$,  $(a_n)_{n \in \N}$  and $(f_n)_{n \in \N}$ be two  sequences of numbers such that
   \begin{equation} 
   \label{suitineq}
   a_{n+1} \geq {\rm c}_0 \, a_n  -f_n, {\rm \ for \ all \ } n \in \N, n \leq  n_\star,  
   \end{equation}
    where ${\rm c}_0>1$  represents a  given constant.    Then we have  the inequality,
    \begin{equation}
    \label{suiteconc}
    a_n\geq {\rm c}_0^n \left(a_0 -\underset{k=0}{\overset {n}\sum  }\frac{1}{{\rm c}_0^{k+1}} f_k\right)  {\rm \ for \ }
    n \in \N^*n \leq n_\star.
       \end{equation}
  \end{lemma}
  
  We postpone the proof of Lemma \ref{suites} and complete first the proof of Proposition \ref{benko0}.

\medskip
\noindent
{\it Step 3: Choice of $\upeta_1$ and energy decay estimates}. 
Applying Lemma \ref{suites} to the sequences 
  $(A_n)_{n \in \N}$ and  $(f_n)_{n \in \N}$ with $\displaystyle{f_n= \frac{(\log 2)}{2} \, n+\log (2\Cdec)}$, for any $n \in \N$, so that inequality \eqref{suitineq} is satisfied with  $\displaystyle{{\rm c}_0=\frac 32}$, $n_\star=n_\eps-1$,   we are led to the inequality, for $n=0, \ldots n_\eps-1,$
\begin{equation}
\label{meray0}
\begin{aligned}
A_n =-\log E_n &\geq   \left(\frac32 \right)^n \left[ 
\log  \left( \frac{1}{E_\eps(u_\eps) }\right)
- \gamma_0 \right] \\
&\geq \left(\frac32 \right)^n \left[ 
  -\log \upeta_1
- \gamma_0 \right], \\
\end{aligned}
\end{equation}
where we have used, for the second inequality, the assumption  that inequality \eqref{benkito} holds and where we have set
$$\gamma_0  =\underset{k=0}{\overset {\infty}\sum} \left(\frac 23\right)^{k+1} (\frac{(\log 2)}{2} \, k+\log (2\Cdec) )<+\infty.$$
  We impose a first constraints on the constant $\upeta_1$  namely  
  \begin{equation}
  \label{firstconstra}
   \upeta_1 \leq \exp\left[- (1+\gamma_0)\right] {\rm \ so \ that \ }    -\log \upeta_1 \geq 1+\gamma_0,  
   \end{equation}
 It follows that inequality \eqref{meray0} becomes, provided inequality \eqref{benkito} holds,   
 \begin{equation}
 \label{meray}
 \E_n \leq  \exp \left[-\left(\frac32 \right)^n\right]   \   {\rm \ for \  } n=0, \ldots n_\eps-1.  
  \end{equation}
yielding a very fast decay of the energy.

\medskip
\noindent
{\it Step 4:  Estimates of $n_\eps$ and $r_{n_\eps}$}. 
It follows from \eqref{meray}  and the definition of $n_\eps$ that 
$$
 \exp (2 \log \eps ) \leq 2^{-n}  \E_n \leq \exp \left[-\left(\frac32 \right)^n-n \log 2 \right]  \   {\rm \ for \  } n=0, \ldots n_\eps-1, 
$$
So that 
$$\left(\frac32 \right)^{n_\eps} +n_\eps \log 2\leq  2 \vert \log \eps \vert {\rm \ and \ hence \ }  
\left(\frac32 \right)^{n_\eps} \leq  2 \vert \log \eps \vert$$
Taking the logarithm of both sides, we are led to the upper bound for $n_\eps$
$$ n_\eps \leq\frac{  \log (2\vert \log \eps \vert)}{\log 3-\log 2}.   $$
It  yields a lower bound for $n_{r_\eps}$, given by 
\begin{equation}
\label{rateau}
\begin{aligned}
r_{n_\eps}=2^{-n_\eps}=\exp (-(\log 2)\,  n_\eps)
& \geq 
\exp \left(-\log (2\vert \log \eps \vert)\frac{\log 2} {\log 3- \log 2}\right)   \\
\geq (2\vert \log \eps \vert)^{-\upgamma_1}, 
\end{aligned}
\end{equation}
where we have set 
$$\upgamma_1=\frac{\log 2} {\log 3- \log 2},   {\rm \ so \ that \ } 1\leq \upgamma_1\leq 2.$$
On the other hand, the defintion of $n_\eps$ yields
\begin{equation}
\E_{n_\eps}^\eps \leq \eps^2  r_{n_\eps}^{-1}=2^{n_\eps} \eps^2
\leq (2\vert \log \eps \vert)^{\upgamma_1} \eps^2.
\end{equation}

\medskip
\noindent
{\it Step 3: Use of Proposition \ref{brioche}}. We consider the scaled map
 $\tilde u_\eps$ and the scaled parameter $\tilde \eps \geq \eps$ defined by
 $$
  \tilde u_\eps(x)= u_\eps (r_{n_\eps} x),  {\rm \ for \ } x \in \D^2, {\rm \ 	and \ the \ scaled \ parameter \ } 
  \tilde \eps=r_{n_\eps}^{-1}\eps \leq (2\vert \log \eps \vert)^{\upgamma_1}\eps, 
  $$
  where the last inequality is a consequence of  inequality \eqref{rateau}. Turning back to \eqref{scalingv},  we  are led to 
  the estimate for the energy
  $$
  \E_{ \tilde \eps}(\tilde u_{ \eps}) = r_{n_\eps}^{-1} \Eps (u_\eps, \D^2(r_{n_\eps})) \leq
  \frac{\eps^2}{r_{n_\eps}^2}
  \leq 
  \vert \log \eps \vert)^{2\upgamma_1}\,  \eps^2    $$
  and 
  $$\frac{\E_{ \tilde \eps}(\tilde u_{ \eps})}{\tilde \eps}=\frac{\Eps (u_\eps, \D^2(r_{n_\eps}))}{\eps}\leq  (2\vert \log \eps \vert)^{\upgamma_1} \eps.
  $$
  Since the map $s\to  \vert \log s  \vert^{\upgamma_1} s$ is decreasing on the interval $(0, e^{-\upgamma_1})$, assuming that the constant $\upeta_2$ is choosen to be sufficiently small, there exists a unique  number $\eps_1 \in(0, e^{-\upgamma_1})$, such that  
  \begin{equation}
  \label{epson}
(2\vert   \log \eps_1  \vert) ^{\upgamma_1}\eps_1 =\upeta_2. 
  \end{equation}

  \medskip
  \noindent
  {\it Proof of Proposition \ref{benko0} completed}. We distinguish two cases:
  
  \smallskip
  \noindent
 {\bf Case 1: $0<\eps \leq \eps_1$}. It follows in this case  from the definition \eqref{epson} that 
 \begin{equation}
 \label{kakuta}
 \left\{
 \begin{aligned}
 \E_{ \tilde \eps}(\tilde u_{ \eps})&\leq \upeta_2\,  \tilde \eps  {\rm \ and \ } \\
 \tilde \eps &\leq 1. 
  \end{aligned}
 \right. 
 \end{equation}
 so that it follows  from \eqref{kakuta} 
 that we are in position to apply Proposition \ref{brioche} to the map $\tilde u_\eps$  with parameter $\tilde \eps$:  Hence  there exists some point $\upsigma \in \Sigma$ such that 
 $$\vert \tilde u_\eps(0)-\upsigma \vert \leq \frac{\upmu_0}{2}. $$
  since  $u_\eps(0)= \tilde u_\eps(0)$ the conclusion of Proposition \ref{benko0} follows. 
  
   \medskip
  \noindent
 {\bf Case 2: $1 \geq \eps > \eps_1$}. Besides \eqref{firstconstra} we impose the additional condition 
$\displaystyle{\upeta_1 \leq {\upeta_2}{\eps_1}}$ on $\upeta_1$,  so that we finally may choose the constant $\upeta_1$ as 
 \begin{equation}
 \label{seconstraint}
 \upeta_1=\inf\{ {\upeta_2}\, {\eps_1}, \exp [-(1+\upgamma_1), 1] \}. 
 \end{equation}
  With this choice, we have, for $\eps \geq \eps_1$, 
 $$ \Eps(u_\eps) \leq \upeta_1 \leq  {\upeta_2}{\eps_1}  \leq  \upeta_2 \eps.  $$
  Hence $u_\eps$ fullfills the assumptions of Proposition \ref{brioche}, so that its conclusion yields  again  the existence of an element $\upsigma \in \Sigma$ such that  $\displaystyle{\vert  u_\eps(0)-\upsigma \vert \leq \frac{\upmu_0}{2}.} $

  \medskip
  In both cases, we have hence established  the conclusion of Proposition \ref{benko0} so that the proof is complete. 
  

  In the course of the proof, we have used Lemma \ref{suites}, which has not been proved yet.
   
        \begin{proof} [Proof of Lemma \ref{suites}]
     We introduce, inspired by the method of variation of constant,  the sequence $(b_n)_{n \in \N}$ defined by $a_n={\rm c}_0^n  \, b_n$, for any $n \in \N$.  Substituting into \eqref{suitineq},  we obtain
     $${\rm c}_0^{k+1} b_{k+1} \geq {\rm c}_0^{k+1}b_k-f_k, {\rm \ for \ all \ } k \in \{0, \ldots, n_\star\},  $$
      so that 
      $$b_{k+1}-b_k \geq - \frac{1}{{\rm c}_0^{k+1}} f_k,  {\rm \ for \ all \ } k \in \{0, \ldots, n_\star\}.$$
     Let $n\in \N$, $n\leq n_\star$. Summing these relations for $k=0$ to $k=n-1$, we are led to
      $$b_n \geq b_0 -\underset{k=0}{\overset {n}\sum  }\frac{1}{{\rm c}_0^{k+1}} f_k=
   	a_0 -\underset{k=0}{\overset {n}\sum  }\frac{1}{{\rm c}_0^{k+1}} f_k,$$
   which, in view of the definition of $b_n$,  yields the desired conclusion \eqref{suiteconc}. 
     \end{proof}
     \subsection {Proof of assertion  \eqref{benkon}}
      We  impose a first constraint to  the value of the constant $\upeta_0$ of Theorem \ref{clearingoutth}  by requiring that
   \begin{equation}
   \label{pinault}
      \upeta_0\leq  \frac{\upeta_1}{4}, {\rm \ where \ } \upeta_1 {\rm \ is \ introduced \ in \ Proposition \ } \ref{brioche}.
      \end{equation}
    Next let 
     $\displaystyle{x_0 \in \D^2(\frac 3 4)}$  be an arbitrary point.  We consider the  scaled parameter $\tilde \eps= 4\eps$  scaled-translated map $\tilde u_\eps$ defined on $\D^2$ by 
     $$\tilde u_\eps (x)=u_\eps (x_0+ \frac{1}{4} x) {\rm \ for \ every \ } x \in \D^2, $$
      so that
      \begin{equation}
      \label{arnault}
       \E_{\tilde \eps}(\tilde u_\eps)=4 \E_\eps\left (u_\eps, \D^2 (x_0, \frac 14)\right)
       \leq 4  \E_\eps (u_\eps) \leq 
      4 \eta_0
      \leq   \eta_1, 
     \end{equation}
  where we have used assumption \eqref{petitou} and  \eqref{pinault} for the last  inequality. As above, we distinguish two cases. 
  
  \medskip
  \noindent
  {\bf Case 1: $\eps\leq \frac{1}{4}$}. In this case $\tilde\eps \leq 1$, so that, in view of \eqref{arnault}, we are in position to apply Proposition \ref{brioche}: It yields  an element $\upsigma_{x_0} \in \Sigma$, depending possibly on the point $x_0$,   such that  
  $$\displaystyle{\vert  \tilde u_\eps(0)-\upsigma_{x_0} \vert \leq \frac{\upmu_0}{2}.} $$
    Since 
  $\tilde u_\eps (0)=u_\eps (x_0)$, we conclude that
 \begin{equation}
 \label{upsigmax0}
  \vert u_\eps(x_0)-\upsigma_{x_0} \vert \leq \frac{\upmu_0}{2}. 
  \end{equation}
  Since inequality \eqref{upsigmax0} holds for \emph{any point} $x_0\in \D^2(3 \slash 4)$, a continuity argument shows  that the point $\upsigma_{x_0}$ does not depend on $x_0$, so that the proof  of Proposition \ref{benko0} is complete in Case 1. 
  
   \medskip
  \noindent
  {\bf Case  2: $1\geq \eps\geq \frac{1}{4}$}. In this case $\tilde \eps \leq 4$. we impose the additional constraint on the constant $\upeta_0$ requiring that $16 \upeta_0\leq \upeta_2$, so that we may choose 
  $$ \upeta_0=\inf \{\frac{1}{4} \upeta_1, \frac{1}{16} \upeta_2\}. $$  
  It then follows from assumption \eqref{petitou} that 
        \begin{equation}
      \label{arnault2}
       \E_{\tilde \eps}(\tilde u_\eps)=4 \E_\eps\left (u_\eps, \D^2 (x_0, \frac 14)\right)
       \leq 4  \E_\eps (u_\eps) \leq 
      4 \eta_0
      \leq   \frac{\eta_2} {4}\leq \upeta_2 \tilde \eps. 
     \end{equation}
Hence, we are once more  in position to apply Proposition \ref{brioche}, so that there exists  an element $\upsigma_{x_0} \in \Sigma$, depending possibly on the point $x_0$  such that  
  $\displaystyle{\vert  \tilde u_\eps(0)-\upsigma_{x_0} \vert \leq \frac{\upmu_0}{2}.} $  Since 
  $\tilde u_\eps (0)=u_\eps (x_0)$, we conclude that
  $$
  \vert u_\eps(x_0)-\upsigma_{x_0} \vert \leq \frac{\upmu_0}{2}. 
  $$
The proof of assertion \eqref{benkon} is hence complete. 
     \subsection{Proof of Theorem \ref{clearingoutth} completed}
    The only remaining unproved assertion is  the energy estimate \eqref{engie}, which we establish next. The proof is  parallel and actually much easier then our earlier energy estimate.  We first invoke Lemma \ref{moyenne} with 
    $\displaystyle{r_1=\frac 34}$  and  $\displaystyle{r_0= \frac{5}{8}}$: This yields a radius $\displaystyle{\mathfrak r_\eps\in [\frac 58, \frac 34]}$  and an element $\upsigma \in \Sigma$ such that  
    \begin{equation}
    \label{ret}
     \int_{\S^1(\mathfrak r_\eps)}e_\eps(u_\eps){\rm  d } \ell  \leq 8 \, \E_\eps(u, \D^2)) {\rm \ and \ }
         \int_{\S^1(\mathfrak r_\eps) }\vert u_\eps- \upsigma\vert \vert \nabla  u_\eps\vert \leq 16\sqrt{\lambda_0^{-1}} \E_\eps(u_\eps, \D^2)).
     \end{equation}
 We multiply the equation by $(u_\eps-\upsigma)$ and integrate on the disk $\D^2(\mathfrak r_\eps)$ which yields, as in \eqref{stokounette}   
     \begin{equation}
  \label{stokom}
  \begin{aligned}
  \int_{\D^2(\mathfrak r_\eps)} \eps \vert \nabla u_\eps \vert^2 +\eps^{-1} \nabla_u V(u_\eps)\cdot (u_\eps-\upsigma)
  &= \eps  \int_{\S^1(\mathfrak r_\eps)}  \frac{\partial u_\eps}{\partial r} \cdot (u_\eps-\upsigma). 
 \end{aligned}
  \end{equation} 
     We deduce from \eqref{ret}   that
     \begin{equation}
     \label{reti}
     \int_{\S^1(\mathfrak r_\eps)}  \frac{\partial u_\eps}{\partial r} \cdot (u_\eps-\upsigma)
  \leq     \int_{\S^1(\mathfrak r_\eps) }\vert u_\eps- \upsigma\vert \vert \nabla  u_\eps \vert
  \leq 16\sqrt{\lambda_0^{-1}} \E_\eps(u_\eps, \D^2)). 
     \end{equation}
 We use next  the fact that, in view of assertion \eqref{benkon}, we have  $\displaystyle{\vert u_\eps-\upsigma  \vert   \leq \frac{\upmu_0}{2}}$ on the disk 
 $\D^2(\mathfrak r_\eps)$.  Arguing  as in \eqref{duglandmoins}, we have  the point-wise inequality
     \begin{equation}
     \label{retina}
     \eps \vert \nabla u_\eps \vert^2 +\eps^{-1} \nabla_u V(u_\eps)\cdot (u_\eps-\upsigma)\geq  \frac{\lambda_0} {2\lambda_{\rm max}} e_\eps(u). 
     \end{equation}
  Combining \eqref{stokom}    with \eqref{retina} and \eqref{reti},  we obtain
  $$
   \int_{\D^2(\mathfrak r_\eps)} e_\eps(u_\eps){\rm d}x \leq  {16 }{\lambda_0^{-\frac 32}}\lambda_{\rm max}\, \eps \E_\eps(u, \D^2)), 
  $$
 Which yields the energy estimate \eqref{engie}  choosing $\displaystyle{\Cnrg= {16}{\lambda_0^{-\frac 32}}\lambda_{\rm max}}$. The proof of Theorem \ref{clearingoutth} is hence complete. 
 \section{Properties of the concentration set $\mathfrak S_\star$}
 \label{sectionsix}
 The purpose of this section is to provide the proof of  assertion i) of  Theorem \ref{maintheo}.  We start with the proof of Theorem \ref{claire}, the clearing-out property for the measure $\upnu_\star$.
\subsection{Proof of Theorem \ref{claire}} 
   Assume that  that $x_0$ and $r>0$ are such that 
   $$\upnu_\star (\D^2(x_0, r))<\upeta_0\, r.$$
 It follows from the definition of the measure $\upnu_\star$, which is a limit of energy densities, that there exists some integer $n_0 \in \N$ such  that, for $n \geq n_0$ we have 
 $$\E_\eps(u_{\eps_n}, \D^2(x_0, r)) \leq \upeta_0\, r.$$
   Hence, we are in position to apply Proposition \ref{cestclair}, so that
 $$
  \E_\eps\left(u_{\eps_n} , \D^2\left(x_0, \frac{ 5r}{8}\right)\right) \leq \Cnrg \, \frac{\eps_n}{r} E_{\eps_n} \left(u_{\eps_n}, \D^2\left(x_0, r\right)\right) \to 0 {\rm \ as \ } n \to + \infty.
$$
 It follows that $\displaystyle{\upnu_\star \left(\D^2(x_0,\frac{r}{2})\right)=0}$ and the proof is complete. 
 \subsection{Elementary consequences of the clearing-out property}
We present here some simple consequences of the definition of $\mathfrak S_\star$,  as well as of the clearing out property stated in Proposition \ref{danes}.

  \begin{proposition} 
  \label{proptheo1}
  The set $\mathfrak S_\star$ is  a closed subset of $\Omega$.
 \end{proposition}
 \begin{proof} It suffices to prove that its complement, the set $\mathfrak U_\star=\Omega\setminus \mathfrak S_\star$ is an \emph{open  subset}  of $\Omega$. This property is actually a direct consequence of the clearing out  property stated in Theorem \ref{claire}. Indeed let $x_0$ be an arbitrary point in $\mathfrak U_\star$. It follows from the definition  \eqref{mathfrakSstar} of $\mathfrak S_\star$ that  $\theta_\star (x_0) <\upeta_0$,  so that there exists some radius 
 $r_0>0$ such that $\D^2(x_0, r_0) \subset \Omega$ and such that 
  $$ \upnu_\star (\D^2(x_0, r_0) ) <r_0\upeta_0.$$
  In view   of Theorem \ref{claire}, we  deduce  that
  $\displaystyle{ \upnu_\star (\D^2(x_0, \frac{r_0}{2}) )=0.}$
  Hence, for any point $\displaystyle{x\in \D^2(x_0, \frac{r_0}{4})}$, we  have $\theta_\star (x)=0$ and  therefore
  $$  \D^2(x_0, \frac{r_0}{4}) \subset \mathfrak U_\star. $$
Hence,  $\mathfrak U_\star$ is an open set.
 \end{proof}

 \begin{proposition}
 \label{danes}
 The set $\mathfrak S_\star$ has finite one-dimensional Hausdorff dimension. There exist a  constant ${\rm C_H} >0$ depending only on the potential $V$ such that 
 $$\mathcal H^1 (\mathfrak S_\star)\leq {\rm C_H} M_0.$$
 \end{proposition}
 
 \begin{proof} The proof relies on a standard covering argument. Let 
$0<\rho<\frac{1}{4}$ be given, and  consider the set 
$$\displaystyle{\Omega_\rho=\{ x \in \Omega, {\rm dist}(x, \partial \Omega) \geq \rho\}}.$$
 Next let $0<\delta<\rho$ be given.  We introduce    a standard finite  covering of $\Omega_\rho$  of size $\delta$, that is such that
$$
\Omega_\rho \subseteq \underset{ j \in I}\cup \D^2\left(x_j,\delta\right) \,  \  \text{and} \ 
\D^2\left(x_i,\frac{\delta}{2}\right) \cap \D^2\left(x_j,\frac{\delta}{2}\right) = \emptyset
\text{ for } i\neq j \in I.
$$
One may take for instance the points $x_i$ on a uniform  square lattice of $\R^2$, with nearest neighbor distance being 
$\displaystyle{\frac \delta 2}$. 
We introduce  then the set of indices 
$$\displaystyle{
I_\delta = \left\{ i  \in I, \text{ such \ that \  } \D^2(x_i,\delta)\cap \mathfrak S_\star \neq \emptyset\right\}, }$$
 so that 
given any arbitrary index  $i\in I_\delta,$ there exists a point  $y_i \in \mathfrak S_\star \cap
\D^2(x_i,\delta)$. It follows from the definition of $\mathfrak S_\star$ that 
\begin{equation}
\label{davy}
\theta_\star (y_i) \geq \upeta_0.
\end{equation}
We claim that, for any $0<r\leq \delta$, we have 
\begin{equation}
\label{pret}
\upnu_\star(\D^2(y_i, r )) \geq  \upeta_0 \,  r.
\end{equation}
Indeed, if \eqref{pret} were not true, then  we would be in position to apply Theorem  \ref{claire}, which would imply  that 
$\displaystyle{\upnu_\star(B(y_i, \frac{\delta}{2} ))=0}$, and  hence that $\theta_\star (y_i)=0,$ a contradiction which \eqref{davy}. hence \eqref{pret} is established.

Since the balls $\displaystyle{\D^2(x_i,\frac{\delta}{2})}$ are
disjoints,
 we have
\begin{equation}
 M_0 \geq \upnu_\star (\Omega_\rho) \geq   {\underset {i \in I_\delta } \sum} \upnu_\star \left( \D^2(x_i,\frac{\delta}{2})\right) \geq {\underset {i \in I_\delta } \sum} \upeta_0 \frac{ \delta}{2} =\upeta_0 \, \sharp (I_\delta) \frac{\delta}{2}. 
\end{equation}
It follows therefore that
$$\sharp (I_\delta) \delta \leq \frac{2M_0} {\upeta_0}.$$
Therefore, letting $\delta\to 0$, it follows, as a consequence of the  definition of the  one-dimensional Hausdorff measure that 
$$
\mathcal H^1 (\mathfrak S_\star \cap\Omega_\rho) \leq \underset { \delta \to 0}\liminf \, 2\, \sharp (I_\delta) \delta  \leq 
 \frac{4M_0} {\upeta_0}.
$$
The conclusion follows letting $\rho \to 0$, choosing $\displaystyle{{\rm C_H}=\frac{4}{\upeta_0}  }$. 
\end{proof}
\subsection{ Proof of Theorem \ref{bordurer}}
Theorem \ref{bordurer}  is a direct consequence of  Proposition \ref{pave} which has actually been taylored for this purpose.      
 Indeed,  since $\upnu_\star(\mathcal V_\updelta)=0$,  we have the convergence
  $$
  \int_{\mathcal V_\updelta} 
      e_{\eps_n}(u_{\eps_n}) \, {\rm d}\, x \to  0 {\rm \ as \ } n \to +\infty, 
  $$    
   so that condition \eqref{carpediem} is fullfilled for $\eps=\eps_n$ and the map $u_{\eps_n}$, provided $n$ is sufficiently large, say larger than some given value $n_0$. We are therefore in position to conclude, thanks to Proposition \ref{pave}, provided $n\geq n_0$ is sufficiently large,  that 
 \begin{equation*}
\label{borges2}
\begin{aligned}
 \int_{\mathcal U_{\frac{\delta}{4}} }e_{\eps_n} (u_{\eps_n}) {\rm d} x 
 &\leq {\rm C}_{\rm ext} (\mathcal  U, \delta)
  \left(
  \int_{\mathcal V_\updelta} 
      e_{\eps_n}(u_{\eps_n}){\rm d}x  + \eps_n \int_{\mathcal U_\delta} e_{\eps_n}(u_{\eps_n}) {\rm d} x 
\right)\\
& \leq   {\rm C}_{\rm ext} (\mathcal  U, \delta)
  \left(
  \int_{\mathcal V_\updelta} 
      e_{\eps_n}(u_{\eps_n}){\rm d}x  + \eps_n {\rm M}_0 \right).
\end{aligned}
 \end{equation*}
It follows that 
$$
 \int_{\mathcal U_{\frac{\delta}{4}} }e_{\eps_n} (u_{\eps_n}) {\rm d} x  \to 0 {\rm \  as \ } n \to +\infty,
$$ 
 so that the proof is complete. 
      \qed

\subsection{ Connectedness properties of  $\mathfrak S_\star$ }
 The purpose of the present  section is, among other things,  to provide the proof of Proposition \ref{connective}.  Given $r>0$ and $x_0\in \Omega$  such that $\D^2(x_0, 2r) \subset \Omega$, we consider the  closed  set 
 $$\mathfrak S_{\star, \varrho}=  \mathfrak S_{\star, \varrho}(x_0)\equiv\mathfrak S_\star \cap \overline{\D^2(x_0, \varrho)} {\rm \ for \  } \varrho \in [0, 2r). $$ 
The main result of this section is:

  \begin{proposition}
  \label{localconnect}
   Let $r>0$ and $x_0\in \Omega$ be such that $\D^2(x_0, 2r) \subset \Omega$. Then the set $\mathfrak S_{\star, r}(x_0)$ contains a finite number of 
   path-connected components. 
  \end{proposition} 

 The proof of Proposition \ref{localconnect} relies on several intermediate properties we present next. 

 \begin{proposition} 
 \label{lesconti}
 Let $r>0$ and $x_0\in \Omega$ be as above.
  The closed set 
  \begin{equation}
  \label{theunion}
  \mathfrak Q_{\star, r}(x_0)=\mathfrak S_{\star, r}(x_0)\cup \S^2(x_0, r)
  \end{equation}
   is  a continuum, that is,   it is compact and  connected. 
 \end{proposition}

 

 \noindent
 {\it Proof}.  The proof of compactness  of $ \mathfrak Q_{\star, r}(x_0)$  is a straightforward consequence of Proposition \ref{proptheo1}, since both sets composing the union \eqref{theunion} are compact. The  proof of  connectedness of $ \mathfrak Q_{\star, r}(x_0)$ is more involved, and strongly  relies  on Theorem \ref{bordurer}, as we will see next.  In order to invoke Theorem \ref{bordurer},  a first step is to   approximate $\mathfrak S_{\star, r}$ by  sets $\mathfrak S_{\updelta, r}$ with a simpler structure.
 
 \medskip
 \noindent
 {\it Definition of the approximating sets $\mathfrak S_{\updelta, r}$}. These sets are defined  using a \emph{Besicovitch covering} of $\mathfrak S_{\star, r}$. 
 Let   
 $$\updelta_{x_0, r}={\rm dist}(\D^2(x_0, r), \partial \Omega)>0.$$
  For given $0<\updelta<\updelta_{x_0, r}$, we consider the covering of $\mathfrak S_{\star, r}$ by the collection of open  disks 
 $\displaystyle{\{\D^2(x_0, \updelta)\}_{x \in \mathfrak S_{\star, r}}}$, which is obviously a covering of $\mathfrak S_{\star, r}$, and actually a Besicovitch covering.  We may therefore invoke  Besicovitch covering theorem, to  asserts that there exists 
 a  universal constant $\mathfrak p$, depending only on the dimension $N=2$,  and  $\mathfrak p$ families of  points 
 $\{x_{i_1}\}_{i_1\in A_1}$,  $\{x_{i_2}\}_{i_2\in A_1}, \ldots,  \{x_{i_{\mathfrak  p}}\}_{i_{\mathfrak p}\in A_{\mathfrak p}}$, 
  such that $x_i \in \mathfrak S_{\star, r}(x_0)$, for any $i\in  A\equiv A_1\cup A_2 \ldots\cup A_{\mathfrak p}$, 
  \begin{equation}
  \label{labbe}
  \mathfrak S_{\star, r} \subset \mathfrak V_{\updelta, r}\equiv  \underset {\ell=1} {\overset {\mathfrak p}\cup} \, 
  \left( { \underset{i_\ell  \in A_\ell} \cup} \D^2(x_{i_\ell}, \updelta)
  \right) = \underset{i \in A } \cup \D^2(x_{i}, \updelta), 
  \end{equation}
   and such that the balls in each collection $\{\D^2(x_i, \updelta)\}_{i \in A_\ell}$ are disjoint, that is, for any 
   $\ell=1, \ldots, \mathfrak p$, we have 
   \begin{equation}
   \label{zikovich}
   \D^2(x_i, \updelta) \cap \D^2(x_j, \updelta)=\emptyset {\rm \ for \ } i\not=j {\rm \ with \ }  i, j \in A_\ell.
   \end{equation}
   As a consequence of the above constructions, a point $x \in \mathfrak   V_{\updelta, r}$, where $\mathfrak   V_{\updelta, r}$ is defined in \eqref{labbe}, belongs to at most $\mathfrak p$ distinct disks of the collection 
   $\{\overline{\D^2(x_{i}, \updelta)}\}_{i\in A}$.
   We define the set $\mathfrak S_{\updelta, r}$ as the closure  of the set $\mathfrak V_{\updelta, r}$ that is 
 $$ 
 \mathfrak S_{\updelta, r}\equiv \overline{\mathfrak V_{\updelta, r}}= \underset {\ell=1} {\overset {\mathfrak p}\cup} \, 
  { \underset{i_\ell  \in A_\ell} \cup}\overline{\D^2(x_{i_\ell}, \updelta)},  
  $$
     Notice that, by construction,  the total number $\sharp (A)$ of distinct disks  is finite. Actually, we have the bound
   \begin{equation}
   \label{sharpa}
    \sharp (A) \leq  \frac{4\mathfrak p {r^2}}{\updelta^2}. 
   \end{equation} 
   Indeed, since the famille of balls  $\{\D^2(x_{i_\ell}, \updelta)\}_{i\in A_\ell}$ are disjoint disks of radius $\updelta$ which are included in a ball of radius $2r$, we have 
   $$
   \sharp (A_\ell) \leq  \frac{ {4r^2}}{\updelta^2}  {\rm \ for \ } \ell=1, \ldots, \mathfrak p,
   $$
    so that \eqref{sharpa} follows by summation.
    
   We next consider  the set 
   $$
   \mathfrak Q_{\updelta, r}= \mathfrak S_{\updelta, r} \, \cup \S^2(x_0, r)
   $$
 and    its  distinct connected components  $\{\mathfrak T^k_{\updelta, r}\}_{_{k \in \mathcal J_\updelta }} $.  In view of the structure of $\mathfrak T_{\updelta, r}$,  which is an union of $\sharp (A)$ disks with a circle,  the total  number of connected components $\sharp {\mathcal J_\updelta}$ is  finite and actually bounded by $\sharp(A)+1$, hence the number on the right hand side of inequality \eqref{sharpa} plus one. As a matter of  fact,  we claim 
  
\begin{equation}
\label{fauconnier}
{\it    The  \ set \   } \mathfrak Q_{\updelta, r}   {\it  \ is  \ simply  \ connected, \ so  \ that \  }  \sharp(\mathcal J_\updelta)=1.
\end{equation}
   
   \medskip
   \noindent
 {\it Proof of the claim  \eqref{fauconnier}. }
   We assume by contradiction that $\mathfrak Q_{\updelta, r}$ has at least two distinct connected components and denote  by 
   $\mathfrak Q^1_{\updelta, r}$  the connected component which contains the circle $\S^1(x_0, r)$. Let 
   $\mathfrak Q^2_{\updelta, r}$ be a connected component distinct  from $\mathfrak Q^1_{\updelta, r}$, and set
   $$
   \upbeta\equiv  \inf\left\{ {\rm dist} (\mathfrak Q^2_{\updelta, r}, \mathfrak Q^j_{\updelta, r}),
    j \in \mathcal J_\updelta, j\not =2\right\}>0.
   $$
     We consider the open set 
     $$ \mathcal U=\left\{ x\in \R^2, {\rm dist }\left(x, \mathfrak Q^2_{\updelta, r}\right)<\frac{\upbeta}{4}\right\}
      \subset  \D^2(x_0, r) \setminus { \underset{ j \in \mathcal J_\updelta \setminus \{2\}}  \cup}\mathfrak Q^j_{\updelta, r}, 
      $$
      so that using the notation \eqref{Udelta}, we have
     $$
     \mathcal U_{\frac{\upbeta}{4}}=\left\{ x\in \R^2, {\rm dist }\left(x, \mathcal  U\right)<\frac{\upbeta}{4}\right\}  \subset 
      \D^2(x_0, r) \setminus { \underset{ j \in \mathcal J_\updelta \setminus \{2\}}  \cup}\mathfrak Q^j_{\updelta, r}
     $$
    and 
    \begin{equation}
    \label{udelta2}
    \mathcal V_{\frac{\upbeta}{4}}\equiv     \mathcal U_{\frac{\upbeta}{4}} \setminus \mathcal U \subset 
    \left\{ x \in \R^2, \frac{\upbeta}{4}\leq   {\rm dist }\left(x, \mathfrak Q^2_{\updelta, r}\right)\leq \frac{\upbeta}{2}
    \right\}
    \end{equation}
     and hence, combining \eqref{udelta2}  with the definition of $\upbeta$, we obtain
 \begin{equation}
 \label{udelta3}
 \mathcal V_{\frac{\upbeta}{4}} \cap \mathfrak S_{\star}=\emptyset  {\rm \ and \ }
 \upnu_\star\left(\mathcal V_{\frac{\upbeta}{4}}\right)=0.
 \end{equation}
     We are therefore in position to apply Theorem \ref{bordurer}  to assert that
   $ \displaystyle{\upnu_\star (\mathcal U)=0.}$
     However, since  by definition $\mathfrak Q^2_{\updelta, r}  \subset  \mathcal U $, it follows that 
     $\mathcal U \cap \mathfrak S_\star \not = \emptyset,$ so that $\upnu_\star(\mathcal U)>0$. We have hence reached a contradiction, which establishes the proposition.
       
 \begin{proof}[Proof of Proposition \ref{lesconti} completed]   It follows from the definition of $\mathfrak S_{\updelta, r}$ that
  $${\rm dist} (  \mathfrak Q_{\updelta, r} ,\mathfrak Q_{\star, r} ) \leq \updelta, $$
 so that  $\mathfrak Q_{\updelta, r}$ converges as $\updelta \to 0$ to $\mathfrak Q_{\star, r}$ in the Hausdorff metric.  Since for every $\updelta$, the set $\mathfrak S_{\updelta, r}$ is a continuum,    it then follows (see e.g. \cite{falconer}, Theorem 3.18)   that the Hausdorff limit $\mathfrak Q_{\star, r}$ is also a continuum and the proof is complete. 
 \end{proof}     
      
 We deduce as a consequence of Proposition \ref{lesconti}:
 
 \begin{corollary}  
 \label{lesconti2}
 The set $\mathfrak Q_{\star, r}$ is arcwise connected. 
 \end{corollary}    
   
  \begin{proof}  Indeed, any continuum with finite one-dimensional Hausdorff dimension is arc wise connected, see e.g \cite{falconer}, Lemma 3.12, p 34. 
  \end{proof}

   \begin{remark} {\rm  In the present context arcwise connected is equivalent to path-wise connected.
  }
  \end{remark} 
  
    \subsubsection{Proof of Proposition \ref{connective}}
      Invoking Fubini's theorem together with a mean value argument, we may choose some radius 
    $r_0 \in [r, 2r)$ such that  the number of points in $\mathfrak S_\star \cap \partial \D^2(x_0, r_0)$ is finite, more precisely
    $$
 m_0\equiv    \sharp \left( \mathfrak S_\star \cap \partial \D^2(x_0, r_0)\right)  \leq  \frac{ {\rm C}_{\rm H}}{r} M_0, 
    $$
   where we have used estimate \eqref{herbert} of the $\mathcal H^1$ measure of $\mathfrak S_\star$.  We may hence write
   \begin{equation} 
   \label{vacherol}
   \mathfrak S_\star \cap \partial \D^2(x_0, r_0)=\{a_1, \ldots, a_{m_0}\}.
   \end{equation}
    Next, we claim  that for any point $y \in \mathfrak S_{\star, r_0}$, there exists a  continuous path $p: [0, 1] \mapsto   \mathfrak S_{\star, r_0}$ connecting  the point $y$  to one of the points $a_1, \ldots, a_{m_0}$, that is  such that 
   \begin{equation}
   \label{lebof}
   p(0)=y {\rm \ and \ }  p(1)\in \{a_1, \ldots, a_{m_0}\}.
   \end{equation}
   \noindent
  {\it Proof of the claim \eqref{lebof}}.
If $\vert y-x_0 \vert =r_0$, then $y \in \mathfrak S_\star \cap \partial \D^2(x_0, r_0)$, and it therefore suffices to choose $p(s)=y$, for all $s \in [0, 1]$.    Otherwise, since, in view of Corollary  \ref{lesconti2}  applied at $x_0$ with radius $r_0$, the set  $\mathfrak S_{\star, r_0}\cup  \partial \D^2(x_0, r_0) $ is path-connected, there exists a continuous path $ \tilde p: [0, 1]\to \mathfrak S_{\star, r_0}\cup  \partial \D^2(x_0, r_0)$ such that 
   $$\tilde p(0)=y  {\rm \ and \ }  \tilde  p(1) \in  \partial \D^2(x_0, r_0).$$
 By continuity, there exists some number $s_0 \in [0, 1]$  such that
 $$ \vert   \tilde p(s)  \vert <r_0, {\rm \ for \ } 0\leq s<s_0  {\rm \ and \ } \vert   \tilde p(s_0)  \vert=r_0.
 $$
 It follows that 
 $$\tilde p(s_0)  \in \mathfrak S_\star \cap \partial \D^2(x_0, r_0)=\{a_1, \ldots, a_{m_0}\}.$$
 We then set 
 $$ p(s)=\tilde p(s), {\rm \ for \ } 0\leq s<s_0,   {\rm \ and \ } p(s)=\tilde p(s_0),  {\rm \ for \ }  s_0 \leq s\leq 1,$$
 and verify that $p$ has the desired property, so that the proof of the claim is complete.

 \medskip
 \noindent
  {\it Proof of  Proposition \ref{connective} completed}.  It follows from the claim \eqref{lebof} that any point  $y \in \mathfrak S_{\star, r_0}$ is connected to one of the points $a_1, \ldots, a_{m_0}$ given in \eqref{vacherol}. Hence $\mathfrak S_{\star, r_0}$ has at most $m_0$ connected components and the proof is complete. 
\qed 

       
       \subsection{Rectifiability of $\mathfrak S_\star$}
       In this section,  we prove:
       
       \begin{theorem}
       \label{rectifiable}
        The set $\mathfrak S_\star$ is rectifiable. 
       \end{theorem}

       \begin{proof}  The result is actually an immediate consequence of Proposition \ref{lesconti} and the fact that any 1-dimensional continuum is rectifiable, a result due to  Wazewski and independently Besicovitch (see e.g \cite{falconer}, Theorem 3.12). Indeed, given any $x_0 \in \Omega$,  $r>0$  such that $\D^2(x_0,r)\subset \Omega$, the set $\mathfrak S_{\star, r} \cup  \S^2(x_0, r)$ is a continuum, hence  rectifiable in view of the result quoted above, and hence so is the set $\mathfrak S_{\star, \frac{r}{2}}$. Since rectifiability is a local property,  the conclusion follows.
      \end{proof}
      
      \section{Proof of Theorem \ref{maintheo} completed}
      \label{proofmain}
      All statements in Theorem \ref{maintheo} have been obtained so far. Indeed, assertions i) follows combining several result in Section \ref{sectionsix}, namely Proposition \ref{proptheo1}, Proposition \ref{danes}, Proposition \ref{lesconti}, Proposition \ref{localconnect}  and Theorem  \ref{rectifiable}.

      \section{Additional properties of $\mathfrak S_\star$ and $\mu_\star$}
      \subsection{ On the tangent line at regular points  of $\mathfrak S_\star$}
      In this subsection, we provide the proof to Proposition \ref{tangentfort}. It relies on the following Lemma, which is actually a weaker statement:
      
      \begin{lemma}
      \label{ratus}
       Let $x_0$ be a regular point of $\mathfrak S_\star$. Given any $\uptheta>0$ there exists  a radius $R_{\rm cone}(\uptheta, x_0)$ such that 
 \begin{equation}
 \label{radius}
\mathfrak S_\star \cap \left(  \D^2\left(x_0, \uptau \right) \setminus \D^2\left(x_0, \frac{\uptau}{2}\right)\right)
 \subset   
 \left(  
  \mathcal C_{\rm one}\left(x_0, \vec e_{x_0}, \uptheta  \right)
   \right)
  {\rm  \  for \   any \  }  0<\uptau \leq  R_{\rm cone}(\uptheta, x_0).
 \end{equation}
      \end{lemma}      
      
\begin{proof}  Since we have the inclusion 
 $$\mathcal C_{\rm one}\left(x_0, \vec e_{x_0}, \uptheta  \right)  \subset  \mathcal C_{\rm one}\left(x_0, \vec e_{x_0}, \uptheta'  \right)$$
for  $<0\leq \uptheta \leq \uptheta'$, it suffices to establish the statement for $\uptheta$ arbitrary small. 
For  a given regular point $x_0$ of $\mathfrak S_\star$, we may invoke the convergence \eqref{tangent}  to assert that  there exists some 
$r_1> 0$ such that for $0<\uptau \leq r_1$  we have 
 \begin{equation}
 \label{tangenti}
 \mathcal H^1 \left(\mathfrak S_\star \cap  \D^2\left(x_0, 2\uptau \right) 
   \setminus   \mathcal C_{\rm one}\left(x_0, \vec e_{x_0}, \frac{\uptheta}{2}  \right)\right)\leq   
 \frac{\theta \uptau}{8}.
 \end{equation}
 Set
 $$A(x_0, \uptau, \uptheta)=\left(  \mathfrak S_\star \cap   \D^2\left(x_0, \uptau \right)
 \right)
 \setminus  \left(   \mathcal C_{\rm one}\left(x_0, \vec e_{x_0}, \uptheta  \right)
 \cup 
 \D^2\left(x_0, \frac{\uptau}{2}\right) 
 \right).
 $$
We have
\begin{equation}
\label{quesaisje}
A(x_0, \uptau, \uptheta) \cap  \mathcal C_{\rm one}\left(x_0, \vec e_{x_0}, \frac{\uptheta}{2}  \right)=\emptyset {\rm \ and \ hence \ }
\mathcal H^1 \left( A(x_0, \uptau, \uptheta) \right)\leq   
 \frac{\theta \uptau}{8}, 
\end{equation}
In view of \eqref{tangenti}.
We notice that 
\begin{equation*}
\left\{
\begin{aligned}
{\rm dist } \left( A(x_0, \uptau, \uptheta),  \mathcal C_{\rm one}\left(x_0, \vec e_{x_0}, \frac{\uptheta}{2}  \right)
\right)  &\geq   \frac{\uptau} {2} \sin \left( \arctan \frac{\uptheta}{2}\right) \\
{\rm dist } \left( A(x_0, \uptau, \theta),\partial \D^2(x_0, 2\uptau) \right)
&\geq \uptau.
\end{aligned}
\right.
\end{equation*}
So that, if $\uptheta>0$ is sufficiently small
\begin{equation}
\label{distus}
{\rm dist } \left( A(x_0, \uptau, \uptheta), \mathcal C_{\rm one}\left(x_0, \vec e_{x_0}, \frac{\uptheta}{2}  \right) 
\cup \partial \D^2(x_0, 2\uptau) 
\right)  \geq  \frac{\uptau} {2} \sin \left( \arctan \frac{\uptheta}{2}\right).
\end{equation}
Next we assume by contradiction that the set $A(x_0, \uptau, \uptheta)$ is not empty, so that exists a point $x_1 \in A(x_0, \uptau, \uptheta)$.  Since the set $\mathfrak Q_{\star, 2\uptau}(x_0)\equiv\mathfrak S_\star \cap  \D^2\left(x_0, 2\uptau \right)$ is path-connected, there exists a continuous path $p$ joining $x_1$ to some  point $x_2\in \partial \D^2(x_0, 2\uptau)$  which stays inside $\mathfrak S_{\star, 2\uptau}(x_0)$. On the other hand, since $x_1\in \D^2(x_0, \uptau)$ the length $\mathcal H^1(p)$ of this path is larger than $\uptau$.  We claim that 
\begin{equation}
\label{grouinox}
p \, \cap  \mathcal C_{\rm one}\left(x_0, \vec e_{x_0}, \frac{\uptheta}{2}  \right) \not =\emptyset. 
\end{equation}
Otherwise, indeed, $p$ would be a path inside $\mathfrak S_\star \cap  \D^2\left(x_0, 2\uptau \right) 
   \setminus   \mathcal C_{\rm one}\left(x_0, \vec e_{x_0}, \frac{\uptheta}{2}  \right)$. Since its length is larger then $\uptau$, this would contradict \eqref{tangenti}.  Next, combining  \eqref{grouinox}  and \eqref{distus}, we obtain 
   $$\mathcal H^1\left (p\cap \mathcal C_{\rm one}\left(x_0, \vec e_{x_0}, \frac{\uptheta}{2}  \right)  \right) \geq  \frac{\uptau} {2} \sin \left( \arctan \frac{\uptheta}{2}\right)\underset{\uptheta \to 0} \sim\frac{ \uptau \uptheta} {4}. 
   $$
   Since $p$ is a path inside $\mathfrak S_{\star, 2\uptau}(x_0)$ this contradicts \eqref{tangenti}, provided $\uptheta$ is chosen sufficiently small. This completes the proof of the Lemma, choosing $R_{\rm cone}(\uptheta, x_0)=r_1$. 
   \end{proof}      
      
      \begin{proof} [Proof of Proposition \ref{tangentfort} completed] Given $\uptau <R_1$,  we apply Lemma \ref{ratus},  the sequence of radii 
      $(\uptau_k)_{k \in \N}$ given by 
      $$\uptau_k=\frac{\uptau}{2^k}  {\rm \ for \  } k \in \N,  $$
      so that 
      \begin{equation*}
 \mathfrak S_\star \cap 
\left(  \D^2\left(x_0, \uptau_k \right) \cap \D^2\left(x_0, \uptau_{k+1}\right) \right)
 \subset      \mathcal C_{\rm one}\left(x_0, \vec e_{x_0}, \uptheta  \right), 
 {\rm  \  for \   any \  }  k \in \N.
 \end{equation*}
We take the union of these sets  on theft hand side, we obtain
\begin{equation*}
\mathfrak S_\star \setminus \{x_0\} =\underset {k \in \N}  \cup\mathfrak S_\star \cap 
\left(  \D^2\left(x_0, \uptau_k \right) \cap \D^2\left(x_0, \uptau_{k+1}\right) \right)
\subset \mathcal C_{\rm one}\left(x_0, \vec e_{x_0}, \uptheta  \right).
\end{equation*} 
This yields the result. 
      \end{proof}
      
   \subsection{The   limiting Hopf differential $\omega_\star$}
    The Hopf differential  
   $$\omega_\eps\equiv \eps \left( \vert (u_\eps)_{x_1} \vert^2-\vert (u_\eps)_{x_2} \vert^2 -2i (u_\eps)_{x_1}\cdot (u_\eps)_{x_2}\right) $$
   defined in \eqref{hopfique} has turned out to be  a central tool in our analysis so far.  We show in this subsection how it may yield some additional properties.  Since  the bound \eqref{naturalbound} holds true throughout our discussion, the measure 
   $\omega_{\eps_n}{\rm d}x_1 {\rm d}x_1$   
 is uniformly bounded, so that   we may assume, passing possibly to  further subsequence,  that 
   \begin{equation}
   \label{gloubi}
   \omega_{\eps_n} \rightharpoonup  \omega_\star,  {\rm \ in \ the \ sense \ of \  measures \ on  \  } \Omega,  {\rm \ as \ } n \to + \infty. 
   \end{equation}
   and similarly
   \begin{equation}
   \label{boulga}
 \frac{  V(u_{\eps_n})}{\eps_n}\rightharpoonup  \upzeta_\star,  {\rm \ in \ the \ sense \ of \  measures \ on  \  }, \Omega  {\rm \ as \ } n \to + \infty.
   \end{equation}
We present in this subsection some additional properties of the Hopf differential, which might be of interest for further studies. .

   \subsubsection{The limiting differential relation for $\omega_\star$}    
   Passing to the limit in \eqref{canardwc2}, we are  led to:

\begin{lemma}  
\label{scratch}
Let $(u_{\eps_n})_{n \in \N}$ be a sequence of solutions to \eqref{elipes} on $\Omega$ with $\eps_n \to 0$ as $n \to +\infty$ and assume that \eqref{naturalbound} holds. Let $\omega_\star$ and $\upzeta_\star$ be the bounded measures on $\Omega$ given by \eqref{gloubi} and \eqref{boulga} respectively. Then, we have 
  \begin{equation}
 \label{canardwcstar}
 \int_{\Omega}\mathrm{Re}\left(  \omega_\star \frac{\partial X}{\partial
\bar{z} }\right)=4 \int_{\Omega}  {\upzeta_\star}\, \mathrm {Re} \left(\frac{\partial X}{\partial
{z} }\right),  {\rm \ for  \ any \ } X\in C_0^\infty (\Omega, \C).
\end{equation}
 \end{lemma}  
 
 \begin{remark}{\rm The definition of the Hopf differential clearly depends on the choice of coordinates. Let $(\vec {\bf  e'}_1, \beprime_2)$  be a new orthonormal basis such that 
 \begin{equation*}
 \label{}
 \left\{
 \begin{aligned}
 \beprime_1&=\cos \theta \, \be_1+\sin\theta \, \be_2\\
 \beprime_2&=-\sin \theta  \, \be_1+\cos \theta \, \be_2, 
 \end{aligned}
 \right.
 \end{equation*}
let $(x'_1, x'_2)=(\cos \theta \, x_1-\sin \theta \,  x_2, \sin \theta\, x_1+ \cos \theta\, x_2)$ denote the  coordinates related  to the new basis and $\omega'_\eps$ the corresponding Hopf differential. Then, we have, for any map $u: \Omega \to \R^2$
\begin{equation*}
\left\{
\begin{aligned}
 \vert u_{x'_1} \vert^2-\vert u_{x'_2} \vert^2&=\cos 2 \theta \left ( \vert u_{x_1} \vert^2-\vert u_{x_2} \vert^2 \right)  -2\sin 2 \theta \,  u_{x_1}\cdot u_{x_2} \\
2 u_{x'_1} \cdot u_{x'_2}&= \sin 2 \theta \left ( \vert u_{x_1} \vert^2-\vert u_{x_2} \vert^2 \right) +2 \cos  2 \theta u_{x_1}\cdot u_{x_2},
\end{aligned}
\right.
\end{equation*}
and therefore
\begin{equation}
\omega' (u)=(\cos 2 \theta -i \sin 2  \theta) \omega (u)=\exp (-2 i \theta) \omega(u).
\end{equation}
 It follows in particular from the above relations that, if the limits \eqref{gloubi} and \eqref{boulga} exist for a given orthonormal  basis, then they exist also for any other one. 
 }
 \end{remark}
 
  We describe next some  additional properties of the measures$\omega_\star$ et $\zeta_\star$, mostly bases on Lemma \eqref{scratch}, choosing various kinds of test vector fields $\vec X$.
 Whereas we have used so far vector fields yielding dilatations of  the domain, we consider  also vector fields of different nature. Given a point 
   $x_0=(x_{0, 1}, x_{0, 2})\in  \Omega$, $r>0$ such that $\D^2(x_0, 2R) \subset \Omega$,   the fields we  will consider in the next paragraphs are  are of the form
 \begin{equation}
 \label{sheraton}
 \vec X_{f}(x_1, x_2)=  f_1(x_1)f_2(x_2) \be_2=i f_1(x_1)f_2(x_2),
 \end{equation}
 where,  $f_i$ represents, for $i=1, 2$ an arbitrary function in $C_c^\infty\left (\left(x_{0, i}-r, x_{0, i}+r\right)\right)$.   Thse vector fields have hence support on  the square $Q_r(x_0)$,   defined by
  \begin{equation}
  \label{cubuc}
  Q_r(x_0)   = \mathcal I_r(x_{0, 1}) \times  \mathcal I_r(x_{0, 2}),{\rm \ where \   }  \mathcal I_r(s)=[s-r, s+r]=\B^1(s, r), {\rm \ for \ } s>0, 
 \end{equation}
A short computation shows that
   \begin{equation}
   \left\{
   \label{lagaffe}
   \begin{aligned}
    \frac{\partial X_f}{\partial z}&=\frac{1}{2}f(x_1)f'_2 (x_2) +\frac i2 f'(x_1)f_2 (x_2),\\
     \frac{\partial X_f}{\partial \bar{z}}&=-\frac{1}{2}f(x_1)f'_2(x_2) +\frac i2 f'(x_1)f_2(x_2), \\
   \end{aligned}
   \right.
   \end{equation}
   and hence
  \begin{equation}
  \label{lagaffe2}
  \left\{
  \begin{aligned}
    {\upzeta_\star}\, \mathrm {Re} \left(\frac{\partial X}{\partial
{z} }\right)  &=   \frac{1}{2}f(x_1)f'_2(x_2) {\upzeta_\star}\, {\rm \ and \ } \\
\mathrm{Re}\left(  \omega_{\star} \frac{\partial X}{\partial
\bar{z} }\right)&=-\frac{\mathrm{Re}(\omega_\star)}{2}f(x_1)f'_2 (x_2)-
\frac{{\rm Im} (\omega_\star)}{2}f'(x_1)f (x_2).
\end{aligned}
  \right.
   \end{equation}

  \subsubsection{Shear vector fields}
 We choose, in this subsection  as functions $f_1,  f_2$ in  \eqref{sheraton}   $f_1=f$, where $f$ is an arbitrary function in $C_\infty(\mathcal I_r(x_0)$ and,  for $f_2$, a function of the form 
$$\displaystyle{f_2(x_2)=\varphi( \frac{ x_2-x_{0, 2}}{r})}, $$ where $\varphi$ is a
  non-negative given  smooth plateau function such that 
 \begin{equation}
 \label{varphiness}
 \varphi(s)=1, {\rm \ for \ } s \in [-\frac  34, \frac 34],  {\rm \ and \ } \varphi (s)=0,  {\rm \ for \ } \vert s \vert \geq 1.
\end{equation}
Such a vector field  corresponds to  \emph{shear vector field}.   We consider the subset $\mathcal R_r(x_0)$ of $Q_r(x_0)$ given by 
  $$\mathcal R_r(x_0)\equiv \mathcal I_r(x_{0, 1}) \times \mathcal I_{\frac {3r}{4}}(x_{0, 2})\subset  Q_r(x_0),$$
  so that $Q_r(x_0) \setminus \mathcal R_r(x_0)$ is the union of to disjoint rectangles
  $$
   Q_r(x_0) \setminus \mathcal R_r(x_0)= \left(\mathcal I_r(x_{0, 1}) \times (x_{0, 2}+\frac{3r}{4}, x_0+r) \right) \cup 
    \left(\mathcal I_r(x_{0, 1}) \times( x_{0, 2}-r, 
     x_{0,2}-\frac{3r}{4}) \right).
  $$
Using shear vector fields,  as test vector fields in \eqref{canardwc},  we obtain:

\begin{proposition}
\label{scratchiness}
Assume that $\upnu_\star (\overline{Q_r(x_0) \setminus \mathcal R_r(x_0)}) =0$. Then, the functions $J$ defined on $\overset{\circ}  { \mathcal I_r} (x_0)$ by 
$$
J(s)=\int_{\{s\} \times \mathcal I_{\frac{3r}{4}}(x_{0, 2})} {\rm Im}(\omega_\star) \rd x_2,  { \rm \ for \   } s \in  (x_{0, 1}-r, x_{0, 1}+r), 
$$
is constant.
\end{proposition}
\begin{proof}
We first show  that, for   any function $f \in C_c^\infty(\mathcal I_r(x_{0, 1}))$, we have 
\begin{equation}
\label{scratchitude}
\int_{\mathcal R_r(x_0) } f'(x_1) {\rm Im } (\omega_\star)\rd x_1 \rd  x_2=0.
\end{equation}
Indeed, identity \eqref{scratchitude} follows  combining \eqref{canardwcstar} and  \eqref{lagaffe2} and the fact that $\upnu_\star (\overline{Q_r(x_0) \setminus \mathcal R_r(x_0)}) =0$ .   Applying Fubini's theorem, we notice that 
$$
\int_{\mathcal R_r(x_0) } f'(x_1) {\rm Im } (\omega_\star)\rd x_1 \rd  x_2=\int_{I_r(x_{0, 1})} f'(x_1)J(x_1) \rd x_1, 
$$
which, combined with \eqref{scratchitude},  implies that $J_1$ is constant. 
\end{proof}
\subsubsection{Stretching vector fields}
In this subsection, we assume that
   $f_1=f$, where $f$ is an arbitrary function in $C_\infty(\mathcal I_r(x_0))$ as above, and,  that $f_2$ is given by  
   $$\displaystyle{f_2(x_2)=x_2\varphi( \frac{ x_2-x_{0, 2}}{r})}, $$ 
   where $\varphi$ is a
  non-negative given  smooth plateau function such that \eqref{varphiness} holds. Combining as above  \eqref{canardwcstar} and  \eqref{lagaffe2} we obtain:

 \begin{lemma} Assume that  $\upnu_\star (\overline{Q_r(x_0) \setminus \mathcal R_r(x_0)}) =0$.  We have, for any function in $C_\infty(\mathcal I_r(x_0))$
$$
\int_{\mathcal R_r(x_0) }f(x_1)(\left(  \mathrm{Re} (\omega_\star)-\upzeta_\star\right)  +(f'(x_1) {\rm Im } (\omega_\star)\left[x_2\varphi( \frac{ x_2-x_{0, 2}}{r})\right] \rd x_1 \rd  x_2=0.
$$
\end{lemma}

\subsubsection {Dilation vector fields}
We use here   as test vector fields in \eqref{canardwc}, vector fields of  the form
$$
X_d(x_1, x_2)=\varphi( \frac{ x_2-x_{0, 2}}{r}) f(x_1)\be_1.
$$
  Computations similar to the proof of Proposition \ref{scratchiness} then yield:
  
  \begin{proposition}
\label{scratchinou}
Assume that $\upnu_\star (\overline{Q_r(x_0) \setminus \mathcal R_r(x_0)}) =0$. Then, the functions $L$  	   defined on $\overset{\circ}  { \mathcal I_r} (x_0)$ by 
$$
L(s)=\int_{\{s\} \times \mathcal I_{\frac{3r}{4}}(x_{0, 2})}\left(  \mathrm{Re} (\omega_\star)-\upzeta_\star\right) \rd x_2,  { \rm \ for \   } s \in  (x_{0, 1}-r, x_{0, 1}+r), 
$$
is constant.
\end{proposition}

    

\end{document}